\documentclass[oneside]{amsart}
\usepackage{amsmath, amsfonts, amsthm, amssymb}
\usepackage{setspace}
\usepackage[all]{xy}
\usepackage{color}
\usepackage{tikz}
\usetikzlibrary{arrows,decorations.pathmorphing,backgrounds,positioning,fit,petri}
\usepackage[pdftex,colorlinks,citecolor=blue]{hyperref}
\usepackage{mathrsfs}
\usepackage{soul}%for \ul command, which allows for linebreaks while underlineing
\usepackage{cite}
\usepackage{rotating}

\definecolor{plurp}{RGB}{255,0,255}
\definecolor{job}{RGB}{200,65,0}
\definecolor{kiyoshi}{RGB}{0,0,200}

\newtheorem{introthm}{Theorem}

\newtheorem{theorem}{Theorem}[section] %label prefix thm:
\newtheorem{proposition}[theorem]{Proposition} %label prefix prop:
\newtheorem{lemma}[theorem]{Lemma} %label prefix lem:
\newtheorem{corollary}[theorem]{Corollary} %label prefix cor:

\theoremstyle{definition}
\newtheorem{definition}[theorem]{Definition} %label prefix def:
\newtheorem{remark}[theorem]{Remark} %label prefix rmk:
\newtheorem{notation}[theorem]{Notation} %label prefix note:
\DeclareMathOperator{\Hom}{Hom}

\DeclareMathOperator{\modd}{mod}

\DeclareMathOperator{\Ext}{Ext}

\DeclareMathOperator{\dd}{d\!}

\newcommand{\var}{\mathsf{var}}

\newcommand{\textdef}[1]{\textbf{#1}}

\newcommand{\ZZ}{\mathbb{Z}}
\newcommand{\NN}{\mathbb{N}}
\newcommand{\RR}{\mathbb{R}}
\newcommand{\Rbar}{\overline{\RR}}

\newcommand{\veck}{\text{\normalfont vec}(\Bbbk)}

\newcommand{\Ind}{\text{\normalfont Ind}}

\newcommand{\Npi}{\mathbf{N}_\pi}
\newcommand{\hyper}{\mathfrak{h}^2}
\newcommand{\Meas}{\mathcal{M}}
\newcommand{\geo}{\text{-}}
\newcommand{\minn}[1]{{#1}_{\min}}
\newcommand{\maxx}[1]{{#1}_{\max}}

\newcommand{\modr}{\modd^{\text{r}}}
\newcommand{\Indr}{\Ind^{\text{r}}}

\newcommand{\fpc}{\text{fpc}}

\title[Continuous Stability of Type $\mathbb{A}$ and Measured Laminations]{Continuous Stability Conditions of Type $\mathbb{A}$ and Measured Laminations of the Hyperbolic Plane}
\author{Kiyoshi Igusa}
\address{Department of Mathematics, Brandeis University, Waltham, Massachusetts, USA}
\email{igusa@brandeis.edu}
\author{Job Daisie Rock}
\address{Department of Mathematics W16, Ghent University, Ghent, East Flanders, Belgium}
\email{job.rock@ugent.be}
\date{28 February 2023}

\begin{document}

\maketitle
\begin{center}
\emph{Dedicated to Idun Reiten for her kind support and encouragement}
\end{center}

\tableofcontents

\begin{abstract} We introduce stability conditions (in the sense of King) for representable modules of continuous quivers of type $\mathbb{A}$ 
%and a special criteria for such stability conditions called the four point condition.
along with a special criteria called the four point condition.
The stability conditions are defined using a generalization of $\delta$ functions, called half-$\delta$ functions.
We show that for a continuous quiver of type $\mathbb{A}$ with finitely many sinks and sources, the stability conditions satisfying the four point condition are in bijection with measured laminations of the hyperbolic plane.
Along the way, we extend an earlier result by the first author and Todorov regarding continuous cluster categories for linear continuous quivers of type $\mathbb{A}$ and laminations of the hyperbolic plane to all continuous quivers of type $\mathbb{A}$ with finitely many sinks and sources. We also give a formula for the continuous cluster character.
 \end{abstract}

\newcommand{\undim}{\underline\dim}

\section*{Introduction}

\subsection*{History}
The type of stability conditions in the present paper were introduced by King in order to study the moduli space of finitely generated representations of a finite-dimensional algebra \cite{K94}.

There is recent work connecting stability conditions to wall and chamber structures for finite-dimensional algebras \cite{BST19} and real Grothendieck groups \cite{A21}.
There is also work studying the linearity of stability conditions for finite-dimensional algebras \cite{I20}.

In 2015, the first author and Todorov introduced the continuous cluster category for type $\mathbb{A}$ \cite{IT15}.
More recently, both authors and Todorov introduced continuous quivers of type $\mathbb{A}$ and a corresponding weak cluster category \cite{IRT22,IRT22b}.
The second author also generalized the Auslander--Reiten quiver of type $\mathbb{A}_n$ to the Auslander--Reiten space for continuous type $\mathbb{A}$ and a geometric model to study these weak cluster categories \cite{R19+,R22}.

\subsection*{Contributions and Organization}
In the present paper, we generalize stability conditions, in the sense of King, to continuous quivers of type $\mathbb{A}$.
In Section~\ref{sec:finite case} we recall facts about stability conditions and reformulate them for our purposes.
In Section~\ref{sec:continuous stability conditions} we recall continuous quivers of type $\mathbb{A}$, representable modules, and then introduce our continuous stability conditions.

At the beginning of Section~\ref{sec:half delta and red-blue pairs} we define a half-$\delta$ function, which can be thought of as a Dirac $\delta$ function that only exists on the ``minus side'' or ``plus side'' of a point.
We use the half-$\delta$ functions to define useful functions (Definition~\ref{def:useful}), which are equivalent to functions with bounded variation but better suited to our purposes.
Then we define a stability condition as an equivalence class of pairs of useful functions with particular properties, modulo shifting the pair of functions up and down by a constant (Definitions~\ref{def:red-blue function pair}~and~\ref{def:stability condition}).

We use some auxiliary constructions to define a semistable module (Definition~\ref{def:stable module}).
Then we recall $\Npi$-compatibility (Definition~\ref{def:Npi compatible}), which can be thought of as the continuous version of rigidity in the present paper.
We link stability conditions to maximally $\Npi$-compatible sets using a criteria called the four point condition (Definition~\ref{def:four point condition}).
By $\mathcal{S}_\fpc(Q)$ we denote the set of stability conditions of a continuous quiver $Q$ of type $\mathbb{A}$ that satisfy the four point condition.

\begin{introthm}[Theorem~\ref{thm:four point equivalent to cluster}]\label{thm:A}
    Let $Q$ be a continuous quiver of type $\mathbb{A}$ with finitely many sinks and sources and let $\sigma\in\mathcal{S}(Q)$.
    Then the following are equivalent.
    \begin{itemize}
        \item $\sigma\in\mathcal{S}_\fpc(Q)$.
        \item The set of $\sigma$-semistable indecomposables is maximally $\Npi$-compatible.
    \end{itemize}
\end{introthm}

In Section~\ref{sec:continuous tilting section} we define a continuous version of tilting.
That is, for a continuous quiver $Q$ of type $\mathbb{A}$ we define a new continuous quiver $Q'$ of type $\mathbb{A}$ together with an induced map on the set of indecomposable representable modules.
This is not to be confused with reflection functors for continuous quivers of type $\mathbb{A}$, introduced by Liu and Zhao \cite{LZ22}.
For each stability condition $\sigma$ of $Q$ that satisfies the four point condition, we define a new stability condition $\sigma'$ of $Q'$ (Definition~\ref{def:flipped functions}).
We show that continuous tilting induces a bijection on indecomposable representable modules, preserves $\Npi$-compatibility, and incudes a bijection on stability conditions for $Q$ and $Q'$ that satisfy the four point condiiton.
Denote by $\modr(Q)$ the category of representable modules over $Q$.

\begin{introthm}[Theorems~\ref{thm: continuous tilting gives bijection on max compatible sets}~and~\ref{thm:stability tilting}]\label{thm:B}
Let $Q$ and $Q'$ be continuous quivers of type $\mathbb{A}$ with finitely many sinks and sources.
Continuous tilting yields a triple of bijections: $\phi$, $\Phi$, and $\Psi$.
\begin{itemize}
    \item A bijection $\phi:\Ind(\modr(Q))\to\Ind(\modr(Q'))$.
    \item A bijection $\Phi$ from maximal $\Npi$-compatible sets of $\modr(Q)$ to maximal $\Npi$-compatible sets of $\modr(Q')$. Furthermore if $\mu:T\to T'$ is a mutation then so is $\Phi(\mu):\Phi T\to \Phi T'$ given by $\phi(M_I)\mapsto \phi(\mu(M_I))$.
    \item A bijection $\Psi:\mathcal{S}_\fpc(Q)\to\mathcal{S}_\fpc(Q')$ such that if $T$ is the set of $\sigma$-semistable modules then $\Phi(T)$ is the set of $\Psi(\sigma)$-semistable modules.
\end{itemize}
\end{introthm}

In Section~\ref{sec:measured laminations} we define a measured lamination to be a lamination of the (Poincar\'e disk model of the) hyperbolic plane together with a particular type of measure on the set of geodesics (Definition~\ref{def:measured lamination}).
We denote the Poincar\'e disk model of the hyperbolic plane by $\hyper$.
Then we recall the correspondence between laminations of $\hyper$ and maximally $\Npi$-compatible sets of indecomposable representable modules over the straight descending continuous quiver of type $\mathbb{A}$, from the first author and Todorov (Theorem~\ref{thm:igusa todorov} in the present paper) \cite{IT15}.
We extend this correspondance to stability conditions that satisfy the four point condition and measured laminations (Theorem~\ref{thm:measured laminations and straight four point conditions}).
Combining this with Theorems~\ref{thm:A}~and~\ref{thm:B}, we have the last theorem.

\begin{introthm}[Corollary~\ref{cor:measured laminations and straight four point conditions for all orientations}]\label{thm:C}
    Let $Q$ be a continuous quiver of type $\mathbb{A}$ and $\mathcal{L}$ the set of measured laminations of $\hyper$.
    There are three bijections: $\phi$, $\Phi$, and $\Psi$.
    \begin{itemize}
        \item The bijection $\phi$ from $\Ind(\modr(Q))$ to geodesics in $\hyper$.
        \item The bijection $\Phi$ from maximally $\Npi$-compatible sets to (unmeasured) laminations of $\hyper$ such that, for each maximally $\Npi$-compatible set $T$, $\phi|_T$ is a bijection from the indecomposable modules in $T$ to the geodesics in $\Phi(T)$.
        \item The bijection $\Psi:\mathcal{S}_\fpc(Q)\to \mathcal{L}$ such that if $T$ is the set of $\sigma$-semistable modules then $\Phi(T)$ is the set of geodesics in $\Psi(\sigma)$.
    \end{itemize}
\end{introthm}

In Section \ref{sec:continuous cluster character}, we give a formula for a continuous cluster character $\chi(M_{ab})$.
This is a formal expression in formal variables $x_t$, one for every real number $t$.
We verify some cluster mutation formulas, but {leave further work for a future paper}.

In Section~\ref{sec:finite map using continuous tilting}, we relate continuous tilting to cluster categories of type $\mathbb{A}_n$.
In particular, we discuss how a particular equivalence between type $\mathbb{A}_n$ cluster categories is compatible with continuous tilting.
We conclude our contributions with an example for type $\mathbb{A}_4$ (Section~\ref{sec:finite map and continuous tilting example}).
Then we briefly describe some directions for further related research.

\subsection*{Acknowledgements}
The authors thank Gordana Todorov for helpful discussions.
KI was supported by  Simons Foundation Grant \#686616.
Part of this work was completed while JDR was at the Hausdorff Research Institute for Mathematics (HIM); JDR thanks HIM for their support and hospitality.
JDR is supported at Ghent University by BOF grant 01P12621.
JDR thanks Aran Tattar and Shijie Zhu for helpful conversations.

\section{The Finite Case}\label{sec:finite case}
    
    {There is a relation between stability conditions and generic decompositions which will become more apparent in the continuous case. Here we examine the finite case and impose continuous structures onto the discrete functions in order to give a preview of what will happen in the continuous quiver case.} 
    
    For a finite quiver of type of $\mathbb{A}_n$ with vertices $1,\cdots,n$, we need a piecewise continuous functions on the interval $[0,n+1]$ which has discontinuities at the vertices which are sources or sinks. The stability function will be the derivative of this function. It will have Dirac delta functions at the sources and sinks. Since this is a reformulation of well-known results, we will not give proofs. {We also review the Caldero-Chapoton cluster character for representations of a quiver of type $A_n$ \cite{CCC} in order to motivate the continuous case in Section \ref{sec:continuous cluster character}}.

\subsection{Semistability condition}

Recall that a stability function is a linear map
\[
	\theta: K_0\Lambda=\ZZ^n\to \RR.
\]
A module $M$ is $\theta$-semistable if $\theta(\undim M)=0$ and $\theta(\undim M')\le0$ for all submodules $M'\subset M$. We say $M$ is $\theta$-stable if, in addition, $\theta(\undim M')<0$ for all $0\neq M'\subsetneq M$. For $\Lambda$ of type $\mathbb{A}_n$, we denote by $M_{(a,b]}$ the indecomposable module with support on the vertices $a+1,\cdots,b$. For example $M_{(i-1,i]}$ is the simple module $S_i$. Let $F:\{0,1,\cdots,n\}\to \RR$ be the function 
\[
	F(k)=\sum_{0<i\le k}\theta(\undim S_i)=\theta(\undim M_{(0,k]})
\]
Then we have $\theta(M_{(a,b]})=F(b)-F(a)$. 

Thus, for $M_{(a,b]}$ to be $\theta$-semistable we need $F(a)=F(b)$ and another condition to make $\theta(\undim M')\le0$. For example, take the quiver of type $\mathbb{A}_n$ having a source at vertex $c$ and sinks at $1,n$. Then the indecomposable submodules of $M_{(a,b]}$ are $M_{(a,x]}$ for $a<x<c$, $x\le b$ and $M_{(y,b]}$ for $c\le y<b$, $a\le y$. Therefore, we also need $F(x)\le F(a)=F(b)\le F(y)$ for such $x,y$. (And strict inequalities $F(x)<F(a)=F(b)<F(y)$ to make $M_{(a,b]}$ stable.)

A simple characterization of $x,y$ is given by numbering the arrows. Let $\alpha_i$ be the arrow between vertices $i,i+1$. Then the arrows connecting vertices in $(a,b]$ are $\alpha_i$ for $a<i<b$. $M_{(a,x]}\subset M_{(a,b]}$ if $\alpha_x$ points left (and $a<x<b$). $M_{(y,b]}\subset M_{(a,b]}$ if $\alpha_y$ points to the right (and $a<y<b$). More generally, we have the following.

\begin{proposition}\label{prop: semi-stability using F}
$M_{(a,b]}$ is $\theta$-semistable if and only if the following hold.
\begin{enumerate}
\item $F(a)=F(b)$
\item $F(x)\le F(a)$ if $\alpha_x$ points left and $a<x<b$.
\item $F(y)\ge F(b)$ if $\alpha_y$ points right and $a<y<b$.
%\qed
\end{enumerate}
Furthermore, if the inequalities in (2),(3) are strict, $M_{(a,b]}$ is $\theta$-stable.\qed
\end{proposition}

For example, take the quiver
\begin{equation}\label{eq: example of A4}
	1\xleftarrow{\alpha_1} 2\xrightarrow{\alpha_2} 3\xrightarrow{\alpha_3} 4
\end{equation}
with $\theta=(-1,2,-1,-2)$, $F=(0,-1,1,0,-1)$. Then $F(1)<F(0)=F(3)=0<F(2)$, with $\alpha_1$ pointing left and $\alpha_2$ pointing right. So, $M_{(0,3]}$ is $\theta$-stable. Similarly, $F(1)=F(4)=-2<F(2),F(3)$ implies $M_{(1,4]}$ is also $\theta$-stable 

One way to visualize the stability condition is indicated in Figure \ref{Fig: graph of F, finite case}.

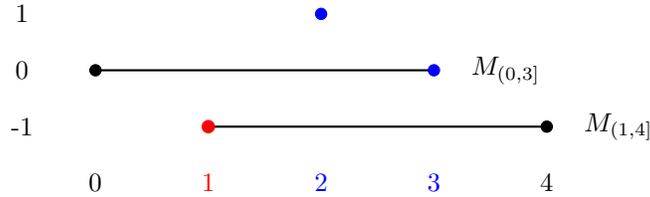
\begin{figure}[htbp]
\begin{center}
\begin{tikzpicture}[scale=.75]
%
%\clip (-1,-4.2)rectangle (11.2,.5);
%\foreach \x/\y in {0/0,2/1,4/2,6/3, 8/4}
%\draw (\x,-3)--(\x,-3) (\x,-3.9) node{\y};
\foreach \x/\y in {0/0,-1/-1,1/1}
\draw (-1.3,\x) node{\y};
\draw[thick] (0,0)--(6,0);
\draw[thick] (2,-1)--(8,-1);
\draw (0,-2) node{$0$};
\draw[red] (2,-2) node{$1$};
\draw[blue] (4,-2) node{$2$};
\draw[blue] (6,-2) node{$3$};
\draw (8,-2) node{$4$};
%\draw[red,<-] (1.2,-2.5)--(2.8,-2.5) ;
%\draw[blue,->] (3.2,-2.5)--(4.8,-2.5) ;
%\draw[blue,->] (5.2,-2.5)--(6.8,-2.5) ;
\draw (6.5,0)node[right]{$M_{(0,3]}$};
\draw (8.5,-1)node[right]{$M_{(1,4]}$};
%\draw[blue] (9.5,0.3)node{$F(2,5]=B(2,5]$};
%\draw[very thick,red] (0,0)--(2,-2)--(4,-3);
%\draw[very thick,blue] (4,0)--(6,-1)--(8,0)--(10,-3);
%\draw[fill,white] (4,0) circle[radius=1mm];
\draw[fill,blue] (4,1) circle[radius=1mm];
\draw[fill,blue] (6,0) circle[radius=1mm];
\draw[fill] (8,-1) circle[radius=1mm];
\draw[fill] (0,0) circle[radius=1mm];
%\draw[thick,blue] (4,0) circle[radius=1mm];
\draw[thick,red,fill] (2,-1) circle[radius=1mm];
%\draw[thick,red,fill] (2,-2) circle[radius=1mm];
%\draw[thick,fill,white] (0,0) circle[radius=1mm];
%\draw[thick,red] (0,0) circle[radius=1mm];
%
\end{tikzpicture}
\caption{The graph of $F:\{0,1,2,3,4\}\to \mathbb R$ shows the $\theta$-semistable modules. When $M_{(a,b]}$ is $\theta$-stable, $F(a)=F(b)$ making the line segment connecting $(a,F(a))$ to $(b,F(b))$ horizontal. Also, the intermediate red points are below and the blue points are above the line segment if we draw as red/blue, the spot $(x,F(x))$ for $\alpha_x$ pointing left/right, respectively.
}
\label{Fig: graph of F, finite case}
\end{center}
\end{figure}

\subsection{Generic decomposition} Stability conditions for quivers of type $\mathbb{A}_n$ also give the generic decomposition for dimension vectors ${\bf d}\in\NN^n$. This becomes more apparent for large $n$ and gives a preview of what happens in the continuous case.

Given a dimension vector ${\bf d}\in \NN^n$, there is, up to isomorphism, a unique $\Lambda$-module $M$ of dimension vector $\bf d$ which is rigid, i.e., $\Ext^1(M,M)=0$. The dimension vectors $\beta_i$ of the indecomposable summands of $M$ add up to $\bf d$ and the expression ${\bf d}=\sum \beta_i$ is called the ``generic decomposition'' of $\bf d$. We use the notation $\beta_{ab}=\undim M_{(a,b]}$ and ${\bf d}=(d_1,\cdots,d_n)$.

There is a well-known formula for the generic decomposition of a dimension vector \cite{abeasis} which we explain with an example. Take the quiver of type $\mathbb{A}_9$:
\begin{equation}\label{eq: A9 quiver}
	{\color{red}1\xleftarrow{\alpha_1} 2\xleftarrow{\alpha_2} 3\xleftarrow{\alpha_3} 4\xleftarrow{\alpha_4} 5\xleftarrow{\alpha_5}} 6 {\color{blue}\xrightarrow{\alpha_6} 7\xrightarrow{\alpha_7} 8\xrightarrow{\alpha_8} 9}
\end{equation}
with dimension vector ${\bf d}=(3,4,1,3,2,4,3,1,3)$. 
To obtain the generic decomposition for $\bf d$, we draw $d_i$ spots in vertical columns as shown in \eqref{eq: spots method} below.
\begin{equation}\label{eq: spots method}
\xymatrixrowsep{10pt}\xymatrixcolsep{15pt}
\xymatrix{%begin xy matrix
&1 & \ar[l] 2 &\ar[l] 3 &\ar[l] 4&\ar[l] 5&\ar[l] 6\ar[r] & 7\ar[r] &8\ar[r] &9\\
&\bullet\ar@{-}[r]& \bullet\ar@{-}[r] &\bullet\ar@{-}[r]&  \bullet\ar@{-}[r]&\bullet\ar@{-}[r]&\bullet \\ % first line
&\bullet\ar@{-}[r]& \bullet&   &   \bullet\ar@{-}[r] & \bullet\ar@{-}[r]  &  \bullet\ar@{-}[r] &\bullet&&\bullet\\ % second line
&\bullet\ar@{-}[r]& \bullet&   &   \bullet& &  \bullet\ar@{-}[r] &\bullet&&\bullet\\ % third line
&&\bullet&&&& \bullet\ar@{-}[r]&  \bullet \ar@{-}[r]&  \bullet \ar@{-}[r]&  \bullet  \\ % fourth line
\bf d: &3& 4 & 1&3&2&4&3&1&3
%XX111\ar[d]\ar@/^2pc/[r]_f &	XX222\ar[d]\ar[r] &	XX333\ar@{..>}[d]\\
%YY111 \ar[r]& 	YY222 \ar[r]^(.4)g&	YY3333333333
	}%end xy matrix
\end{equation}
For arrows going left, such as $3\leftarrow 4$, $5\leftarrow6$ the top spots should line up horizontally. For arrows going right, such as $6\to 7, 7\to 8$ the bottom spots should line up horizonally as shown. Consecutive spots in any row are connected by horizontal lines. For example, the spots in the first row are connected giving $M_{(0,6]}$ but the second row of spots is connected in three strings to give $M_{(0,2]}, M_{(3,7]}$ and $S_9=M_{(8,9]}$. The generic decomposition is given by these horizontal lines. Thus
\[
   {\bf d}=(3,4,1,3,2,4,3,1,3)=\beta_{06}+2\beta_{02}+ \beta_{37}+2\beta_{89}+\beta_{12}+\beta_{34}+\beta_{57}+\beta_{59}
\]
is the generic decomposition of ${\bf d}=(3,4,1,3,2,4,3,1,3)$.

We construct this decomposition using a stablity function based on \eqref{eq: spots method}. We explain this with two examples without proof. The purpose is to motivate continuous stability conditions.

Take real numbers $d_0,d_1,\cdots,d_n,d_{n+1}$ where $d_0=d_{n+1}=0$. Draw arrows where the arrow $\alpha_i$ connecting $i-1,i$ where $\alpha_0$ points in the same direction as $\alpha_1$ and $\alpha_n$ points in the same direction as $\alpha_{n-1}$. To each arrow $\alpha_i$ we associate the real number which is $d_i$ of the target minus $d_i$ of the source. We write this difference below the arrow if the arrow points left and above the arrow when the arrow points right. Then we compute the partial sums for the top numbers and the bottom numbers. Let $B,R$ denote these functions. Thus $B(6)=0,B(7)=-1,B(8)=-3,B(9)=-1,B(10)=-4$ and $R(0)=0, R(1)=-3$, etc. as shown below.

\begin{center}
\begin{tikzpicture}[scale=1.1]
%\draw[help lines=1,thin] (0,0) grid (10,3);
\foreach \x/\y in {0/0,1/3,2/4,3/1,4/3,5/2,6/4,7/3,8/1,9/3,10/0}
\draw (\x,3) node{\x} (\x,1.2) node{\y};
\foreach \x/\y in {1/3,2/4,3/1,4/3,5/2,6/4}
\draw[red!90!black] (\x,0) node{$-\y$};
\draw[red!90!black] (0,0) node{$0$} (-.7,0)node{$R:$};
\foreach \x/\y in {1/-3,2/-1,3/3,4/-2,5/1,6/-2}
\draw[red!90!black] (\x-.5,.6) node{$\y$};
\foreach \x in {0,1,...,5}
\draw[<-,red,thick] (\x+.2,1.2)--(\x+.8,1.2);
\foreach \x/\y in {6/0,7/-1,8/-3,9/-1,10/-4}
\draw[blue] (\x,2.4) node{$\y$};
\draw[blue] (-.7,2.4)node{$B:$};
\foreach \x/\y in {7/-1,8/-2,9/2,10/-3}
\draw[blue] (\x-.5,1.8) node{$\y$};
\foreach \x in {6,7,8,9}
\draw[blue,thick,->] (\x+.2,1.2)--(\x+.8,1.2);
%\foreach \y in {-6,-4,...,8}\draw (0,\y) node{\y};
%\draw[thick,color=blue] (0,1) ellipse [x radius=2.8cm,y radius=2.1cm];
\draw (-.7,3) node{$i:$} (-.7,1.2) node{$d_i$};
\end{tikzpicture}
\end{center}

We extend the blue and red functions by $B(x)=B(6)=0$ for all $x<6$ and $R(x)=R(6)=-4$ for all $x>6$.

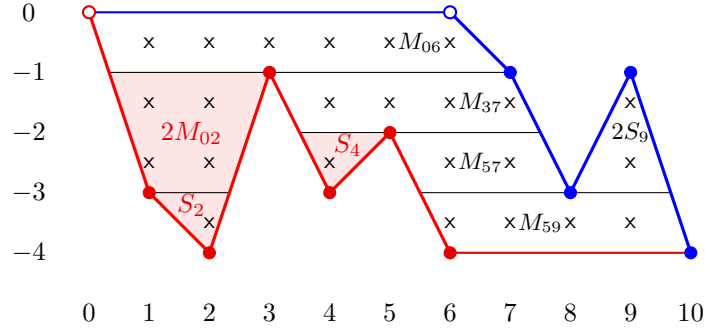
\begin{figure}[htbp]
\begin{center}
\begin{tikzpicture}[scale=.8]

\draw[fill,red!10!white] (1,-3)--(2,-4)--(3,-1)--(.33,-1);
\draw[fill,red!10!white] (3.5,-2)--(4,-3)--(5,-2);

\foreach \x in {-.5,-1.5,-2.5}
\draw (1,\x) node{\small\sf x};
\foreach \x in {-.5,-1.5,-2.5,-3.5}
\draw (2,\x) node{\small\sf x};
\foreach \x in {-.5}
\draw (3,\x) node{\small\sf x};
\foreach \x in {-.5,-1.5,-2.5}
\draw (4,\x) node{\small\sf x};
\foreach \x in {-.5,-1.5}
\draw (5,\x) node{\small\sf x};
\foreach \x in {-.5,-1.5,-2.5,-3.5}
\draw (6,\x) node{\small\sf x};
\foreach \x in {-1.5,-2.5,-3.5}
\draw (7,\x) node{\small\sf x};
\foreach \x in {-3.5}
\draw (8,\x) node{\small\sf x};
\foreach \x in {-1.5,-2.5,-3.5}
\draw (9,\x) node{\small\sf x};
\draw[red!90!black] (1.7,-2) node{$2M_{02}$};
\draw[red!90!black] (1.7,-3.2) node{$S_2$};
\draw[red!90!black] (4.3,-2.2) node{$S_4$};
\draw (9,-2) node{\small$2S_9$};
\draw (5.5,-.5) node{\small$M_{06}$};
\draw (6.5,-1.5) node{\small$M_{37}$};
\draw (6.5,-2.5) node{\small$M_{57}$};
\draw (7.5,-3.5) node{\small$M_{59}$};
\foreach \y in {0,-1,...,-4}
\draw (-1,\y) node{$\y$};
\foreach \x in {0,1,...,10}
\draw (\x,-5) node{$\x$};
\draw (.33,-1)--(7,-1);
\draw (3.5,-2)--(7.5,-2); %(.66,-2)--(2.67,-2)  (8.5,-2)--(9.33,-2);
\draw (1,-3)--(2.33,-3) (5.5,-3)--(9.67,-3);
\draw[thick,blue] (0,0)--(6,0);
\draw[thick,red!90!black] (6,-4)--(10,-4);
\draw[blue,very thick] (6,0)--(7,-1)--(8,-3)--(9,-1)--(10,-4);
\foreach \x/\y in {7/-1,8/-3,9/-1,10/-4}
\draw[blue,fill] (\x,\y) circle[radius=1mm];
\draw[fill,white] (6,0) circle[radius=1mm];
\draw[blue,thick] (6,0) circle[radius=1mm];
\draw[very thick,red] (0,0)--(1,-3)--(2,-4)--(3,-1)--(4,-3)--(5,-2)--(6,-4);
\foreach \x/\y in {1/-3,2/-4,3/-1,4/-3,5/-2,6/-4}
\draw[red!90!black,fill] (\x,\y) circle[radius=1mm];
\draw[fill,white] (0,0) circle[radius=1mm];
\draw[red!90!black,thick] (0,0) circle[radius=1mm];
\end{tikzpicture}
\caption{The function $F:(0,n+1]\to \RR$ is given by the red function $R$ on $(0,6]$ since the first 5 arrows point left and by the blue function $B$ on $(6,10]$. A module $M_{(a,b]}$ is semistable if there is a horizontal line from $(a,y)$ to $(b,y)$ so that $R(x)\le y\le B(x)$ for all $a\le x\le b$. ``Islands'' $M_{(a,b]}$ for $b<5$ are shaded.}
\label{Fig: F=B,R, finite case}
\end{center}
\end{figure}
%
%\newpage

The generic decomposition of $\bf d$ is given by ${\bf d}=\sum a_i\beta_i$ where the coefficient $a_i$ of $\beta_i=\beta_{ab}$ is the linear measure of the set of all $c\in \RR$ so that $M_{(x,y]}$ is semistable with $F(x)=c=F(y)$ and so that $\ZZ\cap (x,y]=\{a+1,\cdots,b\}$.
For example, in Figure~\ref{Fig: F=B,R, finite case}, the coefficient of $\beta_{02}$ is the measure of the vertical interval $[-3,-1]$ which is 2. For $c$ in this vertical interval the horizontal line at level $c$ goes from the red line between $\frac13$ and 1 to the blue line between $\frac{7}3$ and $3$ with blue lines above and red lines below.
(We extend the red and blue functions to the interval $(0,10]$ as indicated.) We require $R(x)\le B(x)$ for all $x\in \RR$.

We interpret the stability function $\theta$ to be the derivative of $F$ where we consider $R,B$ separately. So, $\theta$ is a step function equal to $-3,-1,3,-2,1,-2$ on the six red unit intervals between 0 and 6 and $\theta$ is $-1,-2,2,-3$ on the four blue intervals from 6 to 10. $\theta$ is 4 times the dirac delta function at 6. For example,
\[
    \theta(M_{(a,b]})=\int_a^b \theta(x)\dd x=F(b)-F(a)=0
\]
for $a=3+\varepsilon,b=7+\varepsilon$ with $0\le \epsilon\le 1$ since $F(a)=F(b)=-1-2\varepsilon$ in this range. However, $F(5)=-2$ which is greater than $-1-2\varepsilon$ for $\varepsilon>1/2$. So, $M_{(3+\varepsilon,7+\varepsilon]}$ is semistable only when $0\le \varepsilon\le 1/2$. Taking only the integers in the interval $(3+\varepsilon,7+\varepsilon]$, we get $M_{(3,7]}$ to be semistable.

\subsection{APR-tilting} For quivers of type $\mathbb{A}_n$, we would like all arrows to be pointing in the same direction. We accomplish this with APR-tilting \cite{APR}.

We recall that APR-tilting of a quiver $Q$ is given by choosing a sink and reversing all the arrows pointing to that sink, making it a source in a new quiver $Q'$. Modules $M$ of $Q$ correspond to modules $M'$ of $Q'$ with the property that
\[
    \Hom(M',N')\oplus \Ext(M',N')\cong \Hom(M,N)\oplus \Ext(M,N)
\]
for all pairs of $\Bbbk Q$-modules $M,N$. This gives a bijection between exceptional sequences for $\Bbbk Q$ and for $\Bbbk Q'$. However, generic modules are given by sets of ext-orthogonal modules. So, we need to modify this proceedure.

In our example, we have a quiver $Q$ with 5 arrows pointing left. By a sequence of APR-tilts we can make all of these point to the right.
The new quiver $Q'$ will have all arrows pointing to the right. Any $\Bbbk Q$-module $M_{(a,b]}$ with $a\le5<6$ gives ath $\Bbbk Q'$-module $M_{(5-a,b]}$.
For example $M_{(0,6]}, M_{(3,7]}, M_{(5,7]},M_{(5,9]}$ become $M_{(5,6]}, M_{(2,7]}, M_{(0,7]},M_{(0,9]}$. See Figure \ref{Fig: APR-tilt of Fig 2}.
For $a>5$, such as $M_{(8,9]}=S_9$, the module is ``unchanged''. For $b\le5$, the APR-tilt of $M_{(a,b]}$ is $M_{(5-b,5-a]}$.
However, these are not in general ext-orthgonal to the other modules in our collection. For example, the APR-tilt of $S_4=M_{(3,4]}$ is $M_{(1,2]}=S_2$ which extends $M_{(2,7]}$.
So we need to shift it by $\tau^{-1}$ to get $\tau^{-1}M_{(5-b,5-a]}=M_{(4-b,4-a]}$. There is a problem when $b=5$ since, in that case $4-b=-1$.
This problem will disappear in the continuous case. We call modules $M_{(a,b]}$ with $b<5$ \emph{islands}.
We ignore the problem case $b=5$. Islands are shaded in Figure \ref{Fig: F=B,R, finite case}.
Shifts of their APR-tilts are shaded in Figure \ref{Fig: APR-tilt of Fig 2}.

\begin{figure}[htbp]
\begin{center}
\begin{tikzpicture}[scale=.8]
\draw[fill,red!10!white] (.66,2)--(1,3)--(2,2);
\draw[fill,red!10!white] (2.25,3)--(3,6)--(4,5)--(5,3);

\foreach \x in {.5,1.5,2.5}
\draw (1,\x) node{\small\sf x};
\foreach \x in {.5,1.5}
\draw (2,\x) node{\small\sf x};
\foreach \x in {.5,1.5,2.5,3.5,4.5,5.5}
\draw (3,\x) node{\small\sf x};
\foreach \x in {.5,1.5,2.5,3.5,4.5}
\draw (4,\x) node{\small\sf x};
\foreach \x in {.5,1.5,2.5}
\draw (5,\x) node{\small\sf x};
\foreach \x in {.5,1.5,2.5,3.5}
\draw (6,\x) node{\small\sf x};
\foreach \x in {.5,1.5,2.5}
\draw (7,\x) node{\small\sf x};
\foreach \x in {.5}
\draw (8,\x) node{\small\sf x};
\foreach \x in {.5,1.5,2.5}
\draw (9,\x) node{\small\sf x};
\draw[red!90!black] (3.5,4) node{$2M_{24}$};
\draw[red!90!black] (1.35,2.25) node{$S_1$};
\draw[red!90!black] (3.35,5.25) node{$S_3$};
\draw (9,2) node{\small$2S_9$};
\draw (5.8,3.2) node{\small$M_{56}$};
\draw (6.5,2.5) node{\small$M_{27}$};
\draw (6.5,1.5) node{\small$M_{07}$};
\draw (7.5,0.5) node{\small$M_{09}$};
\foreach \y in {0,1,...,6}
\draw (-1,\y) node{$\y$};
\foreach \x in {0,1,...,10}
\draw (\x,-1) node{$\x$};
\draw (.33,1)--(9.67,1);
\draw (.66,2)--(7.5,2);%  (8.5,2)--(9.33,2);
\draw (2.25,3)--(7,3);
\draw (2.75,5)--(4,5);
\draw[thick,red!90!black] (0,0)--(10,0);
\draw[very thick,blue] (0,0)--(1,3)--(2,2)--(3,6)--(4,5)--(5,3)--(6,4)--(7,3)--(8,1)--(9,3)--(10,0);
\foreach \x/\y in {1/3,2/2,3/6,4/5,5/3,6/4,7/3,8/1,9/3,10/0}
\draw[blue,fill] (\x,\y) circle[radius=1mm];
\draw[fill,white] (0,0) circle[radius=1mm];
\draw[blue,thick] (0,0) circle[radius=1mm];
\end{tikzpicture}
\caption{This is given by APR-tilting of Figure \ref{Fig: F=B,R, finite case}. The modules $M_{(a,b]}$ from Figure \ref{Fig: F=B,R, finite case} with $a\le 5<b$ become $M_{(5-a,b]}$ by APR-tilting. The ``islands'' $M_{(a,b]}$ in Figure \ref{Fig: F=B,R, finite case} gave $\tau^{-1} M_{(5-b,5-a]}=M_{(4-b,4-a]}$ above (shaded).}
\label{Fig: APR-tilt of Fig 2}
\end{center}
\end{figure}
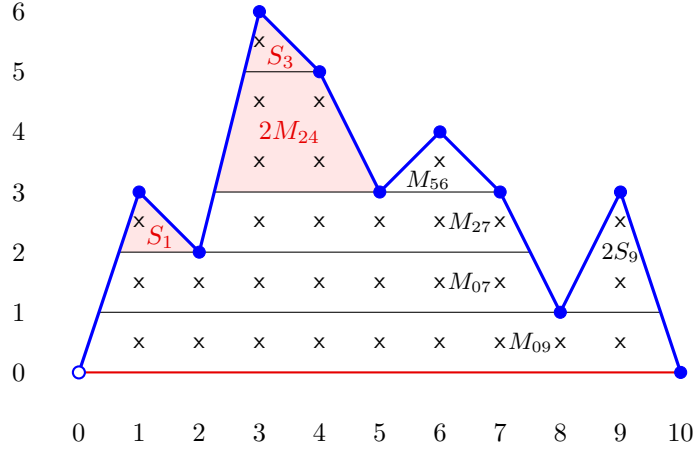

\subsection{Clusters and cluster algebras}

The components of a generic decomposition of any module form a partial tilting object since they do not extend each other. In the example shown in Figure \ref{Fig: APR-tilt of Fig 2}, we have 8 objects:
\[
	M_{01}, M_{07}, M_{09}, M_{23}, M_{24}, M_{27}, M_{56}, M_{89}.
\]
Since the quiver $A_9$ has rank $9$, we need one more to complete the tilting object. There are two other modules that we could add to complete this tilting object. They are $X=M_{26}$ and $X^\ast=M_{57}$. There are always at most two objects that will complete a tilting object with $n-1$ components. Tilting objects are examples of clusters and, in the cluster category \cite{BMRRT}, there are always exact two objects which complete a cluster with $n-1$ components.

These two objects $M_{26}$, $M_{57}$ extend each other in the cluster category with extensions:
\[
	M_{57}\to M_{27}\oplus M_{56}\to M_{26}
\]
and 
\[
	M_{26}\to M_{24}\to M_{46}[1]=\tau^{-1}M_{57}[1].
\]
In the cluster category, a module $M$ over any hereditary algebra is identified with $\tau^{-1}M[1]$. Thus, an exact sequence $\tau^{-1}M \hookrightarrow A\twoheadrightarrow B$ gives an exact triangle $M[-1]\to A\to B\to M$ in the cluster category since $\tau^{-1}M=M[-1]$.

In the cluster algebra \cite{FZ1}, which is the subalgebra of $\mathbb Q(x_1,\cdots,x_n)$ generated by ``cluster variables'', we have a formula due to Caldero and Chapoton \cite{CCC} which associates a cluster variable $\chi(M)$ to every rigid indecomposable module $M$ and, in this case, satisfies the equation:
\begin{equation}\label{eq: mutation of modules}
	\chi(X)\chi(X^\ast)=\chi(M_{27})\chi(M_{56})+\chi(M_{24})
\end{equation}
The Caldero-Chapoton formula for the cluster character of $M_{ab}$ for $1<a<b<n$ with arrows going right is {the sum of the inverses of exponential $g$-vectors of all submodules $x^{g(M_{ib})}=x_i/x_{b}$ times that of the duals of their quotients $M_{ab}/M_{ib}=M_{ai}$ (see \cite{maresca5})}:
\begin{equation}\label{eq: CC(M) finite straight}
	\chi(M_{ab})=\sum_{i=a}^b x^{-g(M_{ib})}x^{-g(DM_{ai})}= \sum_{i=a}^b \frac{x_{b}x_{a-1}}{x_ix_{i-1}}.
\end{equation}
So, $\chi(M_{aa})=\chi(0)=1$. When $b=n+1$, $M_{ab}$ is projective with support $[a,n+1)=[a,n]$. So,
\[
	\chi(P_a)=\chi(M_{a,n+1})=\sum_{i=a}^{n+1} \frac{x_{a-1}}{x_ix_{i-1}}
\]
where $x_{n+1}=1$. This yields:
\[
	\chi(M_{ab})=x_{b}\chi(P_a)-x_{a-1}\chi(P_{b+1}).
\]	
Then, the mutation equation \eqref{eq: mutation of modules} becomes the Pl\"ucker relation for the $2\times 4$ matrix:
\[
\left[\begin{matrix}
x_1 & x_4 & x_6& x_7\\
\chi(P_2) &\chi(P_5) &\chi(P_7) &\chi(P_8)
\end{matrix}
\right].
\]

\section{Continuous stability conditions}\label{sec:continuous stability conditions}

\subsection{Continuous quivers of type $\mathbb{A}$}
Recall that in a partial order $\preceq$, a element $x$ is a  \textdef{sink} if $y\preceq x$ implies $y=x$.
Dually, $x$ is a \textdef{source} if $x\preceq y$ implies $y=x$.

\begin{definition}\label{def:quiver}
    Let $\preceq$ be a partial order on $\Rbar=\RR\cup\{-\infty,+\infty\}$ with finitely many sinks and sources such that, between sinks and sources, $\preceq$ is either the same as the usual order or the opposite.
    Let $Q=(\RR,\preceq)$ where $\preceq$ is the same partial order on $\RR\subseteq\Rbar$.
    We call $Q$ a \textdef{continuous quiver of type $\mathbb{A}$}.
    We consider $Q$ as a category where the objects of $Q$ are the points in $\RR$ and
    \begin{displaymath}
        \Hom_Q(x,y) = \begin{cases}
            \{*\} & y\preceq x \\
            \emptyset & \text{otherwise}.
        \end{cases}
    \end{displaymath}
\end{definition}
    
\begin{definition}\label{def:representation}
    Let $Q$ be a continuous quiver of type $\mathbb{A}$.
    A \textdef{pointwise finite-dimensional $Q$ module} over the field $\Bbbk$ is a functor $V:Q\to \veck$.
    Let $I\subset\RR$ be an interval.
    An \textdef{interval indecomposable module} $M_I$ is given by
    \begin{align*}
        M_I(x) :&= \begin{cases}
            \Bbbk & x\in I \\
            0 & x\notin I
        \end{cases} &
        M_I(x,y) :&= \begin{cases}
            1_{\Bbbk} & y\preceq x,\, x,y \in I \\
            0 & \text{otherwise},
        \end{cases}
    \end{align*}
    where $I\subseteq\RR$ is an interval.
\end{definition}

By results in \cite{BC-B20,IRT22} we know that every pointwise finite-dimensional $Q$ module is isomorphic to a direct sum of interval indecomposables.
In particular, this decomposition is unique up to isomorphism and permutation of summands.
In \cite{IRT22} it is shown that the category of pointwise finite-dimensional modules is abelian, interval indecomposable modules are indecomposable, and there are indecomposable projectives $P_a$ for each $a\in\RR$ given by
\begin{align*}
    P_a(x) &= \begin{cases}
            \Bbbk & x\preceq a \\
            0 & \text{otherwise}
        \end{cases}
    & 
    P_a(x,y) &= \begin{cases}
            1_\Bbbk & y\preceq x \preceq a \\
            0 & \text{otherwise}.
        \end{cases}
\end{align*}
These projectives are representable as functors. %since $P_a(x)\cong \Hom_Q(a,x)$.

\begin{definition}\label{def:representable}
    Let $Q$ be a continuous quiver of type $\mathbb{A}$.
    We say $V$ is \textdef{representable} if there is a finite direct sum $P=\bigoplus_{i=1}^n {P_{a_i}}$ and an epimorphism $P\twoheadrightarrow V$ whose kernal is a direct sum $\bigoplus_{j=1}^m {P_{a_j}}$.
\end{definition}

By \cite[Theorem 3.0.1]{IRT22}, $V$ is isomorphic to a finite direct sum of interval indecomposables.
By results in \cite{R19+}, the subcategory of representable modules is abelian (indeed, a wide subcategory) but has no injectives.
When $\preceq$ is the standard total order on $\RR$, the representable modules are the same as those considered in \cite{IT15}.

\begin{notation}\label{note:rrep}
    We denote the abelian subcategory of representable modules over $Q$ by $\modr(Q)$.
    We denote the set of isomorphism classes of indecomposables in $\modr(Q)$ by $\Indr(Q)$.
\end{notation}

\begin{definition}\label{def:red and blue intervals}
    Let $Q$ be a continuous quiver of type $\mathbb{A}$, $s\in\Rbar$ a sink, and $s'\in\Rbar$ an adjacent source.
    \begin{itemize}
        \item If $s<s'$ and $x\in(s,s')$ we say $x$ is \textdef{red} and $(s,s')$ is \textdef{red}.
        \item If $s'<s$ and $x\in(s',s)$ we say $x$ is \textdef{blue} and $(s',s)$ is \textdef{blue}.
    \end{itemize}
    Let $I$ be an interval in $\RR$ such that neither $\inf I$ nor $\sup I$ is a source.
    We will need to refer to the endpoints of $I$ as being red or blue the following way.
    \begin{itemize}
        \item If $\inf I$ is a sink and $\inf I\in I$ we say $\inf I$ is \textdef{blue}.
        \item If $\inf I$ is a sink and $\inf I\notin I$ we say $\inf I$ is \textdef{red}. 
        \item If $\sup I$ is a sink and $\sup I\in I$ we say $\sup I$ is \textdef{red}.
        \item If $\sup I$ is a sink and $\sup I\notin I$ we say $\sup I$ is \textdef{blue}. 
        \item If $\inf I$ is not a sink ($\sup I$ is not a sink) then we say $\inf I$ ($\sup I$) is red or blue according to the first part of the definition.
    \end{itemize}
\end{definition}

Note that $\inf I$ could be $-\infty$, in which case it is red. Similarly, if $\sup I=+\infty$ then it is blue.
\begin{definition}\label{def:left and right colors}
    We say $I$ is \textdef{left red} (respectively, \textdef{left blue}) if $\inf I$ is red (respectively, if $\inf I$ is blue).
    
    We say $I$ is \textdef{right red} (respectively, \textdef{right blue}) if $\sup I$ is red (respectively, if $\sup I$ is blue).
\end{definition}
We have the following characterization of support intervals.

\begin{proposition}
    Let $I\subset\RR$ be the support of an indecomposable representable module $M_I\in\Indr(Q)$. Then an endpoint of $I$ lies in $I$ if and only if it is either left blue or right red (or both, as in the case $I=[s,s]$ where $s$ is a sink).
\end{proposition}

\subsection{Half-$\delta$ functions and red-blue function pairs}\label{sec:half delta and red-blue pairs}

To define continuous stability conditions we need to introduce half-$\delta$ functions.
A half-$\delta$ function $\delta_x^-$ at $x\in\Rbar$ has the following property.
Let $f$ some integrable function on $[a,b]\subset\Rbar$ where $a<x<b$.
Then the following equations hold:
\begin{align*}
    \int_a^x \left(f(t) + \delta_x^-\right) \dd t &=\left( \int_a^x f(t) \dd t\right) + 1, &
    \int_x^b\left( f(t) + \delta_x^-\right) \dd t &= \int_x^b f(t) \dd t.
\end{align*}
The half-$\delta$ function $\delta_x^+$ at $x\in\RR$ has a similar property for an $f$ integrable on $[a,b]\subset\Rbar$ with $a<x<b$:
\begin{align*}
    \int_a^x \left(f(t) + \delta_x^+\right) \dd t &= \int_a^x f(t) \dd t, &
    \int_x^b \left(f(t) + \delta_x^+\right) \dd t &= \left(\int_x^b f(t) \dd t\right) + 1.
\end{align*}
Consider $f+\delta_x^- - \delta_x^+$.
Then we have
\begin{align*}
    \int_a^x \left(f(t) + \delta_x^- - \delta_x^+\right) \dd t &=\left( \int_a^x f(t) \dd t\right) + 1, \\
    \int_x^b \left( f(t) + \delta_x^- - \delta_x^+\right) \dd t &= \left(\int_x^b f(t) \dd t\right) - 1,\\
    \int_a^b \left(f(t) + \delta_x^- - \delta_x^+\right) \dd t &= \int_a^b f(t)\dd t.
\end{align*}

For each $x\in\RR$, denote the functions
\begin{align*}
    \Delta_x^-(z) &= \int_{-\infty}^z \delta_x^- \dd t = \begin{cases} 0 & z< x \\ 1 & z\geq x \end{cases} \\
    \Delta_x^+(z) &= \int_{-\infty}^z \delta_x^+ \dd t = \begin{cases} 0 & z\leq x \\ 1 & z > x. \end{cases}
\end{align*}

Though not technically correct, we write that a function $f+u_x^-\Delta_x^-+u_x^+\Delta_x^+$ is from $\RR$ to $\RR$.
See Figure~\ref{fig:delta example} for an example.
\begin{figure}
    \centering
\begin{tikzpicture}%[scale=3]
\coordinate (A) at (-2,-2);
\coordinate (Ap) at (-2,-1.7);
\coordinate (AA) at (-5.5,-2);
\coordinate (A1) at (-1,-.8);
\coordinate (A11) at (-1,-.3);
\coordinate (B0) at (0,0);
\coordinate (B2) at (0,1);
\coordinate (B1) at (0,-1);
\coordinate (C11) at (1.1,-.2);
\coordinate (C1) at (1.1,-.7);
\coordinate (C) at (2,1);
\coordinate (Cp) at (2,1.3);
\coordinate (D) at (5.5,1);
\draw[thick] (AA)--(A)--(B0)--(B2) (B1)--(C)--(D);
\draw[very thick] (B1)--(B2);
\draw (Cp) node[right]{$F(x)=1$ for $2<x$};
\draw (Ap) node[left]{$F(x)=-2$ for $x<-2$};
\draw (B2) node[left]{$F(0)=1$};
\draw (A11) node[left]{$F(x)=x$};
\draw (A1) node[left]{for $-2\le x<0$};
\draw (C11) node[right]{$F(x)=x-1$};
\draw (C1) node[right]{for $0<x\le 2$};
\end{tikzpicture}
    \caption{Graph of $F(x)=f(x)+\Delta_0^-(x)-2\Delta_0^+(x)$ when $f$ is given by: $f(x) =x$ if $|x|\le2$, $f(x)=-2$ for $x<-2$, $f(x)=2$ if $x>2$, $u_0^-=1$, $u_0^+=-2$, and $u_x^\pm=0$ for all other $x\in\RR\cup\{\pm\infty\}$.}
    \label{fig:delta example}
\end{figure}
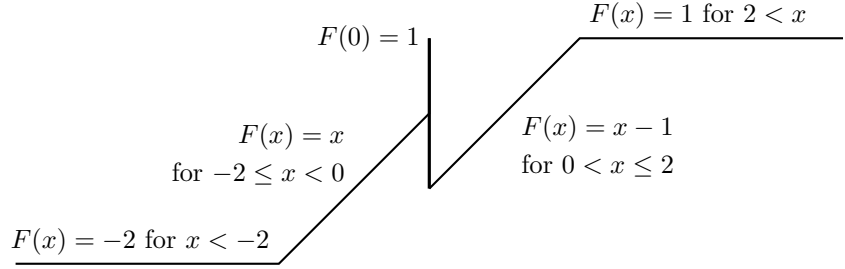
We also allow $\delta_{+\infty}^-$ and $\delta_{-\infty}^+$, which satisfy the relevant parts of the equations above.
We don't allow the other half-$\delta$ functions at $\pm\infty$ because it does not make sense in terms of integration.

Our stability conditions will be comprised of equivalence classes of pairs of useful functions.
\newpage 
\begin{definition}\label{def:useful}
    We call a function $F:\RR\to\RR$ \textdef{useful} if it satisfies the following.
    \begin{enumerate}
        \item $\displaystyle F = f + \sum_{x\in\RR\cup\{+\infty\}} u_x^- \Delta_x^- + \sum_{x\in\RR\cup\{-\infty\}}u_x^+\Delta_x^+$, where $f:\RR\to\RR$ is a continuous function of bounded variation and each $u_x^-, u_x^+$ are in $\RR$.
        \item The sums $\displaystyle\sum_{x\in\RR\cup\{+\infty\}} |u_x^-|$ and $\displaystyle\sum_{x\in\RR\cup\{-\infty\}} |u_x^+|$ both converge in $\RR$.
    \end{enumerate}
\end{definition}

\begin{remark}
Note Definition~\ref{def:useful}(2) implies the set $\{u_x^-\mid u_x^-\neq 0\}\cup\{u_x^+\mid u_x^+\neq 0\}$ is at most countable.
Combining (1) and (2) in Definition~\ref{def:useful} means we think of $F$ as having the notion of bounded variation.

We think of the value of a useful function $F$ at $x$ as being ``the integral from $-\infty$ to $x$'' where the integrand is some function that includes at most countably-many half-$\delta$ functions.
\end{remark}

\begin{proposition}\label{prop:useful functions are useful}
    Let $F$ be a useful function and let $a\in\overline{\RR}$.
    \begin{enumerate}
        \item\label{prop:useful functions are useful:left limit} If $a > -\infty$ then $\lim_{x\to a^-} F(x)$ exists.
        \item\label{prop:useful functions are useful:right limit} If $a < +\infty$ then $\lim_{x\to a^+} F(x)$ exists.
        \item\label{prop:useful functions are useful:limit equations} If $a\in\RR$ then $F(a) = \lim_{x\to a^-} F(x) + u_a^-$ and $F(a)+u_a^+ = \lim_{x\to a^+} F(x)$.
    \end{enumerate}
\end{proposition}
\begin{proof}
    (1) and (2).
    Straightforward computations show that
    \begin{align*}
        \lim_{x\to a^-} F(x)  &= \lim_{x\to a^-} f(x) + \sum_{-\infty<x<a} u_x^- + \sum_{-\infty\leq x < a} u_x^+ &\text{if }a>-\infty \\
        \lim_{x\to a^+} F(x) &= \lim_{x\to a^+} f(x) + \sum_{-\infty<x\leq a} u_x^- + \sum_{-\infty\leq x \leq a} u_x^+ & \text{if }a<+\infty.
    \end{align*}
    Thus, (1) and (2) hold.
    
    (3).
    By definition, we see that \[F(a)= f(a) + \sum_{-\infty<x\leq a} u_x^- + \sum_{-\infty\leq x < a} u_x^+.\]
    Thus, using (1) and (2), we see that (3) holds.
\end{proof}

\begin{notation}\label{note:useful limits}
    Let $F$ be a useful function.
    For each $x\in\RR$, we define
    \begin{align*}
        \minn{F}(a) :&= \min\{F(a), \lim_{x\to a^-} F(x), \lim_{x\to a^+} F(x)\} \\
        \maxx{F}(a) :&= \max\{F(a), \lim_{x\to a^-} F(x), \lim_{x\to a^+} F(x)\}.
    \end{align*}
    We also define
    \begin{align*}
        F(-\infty) :&= \lim_{x\to -\infty^+} F(x) - u_{-\infty}^+ &
        F(+\infty) :&= \lim_{x\to +\infty^-} F(x) + u_{+\infty}^-
    \end{align*}
    \begin{align*}
        \minn{F}(-\infty) :&= \min\{F(-\infty),\lim_{x\to-\infty^+} F(x)\} \\
        \minn{F}(+\infty) :&= \min\{F(-\infty),\lim_{x\to+\infty^-} F(x)\} \\
        \maxx{F}(-\infty) :&= \max\{F(-\infty),\lim_{x\to-\infty^+} F(x)\} \\
        \maxx{F}(+\infty) :&= \max\{F(-\infty),\lim_{x\to+\infty^-} F(x)\}.
    \end{align*}
\end{notation}

\begin{definition}\label{def:graph of useful}
    Let $F$ be a useful function.
    We define the \textdef{graph} $\mathcal{G}(F)$ of $F$ to be the following subset of $\RR^2$: \[ \left\{ (x,y) \left|\, x\in\RR,\ \minn{F}(x) \leq y \leq \maxx{F}(x) \right\} \right. . \]
    
    The \textdef{completed graph}, denoted $\overline{\mathcal{G}(F)}$ of $F$ is the following subset of $\overline{\RR}\times\RR$:
    \[ \left\{ (x,y) \left|\, x\in\overline{\RR},\ \minn{F}(x) \leq y \leq \maxx{F}(x) \right\} \right. . \]
\end{definition}

\begin{remark}\label{rmk:bounded graph}
    Let $\displaystyle F = f + \sum_{x\in\RR\cup\{+\infty\}} u_x^- \Delta_x^- + \sum_{x\in\RR\cup\{-\infty\}}u_x^+\Delta_x^+$ be a useful function.
    For any $a\leq b\in\RR$ there exists  $c\leq d\in\RR$, such that $\mathcal{G}(F)\cap([a,b]\times\RR) = \mathcal{G}(F)\cap([a,b]\times[c,d])$.
\end{remark}

We now define red-blue function pairs, which are used to define equivalence classes of pairs of useful functions.
The red-blue function pairs are analogs of the red and blue functions from Section~\ref{sec:finite case}.
\begin{definition}\label{def:red-blue function pair}
    Let $R=r+\sum_{x\in\RR} u_x^-\Delta_x^-+\sum_{x\in\RR} u_x^+\Delta_x^+$ and let $B=b+\sum_{x\in\RR} v_x^-\Delta_x^- + \sum_{x\in\RR} v_x^+\Delta_x^+$ be useful functions.
    We say the pair $(R,B)$ is a \textdef{red-blue function pair} if the following criteria are satisfied.
    \begin{enumerate}
        \item\label{def:red-blue function pair:make u and v easy} For all $x\in\RR$, we have $\maxx{R}(x) = R(x)$ and $\minn{B}(x)=B(x)$.
        \item\label{def:red-blue function pair:no variation at sources} If $s$ is a source, $u_s^-=u_s^+=v_x^-=v_x^+=0$.
        \item\label{def:red-blue function pair:graphs don't cross} For all $x\in\overline{\RR}$,
        \begin{align*}
            R(x) \leq \maxx{B}(x) && \text{and} && \minn{R}(x) \leq B(x).
        \end{align*}
        \item\label{def:red-blue function pair:equal at infinities} We have $R(-\infty)=B(-\infty)$ and $R(+\infty)=B(+\infty)$.
        \item\label{def:red-blue function pair:R constant on blue} The useful function $R$ is constant on blue intervals. That is: for $s\leq x<y< s'$ in $\RR$ where $(s,s')$ is blue, we have $r(x)=r(y)$ and $u_y^-=u_y^+=0$.
        \item\label{def:red-blue function pair:B constant on red} The useful function $B$ is constant on red intervals. That is: for $s < x < y\leq s'$ in $\RR$ where $(s,s')$ is red, we have $b(x)=b(y)$ and $v_x^-=v_x+=0$.
    \end{enumerate}
\end{definition}

\begin{lemma}\label{lem:local max and local min}
    Let $(R,B)$ be a red-blue function pair.
    \begin{enumerate}
        \item\label{lem:local max and local min:closed box} For any $a\leq b$ and $c\leq d$ in $\RR$, the set $\mathcal{G}(R)\cap ([a,b]\times[c,d])$ is closed in $\RR^2$.
        
        \item For any $a\leq b\in\RR$ the useful function $R$ has a local maximum on $[a,b]$ in the sense that there exists $x\in[a,b]$ such that for all $y\in[a,b]$: $\maxx{R}(y) \leq \maxx{R}(x)$.
        \item For any $a\leq b\in\RR$ the useful function $R$ has a local minimum on $[a,b]$ in the sense that there exists $x\in[a,b]$ such that for all $y\in[a,b]$: $\minn{R}(x) \leq \minn{R}(y)$.
    \end{enumerate}
    Statements (1)--(3) are true when we replace $R$, $r$, and $u$ with $B$, $b$, and $v$, respectively.
\end{lemma}
\begin{proof}
    We first prove (1) for $R$ as the proof for $B$ is identical.
    Let $\{(x_i,y_i)\}$ be a sequence in $\mathcal{G}(R)\cap([a,b]\times[c,d])$ that converges to $(w,z)$.
    Since $\{x_i\}$ converges to $w$ we assume, without loss of generality, that $\{x_i\}$ is monotonic.
    If there exists $i\in \NN$ such that $x_i=w$ then, assuming monotonicity, we know $(w,z)\in\mathcal{G}(R)$.
    Thus, assume $x_i\neq w$ for all $i\in\NN$.

    Without loss of generality, assume $\{x_i\}$ is increasing.
    The decreasing case is similar.
    Since $\sum_{x\in\overline{\RR}} |u_x^-|+|u_x^+|<\infty$, we know that
    \[\lim_{i\to\infty} |\maxx{R}(x_i) - \minn{R}(x_i)| = 0.\]
    And so,
    \[\lim_{i\to\infty} \maxx{R}(x_i) = \lim_{i\to\infty} \minn{R}(x_i) =\lim_{i\to\infty} R(x_i) = \lim_{\to w^-} R(x).\]
    Then we must have
    \[ \lim_{i\to\infty} y_i = \lim_{\to w^-} R(x). \]
    Therefore, $(w,z)\in \mathcal{G}(R)$.
    
    Next, we only prove (2) for $R$ as the remaining proofs are similar and symmetric.
    By Remark~\ref{rmk:bounded graph} there exists $c\leq d\in\RR$ such that \[\mathcal{G}(R)\cap([a,b]\times\RR)=\mathcal{G}(R)\cap([a,b]\times[c,d]).\]
    Then there must be a greatest lower bound $d_0\geq c$ for all $d$ such that the equation above holds.
    Since $\mathcal{G}(R)\cap([a,b]\times[c,d_0])$ must be closed by Lemma~\ref{lem:local max and local min}(\ref{lem:local max and local min:closed box}), there must be a point $(x,d_0)\in\mathcal{G}(R)$ for $a\leq x\leq b$.
    This is the desired $x$.
\end{proof}

\subsection{Stability conditions}

\begin{definition}\label{def:stability condition}
    Let $(R,B)$ and $(R',B')$ be red-blue function pairs.
    We say $(R,B)$ and $(R',B')$ are \textdef{equivalent} if there exists a constant $\mathfrak{c}\in\RR$ such that, for all $x\in \overline{\RR}$ and $y\in\RR$, we have
        \[ (x,y)\in \overline{\mathcal{G}(R)} \text{ if and only if } (x,y+\mathfrak{c})\in \overline{\mathcal{G}(R')}\]
        \[\text{and}\]
        \[(x,y)\in \overline{\mathcal{G}(B)} \text{ if and only if } (x,y+\mathfrak{c})\in \overline{\mathcal{G}(B')}.\]
    
    A \textdef{stability condition on $Q$}, denoted $\sigma$, is an equivalence class of red-blue function pairs.
    We denote by $\mathcal{S}(Q)$ the set of stability conditions on $Q$.
\end{definition}

We now define the \textdef{modified} versions of a continuous quiver $Q$ of type $\mathbb{A}$, an interval $I$ of a module $M_I$ in $\Indr(Q)$, and graphs of red-blue function pairs.
This makes it easier to check whether or not an indecomposable module is semistable with respect to a particular stability condition.

\begin{definition}\label{def:modified}
{~}
\begin{enumerate}
    \item\label{def:modified:quiver}
    Let $Q$ be a continuous quiver of type $\mathbb{A}$ with finitely many sinks and sources.
    We define a totally ordered set $\widehat{Q}$, called the \textdef{modified quiver} of $Q$, in the following way.
    
    First we define the elements.
    \begin{itemize}
    \item For each $x\in\RR$ such that $x$ is not a sink nor a source of $Q$, $x\in \widehat{Q}$.
    \item If $s\in\RR$ is a source of $Q$, then $s\notin \widehat{Q}$.
    \item If $s\in\RR$ is a sink of $Q$, then $s_-,s_+\in \widehat{Q}$.
    These are distinct elements, neither of which is in $\RR$.
    \item If $-\infty$ (respectively, $+\infty$) is a sink then $-\infty_+\in \widehat{Q}$ (respectively, $+\infty_- \in \widehat{Q}$).
    \end{itemize}
    
    Now, the partial order on $\widehat{Q}$ is defined in the following way.
    Let $x,y\in\widehat{Q}$.
    \begin{itemize}
        \item Suppose $x,y\in\RR\cap\widehat{Q}$.
        Then $x\leq y$ in $\widehat{Q}$ if and only if $x\leq y$ in $\RR$.
        \item Suppose $x\in\RR$ and $y=s_\pm$, for some sink $s$ of $Q$ in $\RR$.
        If $x<s$ in $\RR$ then $x<y$ in $\widehat{Q}$.
        If $s<x$ in $\RR$ then $y<x$ in $\widehat{Q}$.
        \item Suppose $x=s_\varepsilon$ and $y=s'_{\varepsilon'}$ for two sinks $s,s'$ of $Q$ in $\RR$.
        We consider $-<+$.
        Then $x\leq y$ if and only if (i) $s<s'$ or (ii) $s=s'$ and $\varepsilon\leq \varepsilon'$.
        \item If $x=-\infty_+\in \widehat{Q}$ (respectively, $y=+\infty_-\in\widehat{Q}$), then $x$ is the minimal element (respectively, $y$ is the maximal element) of $\widehat{Q}$.
    \end{itemize}
    
    If $s\in\RR$ is a sink of $Q$ then we say $s_-$ is blue and $s_+$ is red.
    If $-\infty_+\in\widehat{Q}$ we say $-\infty_+$ is red.
    If $+\infty_-\in\widehat{Q}$ we say $-\infty_+$ is blue.
    All other $x\in\widehat{Q}$ are red (respectively, blue) if and only if $x\in\RR$ is red (respectively, blue).
    
    \item\label{def:modified:interval}
    Let $I$ be an interval of $\RR$ such that $M_I$ is in $\Indr(Q)$.
    The \textdef{modified interval} $\widehat{I}$ of $\widehat{Q}$ has minimum given by the following conditions.
    \begin{itemize}
        \item If $\inf I$ is not $-\infty$ nor a sink of $Q$ then $\min \widehat{I}=\inf I$.
        \item If $\inf I$ is a sink $s$ of $Q$ then (i) $\min \widehat{I}=s_-$ if $\inf I\in I$ or (ii) $\min \widehat{I}=s_+$ if $\inf I\notin I$.
        \item If $\inf I = -\infty$ then $\min \widehat{I}=-\infty_+$.
    \end{itemize}
    The maximal element of $\widehat{I}$ is defined similarly.
    
    \item\label{def:modified:graph}
    Let $(R,B)$ be a red-blue funtion pair.
    The \textdef{modified graph} of $(R,B)$ is a subset of $\widehat{Q}\times\RR$.
    It is defined as follows.
    \begin{itemize}
        \item For each $x\in\RR$ not a sink nor a source of $Q$ and each $y\in\RR$,
        \begin{displaymath}
            (x,y)\in \widehat{G}(R,B) \text{ if and only if }[[(x,y)\in \mathcal{\mathcal{G}}(B)\text{ and }x\text{ is blue}]\text{ or }[(x,y)\in \mathcal{\mathcal{G}}(R)\text{ and }x\text{ is red}]]
        \end{displaymath}
        \item For each $s\in\RR$ a sink of $Q$ and each $y\in\RR$,
        \begin{align*}
            (s_-,y)\in\widehat{G}(R,B)&\text{ if and only if }(s,y)\in\mathcal{\mathcal{G}}(B) \\
            (s_+,y)\in\widehat{G}(R,B)&\text{ if and only if }(s,y)\in\mathcal{\mathcal{G}}(R).
        \end{align*}
        \item If $-\infty_+\in\widehat{Q}$, then for all $y\in\RR$,
        \begin{displaymath}
            (-\infty_+,y)\in\widehat{\mathcal{G}}(R,B)\text{ if and only if }(-\infty,y)\in\overline{\mathcal{G}(R)}.
        \end{displaymath}
        \item If $+\infty_-\in\widehat{Q}$, then for all $y\in\RR$,
        \begin{displaymath}
            (+\infty_-,y)\in\widehat{\mathcal{G}}(R,B)\text{ if and only if }(+\infty,y)\in\overline{\mathcal{G}(B)}.
        \end{displaymath}
    \end{itemize}
\end{enumerate}
\end{definition}

The following proposition follows from straightforward checks.
\begin{proposition}\label{prop:modified interval bijection}
    There is a bijection between $\Indr(Q)$ and intervals of $\widehat{Q}$ with distinct minimal and maximal element.
\end{proposition}

Using the modified definitions, we now define what it means to be semistable.

\begin{definition}\label{def:stable module}
    Let $Q$ be a continuous quiver of type $\mathbb{A}$ with finitely many sinks and sources, $\sigma\in\mathcal{S}(Q)$, and $(R,B)$ be a representative of $\sigma$.
    
    We say an indecomposable module $M_I$ in $\Indr(Q)$ is \textdef{$\boldsymbol{\sigma}$-semistable} if there exists a horizontal line $\ell=\widehat{I}\times\{h\}\subset\widehat{Q}\times\RR$ satisfying the following conditions.
    \begin{enumerate}
        \item\label{def:stable module:endpoints} The endpoints of $\ell$ touch $\widehat{\mathcal{G}}(R,B)$.
        That is, $(\min\widehat{I},h),(\max\widehat{I},h)\in\widehat{\mathcal{G}}(R,B)$.
        \item\label{def:stable module:no crossing} The line $\ell$ may touch but not cross $\widehat{\mathcal{G}}(R,B)$.
        That is, for each $x\in \widehat{I}$ such that $x\notin\{\max\widehat{I},\min\widehat{I}\}$, we have
        \begin{displaymath}
            \maxx{R}(x) \leq h \leq \minn{B}(x),
        \end{displaymath}
        where if $x=s_{\pm}$ then $\maxx{R}(x)=\maxx{R}(s)$ and $\minn{B}(x) = \minn{B}(s)$.
    \end{enumerate}
\end{definition}

\begin{remark}
    Notice that $M_I$ is $\sigma$-semistable whenever the following are satisfied:
    \begin{itemize}
        \item We have $[\minn{F}(\inf I), \maxx{F}(\inf I)]\cap[\minn{F'}(\sup I), \maxx{F'}(\sup I)]\neq\emptyset$, where $F$ is $R$ if $\inf I$ is red and is $B$ if $\inf I$ is blue and similarly for $F'$ and $\sup I$.
        \item For all submodules $M_J$ of $M_I$, $\minn{F'}(\sup J) \leq \minn{F}(\inf J)$, where $F,\inf J$ and $F',\sup J$ are similar to the previous point.
    \end{itemize}
    Thus, this is a continuous analogue to the semistable condition in the finite case.
\end{remark}

\begin{definition}\label{def:four point condition}
    Let $Q$ be a continuous quiver of type $\mathbb{A}$ with finitely many sinks and sources, $\sigma\in\mathcal{S}(Q)$, and $(R,B)$ be a representative of $\sigma$.
    
    We say $\sigma$ satisfies the \textdef{four point condition} if, for any $\sigma$-semistable module $M_I$, we have $|(\widehat{I}\times\{h\})\cap(\widehat{Q}\times\RR)|\leq 3$, where $\widehat{I}\times\{h\}$ is as in Definition~\ref{def:stable module}.
    We denote the set of stability conditions that satisfy the four point condition as $\mathcal{S}_\fpc(Q)$.
\end{definition}

Recall Definition~\ref{def:red and blue intervals}.
\begin{lemma}\label{lem:stable is compatible}
    Let $Q$ be a continuous quiver of type $\mathbb{A}$ with finitely many sinks and sources and let $M_I$, $M_J$ be indecomposables in $\Indr(Q)$.
    Let $a=\inf I$, $b=\sup I$, $c=\inf J$, and $d=\sup J$.
    Then $\Ext^1(M_J,M_I) \cong \Bbbk \cong \Hom(M_I,M_J)$ if and only if  one of the the following hold:
    \begin{itemize}
        \item $a<c<b<d$, and $b,c$ are red;
        \item $c<a<d<b$, and $a,d$ are blue;
        \item $c<a\leq b<d$, and $a$ is blue, and $b$ is red; or
        \item $a<c<d<b$, and $c$ is red, and $d$ is blue.
    \end{itemize}
\end{lemma}
\begin{proof}
     It is shown in \cite{IRT22} that $\Hom$ and $\Ext$ between indecomposables must be 0 or 1 dimensional.

    $\boldsymbol\Rightarrow$\textbf{.}
    Since $\Hom(M_I,M_J)\neq 0$ we obtain one of the items in the list where the first or last inequality may not be strict.
    Since $\Ext(M_I,M_J)\neq 0$ we see all the inequalities must be strict.
    
    $\boldsymbol\Leftarrow$\textbf{.}
    The itemized list implies $\Hom(M_I,M_J)\neq 0$.
    Then there is a short exact sequence $M_I\hookrightarrow M_{I\cup J}\oplus M_{I\cap J} \twoheadrightarrow M_J$ and so $\Ext(M_J,M_I)\neq 0$.
\end{proof}

\begin{definition}\label{def:Npi compatible}
    Let $M_I$ and $M_J$ be indecomposables in $\Indr(Q)$ for some continuous quiver of type $\mathbb{A}$ with finitely many sinks and sources.
    We say $M_I$ and $M_J$ are $\Npi$-\textdef{compatible} if both of the following are true:
    \begin{align*}
    \dim_{\Bbbk}(\Ext(M_J,M_I)\oplus \Hom(M_I,M_J)) &\leq 1 \\
    \dim_{\Bbbk}(\Ext(M_I,M_J)\oplus \Hom(M_J,M_I)) &\leq 1.
    \end{align*}
\end{definition}
One can verify this is equivalent to Igusa and Todorov's compatibility condition in \cite{IT15} when $Q$ has the straight descending orientation.

In terms of colors and set operations, $\Npi$-compatibility can be expressed as follows.
\newpage
\begin{lemma}\label{lem:Npi compatiblity as colors}
$M_I$ and $M_J$ are $\Npi$-compatible if one of the following is satisfied.
\begin{enumerate}
    \item $I\cap J=\emptyset$,
    \item $I\subset J$ and $J\setminus I$ is connected, or vice versa,
    \item $I\subset J$ and both endpoints of $I$ are the same color, or vice versa,
    \item $I\cap J\neq I$, $I\cap J\neq J$, and $I\cap J$ has endpoints of opposite color.
\end{enumerate}
\end{lemma}

\begin{theorem}\label{thm:four point equivalent to cluster}
    Let $\sigma\in\mathcal{S}(Q)$.
    The following are equivalent.
    \begin{itemize}
        \item $\sigma\in\mathcal{S}_\fpc(Q)$.
        \item The set of $\sigma$-semistable indecomposables is maximally $\Npi$-compatible.
    \end{itemize}
\end{theorem}
\begin{proof}
    Let $(R,B)$ be a representative of $\sigma$.
    
    $\boldsymbol\Leftarrow$\textbf{.}
    We prove the contrapositive.
    Suppose $\sigma$ does not satisfy the four point condition.
    Then there are $a<b<c<d$ in $\widehat{Q}$ that determine indecomposable modules $M_{a,b}$, $M_{a,c}$, $M_{a,d}$, $M_{b,c}$, $M_{b,d}$, $M_{c,d}$.
    Here, the notation $M_{x,y}$ means the interval indecomposable with interval $I$ such that $\min\widehat{I}=x$ and $\max\widehat{I}=y$.
    Using Lemma~\ref{lem:Npi compatiblity as colors} we see that at least two of the modules must be not $\Npi$-compatible.
    
    $\boldsymbol\Rightarrow$\textbf{.}
    Now suppose $\sigma$ satisfies the four point condition.
    By Lemma~\ref{lem:Npi compatiblity as colors} we see that the set of $\sigma$-semistable indecomposables is $\Npi$-compatible.
    We now check maximality.
    
    Let $M_J$ be an indecomposable in $\Indr(Q)$ such that $M_J$ is not $\sigma$-semistable.
    Recall left and right colors from Definition~\ref{def:left and right colors}.
    There are four cases depending on whether $J$ is left red or left blue and whether $J$ is right red or right blue.
    However, the case where $J$ is both left and right red red is similar to the case where $J$ is both left and right blue.
    Furthermore, the cases where $J$ is left red and right blue is similar to the case where $J$ left blue and right red.
    Thus we reduce to two cases where $J$ is left red: either (1) $J$ is right blue or (2) $J$ is right red.
    (Notice the case where $M_J$ is a simple projective $M_{[s,s]}$ is similar to the case where $J$ is left red and right blue.)
    
    \textbf{Case (1)}.
    Since $M_J$ is not $\sigma$-semistable, we first consider that $M_J$ fails Definition~\ref{def:stable module}(\ref{def:stable module:endpoints}) but satisfies Definition~\ref{def:stable module}(\ref{def:stable module:no crossing}).
    Notice that, in this case, it is not possible that $\inf J = -\infty$ or $\sup J = +\infty$.
    Since $M_J$ is left red, right blue, and fails Definition~\ref{def:stable module}(\ref{def:stable module:endpoints}), we must have $\maxx{R}(\inf J) < \minn{B}(\sup J)$.
    Otherwise, we could create a horizontal line segment in $\widehat{Q}\times\RR$ satisfying Definition~\ref{def:stable module}(\ref{def:stable module:endpoints}).
    Let $\varepsilon>0$ such that $0<\varepsilon<\minn{B}(\sup J)-\maxx{R}(\inf J)$.
    Let \[\ell = \widehat{Q}\times\{\maxx{R}(\inf J)+\varepsilon\}.\]
    By Lemma~\ref{lem:local max and local min}(\ref{lem:local max and local min:closed box}), there exists $w<\min\widehat{J}$ and $z>\max\widehat{J}$ in $\widehat{Q}$ such that the module $M_I$ corresponding to $[w,z]\subset\widehat{Q}$ (Proposition~\ref{prop:modified interval bijection}) is $\sigma$-semistable.
    
    Now suppose $M_J$ does not satisfy Definition~\ref{def:stable module}(\ref{def:stable module:no crossing}).
    First suppose there exists $x\in J$ such that $\maxx{R}(x) > \maxx{R}(\inf J)$.
    We extend the argument of the proof of Lemma~\ref{lem:local max and local min} to show that $\overline{\mathcal{G}(R)}$ must have global maxima in the following sense.
    There is some set $X$ such that, for all $x\in X$ and $y\notin X$, we have $\maxx{R}(y) < \maxx{R}(x)$ and, for each $x,x'\in X$, we have $\maxx{R}(x)=\maxx{R}(x')$.
    In particular, there is $z\in\widehat{Q}$ such that $\min\widehat{J}<z<\max\widehat{J}$ and for all $x$ such that $\min\widehat{J}\leq x < z$ we have $\maxx{R}(x)<\maxx{R}(z)$.
    If there is $x\in[\min\widehat{J},z]$ such that $\minn{B}(x)<\maxx{R}(z)$ then there is $w\in[\min\widehat{J},z]$ such that the module $M_I$ corresponding to $[w,z]$ is $\sigma$-semistable.
    In particular, $M_I$ is left blue and right red.
    By Lemma~\ref{lem:Npi compatiblity as colors} we see that $M_J$ and $M_I$ are not $\Npi$-compatible.
    If no such $x\in[\min\widehat{J},z]$ exists then there is a $w<\min\widehat{J}$ such that the module $M_I$ corresponding to $[w,z]$ is $\sigma$-semistable.
    Since $M_I$ is right red we again use Lemma~\ref{lem:Npi compatiblity as colors} and see that $M_I$ and $M_J$ are not $\Npi$-compatible.
    
    \textbf{Case (2)}.
    If $M_J$ satisfies Definition~\ref{def:stable module}(\ref{def:stable module:no crossing}) but fails Definition~\ref{def:stable module}(\ref{def:stable module:endpoints}), then the function $\maxx{R}(x)$ must be monotonic.
    If $\maxx{R}(x)$ is decreasing then let $x'=\inf J+\varepsilon$ be red.
    By Lemma~\ref{lem:local max and local min}(\ref{lem:local max and local min:closed box}) we can find some $\widehat{I}$ with left endpoint $x+\varepsilon$ and blue right endpoint $y'$ such that $y'>\sup J$ and $M_I$ is $\sigma$-semistable.
    By Lemma~\ref{lem:Npi compatiblity as colors}, $M_I$ and $M_J$ are not $\Npi$-compatible.
    A similar argument holds if $\maxx{R}(x)$ is monotonic increasing.

    Now suppose $M_J$ fails Definition~\ref{def:stable module}(\ref{def:stable module:no crossing}).
    The argument for the second half of Case (1) does not depend on whether $J$ is right red or right blue.
    Therefore, the theorem is true.
\end{proof}

Let $T$ and $T'$ be maximally $\Npi$-compatible sets.
We call a bijection $\mu:T\to T'$ a \textdef{mutation} if $T'=(T\setminus\{M_I\})\cup\{M_J\}$, for some $M_I\in T$ and $M_J\in T'$, and $\mu(M_K)=M_K$ for all $K\neq I$.
(Then $\mu(M_I)=M_J$.)

\section{Continuous tilting}\label{sec:continuous tilting section}

We construct a continuous version of tilting. Consider a stability condition $\sigma$ on a continuous quiver of type $\mathbb{A}$ where $-\infty$ is a sink and $s$ is either the smallest source or a real number less than the smallest source. Then continuous tilting at $s$ will replace the red interval $K=[-\infty,s)$ with the blue interval $K^\ast=(-\infty,s]$ and keep the rest of $Q$ unchanged. Thus, $\widehat Q=K\coprod Z$ is replaced with $\widehat Q^\ast=K^\ast\coprod Z$. We have an order reversing bijection $\mathfrak t:K\to K^\ast$ given by
\[
	\mathfrak t(x)=\tan\left(
	\tan^{-1}s-\tan^{-1}x-\frac\pi2
	\right).
\]
This extends, by the identity on $Z$, to a bijection $\overline{\mathfrak t}:\widehat Q\to \widehat Q^\ast$.

\subsection{Compatibility conditions}\label{sec 3.1}
We start with the continuous compatibility conditions for representable modules over the real line.
Given a continuous quiver $Q$ of type $\mathbb{A}$, we consider intervals $I$ in $\RR$.
Let $M_I$ denote the indecomposable module with support $I$.
We say that $I$ is \textbf{admissible} if $M_I$ is representable.
It is straightforward to see that $I$ is admissible if and only if the following hold.
\begin{enumerate}
\item $\inf I\in I$ if and only if it is blue, and
\item $\sup I\in I$ if and only if it is red.
\end{enumerate}
By Definition~\ref{def:representable}, neither endpoint of $I$ can be a source.
When $I=[s,s]$ is a sink, $\widehat{I}=[s_-,s_+]$.
We use notation to state this concisely: For any $a<b\in \widehat Q$, let $\widehat I(a,b)$ be the unique admissible interval in $\RR$ with endpoints $a,b$. Thus $a\in \widehat I(a,b)$if and only if$a$ is blue and $b\in \widehat I(a,b)$ if and only if $b$ is red. (Recall that every element of $\widehat Q$ is colored red or blue.)

Recall that for each $\sigma\in\mathcal{S}_\fpc(Q)$, the set of $\sigma$-semistable modules form a maximally $\Npi$-compatible set (Theorem~\ref{thm:four point equivalent to cluster}).

\subsection{Continuous tilting on modules}\label{sec:continuous tilting on modules}

\begin{lemma}\label{lem: continuous tilting preserves compatibility}{~}
\begin{itemize}
    \item [(a)] Continuous tilting gives a bijection between admissible intervals $I=\widehat I(a,b)$ for $Q$ and admissible intervals $I'$ for $Q'$ given by $I'=\widehat I'(\overline{\mathfrak t}(a),\overline{\mathfrak t}(b))$ if $\overline{\mathfrak t}(a)<\overline{\mathfrak t}(b)$ in $\widehat Q'$ and $I'=\widehat I'(\overline{\mathfrak t}(b),\overline{\mathfrak t}(a))$ if $\overline{\mathfrak t}(b)<\overline{\mathfrak t}(a)$.
    \item [(b)] Furthermore, $M_I,M_J$ are $\Npi$-compatible for $Q$ if and only if $M_{I'},M_{J'}$ are $\Npi$-compatible for $Q'$.
\end{itemize}
\end{lemma}

For each admissible interval $I$ for $Q$, denote by $\phi(M_I)$ the module $M_{I'}$, where $I'$ is the admissible interval of $Q'$ obtained from $I$ by continuous tilting.

Lemma \ref{lem: continuous tilting preserves compatibility} immediately implies the following.

\begin{theorem}\label{thm: continuous tilting gives bijection on max compatible sets}
Continuous tilting gives a bijection $\Phi$ between maximal compatible sets of representable indecomposable modules over $Q$ and those of $Q'$.
Furthermore if $\mu:T\to T'$ is a mutation then so is $\Phi(\mu):\Phi T\to \Phi T'$ given by $\phi(M_I)\mapsto \phi(\mu(M_I))$.
\end{theorem}

\begin{proof}[Proof of Lemma \ref{lem: continuous tilting preserves compatibility}]
(a) Since $\overline{\mathfrak t}:\widehat Q\to \widehat Q'$ is a bijection and $\widehat I(a,b)$ is admissible by notation, we get a bijection with admissible $Q'$ intervals by definition.

(b) Suppose that $I=\widehat I(a,b)$ and $J=\widehat I(c,d)$ with $a\le c$ by symmetry. We use Lemma \ref{lem:Npi compatiblity as colors} to check $\Npi$-compatibility.
For this proof, we say ``$I$ and $J$ are compatible'' to mean ``$M_I$ and $M_J$ are $\Npi$-compatible''.
\begin{enumerate}
\item If $a,b,c,d$ are not distinct then $\overline{\mathfrak t}(a), \overline{\mathfrak t} (b), \overline{\mathfrak t}(c), \overline{\mathfrak t}(d)$ are also not distinct. So, $I,J$ are compatible for $Q$ and $I',J'$ are compatible for $Q'$ in this case. So, suppose $S=\{a,b,c,d\}$ has size $|S|=4$.
\item If $S\cap K=\emptyset$ then $I,J\subset Z$. So, $I'=I$ and $J'=J$ are compatible for $Q'$ if and only if $I,J$ are compatible for $Q$. 
\item If $|S\cap K|=1$ then $S\cap K=\{a\}$. Then $\overline{\mathfrak t}$ does not change the order of $a,b,c,d$ and does not change the colors of $b,c,d$. So, $I,J$ are compatible for $Q$ if and only if $I',J'$ are compatible for $Q'$. 
\item If $|S\cap K|=2$ there are three cases: (a) $a<b<c<d$, (b) $a<c<b<d$ or (c) $a<c<d<b$. If $I,J$ are in case (a) then so are $I',J'$ and both pairs are compatible. If $I,J$ are in case (b) then $I',J'$ are in case (c) and vise versa. Since the colors of $a,c$ change in both cases (from red to blue), $I,J$ are compatible for $Q$ if and only if $I',J'$ are compatible for $Q'$.
\item If $|S\cap K|=3$ there are the same three cases as in case (4). If $I,J$ are in case (a), then $I',J'$ are in case (c) and vise-versa. Since the middle two vertices are the same color, both pairs are compatible. If $I,J$ are in case (b) then so are $I',J'$ and both pairs are not compatible.
\item If $S\subset K$ then $a,b,c,d$ reverse order and all become blue. So, $I,J$ are compatible if and only if they are in cases (a) or (c) and $I',J'$ are in the same case and are also compatible.
\end{enumerate}
In all cases, $I,J$ are compatible for $Q$ if and only if $I',J'$ are compatible for $Q'$.
\end{proof}

We can relate continuous tilting to cluster theories, introduced by the authors and Todorov in \cite{IRT22b}.

\begin{definition}\label{def:cluster theory}
    Let $\mathcal{C}$ be an additive, $\Bbbk$-linear, Krull--Remak--Schmidt, skeletally small category and let $\mathbf{P}$ be a pairwise compatibility condition on the isomorphism classes of indecomposable objects in $\mathcal{C}$.
    Suppose that for any maximally $\mathbf{P}$-compatible set $T$ and $X\in T$ there exists at most one $Y\notin T$ such that $(T\setminus\{X\})\cup\{Y\}$ is $\mathbf{P}$-compatible.
    
    Then we call maximally $\mathbf{P}$-compatible sets \textdef{$\mathbf{P}$-clusters}.
    We call bijections $\mu:T\to (T\setminus\{X\})\cup\{Y\}$  of $\mathbf{P}$-clusters \textdef{$\mathbf{P}$-mutations}.
    We call the groupoid whose objects are $\mathbf{P}$-clusters and whose morphisms are $\mathbf{P}$-mutations (and identity functions) the \textdef{$\mathbf{P}$-cluster theory of $\mathcal{C}$.}
    We denote this groupoid by $\mathscr{T}_{\mathbf{P}}(\mathcal{C})$ and denote the inclusion functor into the category of sets and functions by $I_{\mathcal{C},\mathbf{P}}:\mathscr{T}_{\mathbf{P}}(\mathcal{C})\to\mathcal{S}et$.
    We say $\mathbf{P}$ \textdef{induces} the $\mathbf{P}$-cluster theory of $\mathcal{C}$.
\end{definition}

The isomorphism of cluster theories was introduced by the second author in \cite{R22}.
\begin{definition}\label{def:isomorphism of cluster theories}
    An \textdef{isomorphism of cluster theories} is a pair $(F,\eta)$ with source $\mathscr{T}_{\mathbf{P}}(\mathcal{C})$ and target $\mathscr{T}_{\mathbf{Q}}(\mathcal{D})$.
    The $F$ is a functor $F:\mathscr{T}_{\mathbf{P}}(\mathcal{C})\to\mathscr{T}_{\mathbf{Q}}(\mathcal{D})$ such that $F$ induces a bijection on objects and morphisms.
    The $\eta$ is a natural transformation $\eta: I_{\mathcal{C},\mathbf{P}}\to I_{\mathcal{D},\mathbf{Q}}\circ F$ such that each component morphism $\eta_T:T\to F(T)$ is a bijection.
\end{definition}

We see that, for any continuous quiver $Q$ of type $\mathbb{A}$, the pairwise compatibility condition $\Npi$ induces the cluster theory $\mathscr{T}_{\mathbf{\Npi}}(\modr(Q))$.
The following corollary follows immediately from Theorem~\ref{thm: continuous tilting gives bijection on max compatible sets}.

\begin{corollary}[to Theorem~\ref{thm: continuous tilting gives bijection on max compatible sets}]
    For any pair of continuous quivers $Q$ and $Q'$ of type $\mathbb{A}$ with finitely many sinks and sources, there is an isomorphism of cluster theories $\mathscr{T}_{\mathbf{\Npi}}(\modr(Q))\to\mathscr{T}_{\mathbf{\Npi}}(\modr(Q'))$.
\end{corollary}

\subsection{Continuous tilting of stability conditions}\label{sec:continuous tilting of stability conditions}

Given a stability condition $\sigma$ for $Q$, we obtain a stability condition $\sigma'$ for $Q'$ having the property that the $\sigma'$-semistable modules are related to the $\sigma$-semistable modules for $Q$ by continuous tilting (the bijection $\Phi$ of Theorem \ref{thm: continuous tilting gives bijection on max compatible sets}). Later we will see that these stability conditions give the same measured lamination on the Poincare disk.

We continue with the notation from sections \ref{sec 3.1} and \ref{sec:continuous tilting on modules} above. If the stability condition $\sigma$ on $Q$ is given by the red-blue pair $(R,B)$, the tilted stability condition $\sigma'$ on $Q'$ will be given by $(R',B')$ given as follows.
\begin{enumerate}
	\item The pair $(R',B')$ will be the same as $(R,B)$ on $[s,\infty)$.
	\item On $K'=(-\infty,s_-]\subseteq \widehat Q'$, the new red function $R'$ will be constantly equal to $R_-(s)$.
	\item On $K'=(-\infty,s_-]$, the new blue function $B'$ can be given by ``flipping'' $R$ horizontally and flipping each ``island'' vertically, in either order.
\end{enumerate}

\begin{notation}\label{note:sub minus and sub plus}
    Let $F$ be a useful function.
    By $F_-(a)$ we denote $\lim_{x\to a^-} F(a)$, for any $a\in (-\infty, +\infty]$.
    By $F_+(a)$ we denote $\lim_{x\to a^+} F(a)$, for any $a\in [-\infty, +\infty)$.
\end{notation}

\begin{definition}
A (red) \textbf{island} in $K=[-\infty,s)\subseteq \widehat Q$ is an open interval $(x,y)$ in $K$ which is either:
\begin{enumerate}
	\item $(x,r)$ where $x<r$ so that $R(x)\ge R_-(s)$ and $R(z)<R_-(s)$ for all $x<z<s$ or
	\item $(x,y)$ where $x<y<s$, $R(x)\ge R(y)\ge R_-(s)$, $R(z)<R(y)$ for all $x<z<y$ and $R(w)\le R(y)$ for all $y<w<s$.
\end{enumerate}
\end{definition}

\begin{lemma}\label{lem: when z is in an island} 
$z\in (-\infty,s)$ is in the interior of some island in $K$ if and only if there exists $y\in (z,s)$ so that $R(z)<R(y)$.
\end{lemma}

\begin{proof}
$(\Rightarrow)$ If $z$ lies in the interior of an island $(x,y)$ there are two cases. (1) For $y<s$, $R(z)<R(y)$.
(2) For $y=s$, $R(z)<R_-(s)$. But $R_-(s)$ is a limit, so there is a $y<s$ arbitrarily close to $s$ so that $R(z)<R(y)$ and $z<y<s$.

$(\Leftarrow)$ Let $y\in (z,s)$ so that $R(z)<R(y)$. Let $r=sup\{R(y)\,:\, y\in (z,s)\}$. If $r=R(y)$ for some $y\in (z,s)$, let $y$ be minimal. (By the 4 point condition there are at most 2 such $y$.) Then $z$ lies in an island $(x,y)$ for some $x<z$.

If the maximum is not attained, there exists a sequence $y_i$ so that $R(y_i)$ converges to $r$. Then $y_i$ converges to some $w\in [z,s]$. If $w\in (z,s)$ then $R(z)=r$ and we are reduced to the previous case. Since $R(z)<r$, $w\neq z$. So, $w=s$ and $r=R_-(s)$. Then $z$ lies in an island $(x,s)$ for some $s<z$. ($x=\max\{w<z\,:\,R(w)\ge r\}$) In both cases, $z$ lies in an island as claimed.
\end{proof}

To define the new blue function $B'$, we need a function $H$ defined as follows.
\[
	H(z):=\begin{cases} R(y) & \text{if $z\in (x,y]$ for some island $(x,y)$ where $y<s$}\\
	R_-(s) & \text{if $z\in (x,s)$ and $(x,s)$ is an island}\\
   R(z) & \text{for all other $z\in[-\infty,s)$}
    \end{cases}
\]

\begin{remark}\label{rem: when z is in an island}
Note that $H(z)> R(z)$ if $z$ is in the interior of an island and $H(z)=R(z)$ otherwise.
\end{remark}

\begin{lemma}\label{lem: H is nonincreasing}
$H$ is a nonincreasing function, i.e., $H(x)\ge H(y)$ for all $x<y<s$. Also, $H(z)=H_-(z)=\lim_{y\to z-}H(y)$ for all $z<s$ and $H_-(s)=R_-(s)$.
\end{lemma}

\begin{remark}\label{rem: H>H+>R-(s)}
    Since $H$ is decreasing and converging to $R_-(s)$ we must have: $H(x)=H_-(x)\ge H_+(x)\ge R_-(s)$ for all $x<s$. 
\end{remark}

\begin{proof}
If $H(u)<H(z)$ for some $u<z<s$ then $R(u)\le H(u)<H(z)$. But $H(z)$ is equal to either $R(z), R_-(s)$ or $R(y)$ for some $y>z$. So, $R(u)<R(y)$ for some $y\in (u,s)$. By Lemma \ref{lem: when z is in an island}, $u$ lies in the interior of some island, say $(x,y)$ and, by definition of $H$, $H(u)=R(y)\ge R(w)$ for all $w\ge y$ and $H(u)=H(z)=H(y)$ for all $u\le z\le y$. Thus, $H$ is nonincreasing.

To see that $H(z)=H_-(z)$ suppose first that $z\in (x,y]$ for some island $(x,y)$. Then $H(z)=R(y)$ is constant on the interval $(x,y]$. So, $H(z)=H_-(z)=R(y)$. Similarly, $H(z)=H_-(z)$ if $z\in (x,s)$ and $(x,s)$ is an island. If $z$ is not in any island, $H(z)=R(z)$ and $R(z)=R_-(z)$ since, otherwise, $z$ would be on the right end of an island. And, $H_-(z)$ would be the limit of those $H(x)$ where $x<z$ and $H(x)=R(x)$. So, $H_-(z)=R_-(z)=H(z)$ as claimed.

Since $H(y)\ge R(y)$, we have: $H_-(s)=\lim_{y\to s-}H(y) \ge R_-(s)$. If $H_-(s)>R_-(s)$, say $H_-(s)=R_-(s)+c$ then there is a sequence $z_i\to s-$ so that $H(z_i)>R_-(s)+c/2$. For each $z_i$ there is $y_i\in [z_i,s)$ so that $H(z_i)=R(y_i)$. Then $R(y_i)>R_-(s)+c/2$ for all $i$ which is not possible since $y_i\to s-$. So, $H_-(s)=R_-(s)$.
\end{proof}

The monotonicity of $H$ implies that its variation $\var_HI$ on any interval $I$ is the difference of its limiting values on the endpoints. The formula is:
\[
	\var_H{(a,b)}=H_+(a)-H_-(b).
\]
Using $H=H_-$ and $H_+$ we can ``flip'' the islands up to get $\widetilde R$:
\[
	\widetilde R(z)=H(z)+H_+(z)-R(z).
\]
\begin{definition}\label{def:flipped functions}
The new blue function $B'$, shown in Figure \ref{Fig: flip R to B'}, is given on $K^\ast=(-\infty,s]$ by
\[
	B'(z)=\widetilde R({\mathfrak t}(z)).
\]
The new red function is constant on $K^\ast$ with value $R'(x)=R_-(s)$ for all $x\in K^\ast$. On the complement of $K^\ast$ in $\widehat Q'$, the red-blue pair $(R',B')$ is the same as before.
\end{definition}
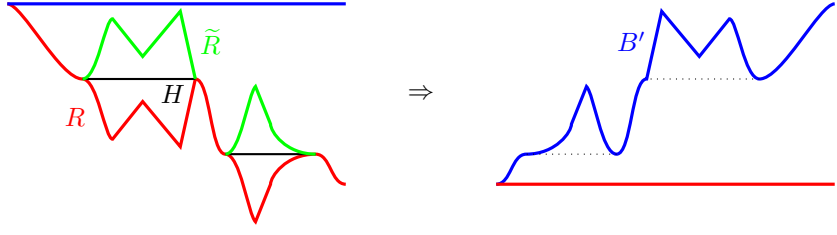
\begin{figure}[htbp]
\begin{center}
\begin{tikzpicture}%[scale=3]
%\draw[help lines=1,thick] (-1,-1) grid (10,4);
%\foreach \x in {-8,-6,...,8}\draw (\x,0) node{\x};\foreach \y in {-6,-4,...,8}\draw (0,\y) node{\y};
\coordinate (A) at (1,2); % (1,2)--(2.5,2) (2.9,1)--(4.1,1)
\coordinate (B) at (2.5,2);
\coordinate (C) at (2.9,1);
\coordinate (D) at (4.1,1);
\coordinate (E) at (4.5,0.6);
\coordinate (X) at (0,3); % (0,0.6)--(4.5,0.6)
\coordinate (Y) at (4.5,3);
\coordinate (Z) at (5.5,1.8);
\coordinate (R) at (.9,1.5);
\coordinate (G) at (2.7,2.5);
\coordinate (H) at (2.2,1.8);
\draw[red] (R) node {$R$};
\draw[green] (G) node {$\widetilde R$};
\draw[black] (H) node {$ H$};
\draw[very thick,red] (X)..controls (.2,3) and (.7,2)..(A)..controls (1.2,2) and (1.3,1.2)..(1.4,1.2)--(1.8,1.7)--(2.3,1.1)--(B)..controls (2.7,2) and (2.7,1)..(C)..controls (3.1,1) and (3.2,0.2)..(3.3,.1)--(3.5,.6)..controls (3.5,0.7) and (3.7,1)..(D)..controls (4.3,1) and (4.3,0.6)..(E);
\draw[thick,black] (A)--(B) (C)--(D);
\draw[very thick,blue] (X)--(Y);
\draw[very thick,green] (A)..controls (1.2,2) and (1.3,2.8)..(1.4,2.8)--(1.8,2.3)--(2.3,2.9)--(B)
(C)..controls (3.1,1) and (3.2,1.8)..(3.3,1.9)--(3.5,1.4)..controls (3.5,01.3) and (3.7,1)..(D);
\draw (Z) node{$\Rightarrow$};
\begin{scope}[xshift=11cm, xscale=-1]
\draw[very thick,blue] (0,3)..controls (.2,3) and (.7,2)..(1,2)..controls (1.2,2) and (1.3,2.8)..(1.4,2.8)--(1.8,2.3)--(2.3,2.9)--(2.5,2)..controls (2.7,2) and (2.7,1)..(2.9,1)..controls (3.1,1) and (3.2,1.8)..(3.3,1.9)--(3.5,1.4)..controls (3.5,01.3) and (3.7,1)..(4.1,1)..controls (4.3,1) and (4.3,0.6)..(4.5,0.6)
(2.7,2.5) node{$B'$};
\draw[dotted] (1,2)--(2.5,2) (2.9,1)--(4.1,1);
\draw[very thick, red] (0,0.6)--(4.5,0.6);
\end{scope}
\end{tikzpicture}
\caption{The function $R$ is in red. $H$, black, flattens the islands of $R$. When the islands are flipped up, we get $\widetilde R$ in green. The horizontal mirror image of this is the new blue function $B'$ on the right. Figures \ref{fig:integral 1}, \ref{fig:integral 2} give another example.}
\label{Fig: flip R to B'}
\end{center}
\end{figure}

We will now show $B'$ is a useful function with the same variation on $(-\infty,s]$ as $R$ has on $[-\infty,s)$. More precisely:

\begin{lemma}\label{lem: B' has same variation as R}
The variation of $R$ on any open interval $(a,b)\subset [-\infty,s)$ is equal to the variation of $B'$ on $({\mathfrak t}(b), {\mathfrak t}(a))$.
\end{lemma}

%We use the notation $V_f(I)$ for the variation of $f$ on an interval $I$.
\begin{proof}
Since $B'$ is obtained from $\widetilde R$ by reversing the order of the first coordinate, we have $\var_{B'}({\mathfrak t}(b), {\mathfrak t}(a))=\var_{\widetilde R}(a,b)$. Thus, it suffices to show that $\var_{\widetilde R}(a,b)=\var_{R}(a,b)$.

First, we do the case when $(a,b)$ is an island. Then $H(z)=H_+(z)=R(b)>R(z)$ are constant for all $z\in (a,b)$. So, $\widetilde R=H+H_+-R$ has the same variation as $R$ on $(a,b)$. 

Write $R=H+(R-H)$. Then we claim that 
\[
	\var_R(a,b)=\var_H(a,b)+\var_{R-H}(a,b).
\]
To see this take any sequence $a<x_0<x_1<\cdots<x_n<b$. Then the sum
\[
	\sum_{i=1}^n |R(x_i)-R(x_{i-1})|
\]
can be broken up into parts. Let $A_1,\cdots,A_m$ be the sequence of disjoint subsets of $S=\{x_0,\cdots,x_n\}$ so that $A_j$ is the intersection of $S$ with some island $(a_j,b_j)$. We may assume that $a_j$ for $1<j\le m$ and $b_j$ for $1\le j<m$ are in the set $S$ since they lie in the interval $(a,b)$. For $1<j\le m$, if $x_i$ is the smallest element of $A_j$, then $x_{i-1}=a_j$ and the $x_i,x_{i-1}$ term in the approximation of $\var_H(a,b)+\var_{H-R}(a,b)$ is
\[
    |H(a_j)-H(x_i)|+|(R-H)(a_j)-(R-H)(x_i)|=|R(a_j)-H(x_i)|+|H(x_i)-R(x_i)|
\]
since $H(a_j)=R(a_j)$. This sum is equal to
$
	|R(a_j)-R(x_i)|%=|R(a_j)-h|+|h-R(x_i)|
$, 
the corresponding term in the approximation of $\var_R(a,b)$, since $R(a_j)\ge H(x_i)>R(x_i)$. Similarly, $H(b_j)=R(b_j)$ by definition and $R(b_j)=H(x_k)>R(x_k)$ for any $x_k\in A_j$. So,
\[
    |H(b_j)-H(x_k)|+|(R-H)(b_j)-(R-H)(x_k)|=|R(b_j)-R(x_k)|.
\]

If $x_i,x_{i+1}$ both lie in $A_j$ then $H(x_i)=H(x_{i+1})$. So, 
\[
    |R(x_i)-R(x_{i+1})|=|(R-H)(x_i)-(R-H)(x_{i+1})|+|H(x_i)-H(x_{i+1})|.
\]
This equation also holds if $x_i,x_{i+1}$ do not lie in any $A_j$ since, in that case, $R=H$ at both $x_i$ and $x_{i+1}$. Thus every term in the sum approximating $\var_R(a,b)$ is equal to the sum of the corresponding terms for $\var_H(a,b)$ and $\var_{R-H}(a,b)$. Taking supremum we get the equation $\var_R(a,b)=\var_H(a,b)+\var_{R-H}(a,b)$ as claimed.

A similar calculation shows that 
\[
	\var_{\widetilde R}(a,b)=\var_{H_+}(a,b)+\var_{\widetilde R-H_+}(a,b).
\]
But this is equal to $\var_R(a,b)=\var_H(a,b)+\var_{R-H}(a,b)$ since $H-R=\widetilde R-H_+$ by definition of $\widetilde R$ and $\var_H(a,b)=H_+(a)-H_-(b)=\var_{H_+}(a,b)$. Thus $\var_R(a,b)=\var_{\widetilde R}(a,b)=\var_{B'}(\mathfrak t(b),\mathfrak r(a))$.
\end{proof}

For $x_0$ in the interior of the domain of $f$ let
\[
\var_f(x_0):=\lim_{\delta\to 0} \var_f(x_0-\delta,x_0+\delta)=\lim_{\delta\to 0} \var_f[x_0-\delta,x_0+\delta]
\]
We call this the \textbf{local variation} of $f$ at $x_0$. If $x_0\in (a,b)$ this is equivalent to:
\[
	\var_f(x_0)=\var_f(a,b)-\var_f(a,x_0)-\var_f(x_0,b)
\]
since this is the limit of $\var_f(a,b)-\var_f(a,x_0-\delta]-\var_f[x_0+\delta,b)=\var_f[x_0-\delta,x_0+\delta]$.

To show that $B'$ is a useful function we need the following lemma.
\begin{lemma}\label{lem: continuous at x iff local variation is 0}
A real valued function $f$ of bounded variation defined in a neighborhood of $x_0$ is continuous at $x_0$ if and only if its local variation, $\var_f(x_0)=0$. In particular, $R$ is continuous at $x\in K$ if and only if $B'$ is continuous at $\mathfrak t(x)\in K^\ast$.
\end{lemma}

\begin{proof}
Suppose that $\var_f(x_0)=0$. Then, for any $\varepsilon>0$ there is a $\delta>0$ so that 
\[
	\var_f(x_0-\delta,x_0+\delta)<\varepsilon.
\]
Then $|f(x)-f(x_0)|<\varepsilon$ for all $x\in (x_0-\delta,x_0+\delta)$. So, $f$ is continuous at $x_0$.

Conversely, suppose $f$ is continuous at $x_0$. Then, for any $\varepsilon>0$ there is a $\delta>0$ so that $|f(x)-f(x_0)|<\varepsilon$ for $|x-x_0|<\delta$. Let $V=\var_f[x_0,x_0+\delta)$. By definition of variation there exist $x_0<x_1<\cdots<x_n<x_0+\delta$ so that
\[
	\sum_{i=1}^n |f(x_i)-f(x_{i-1})|>V-\varepsilon.
\]
Since $|f(x_1)-f(x_0)|<\varepsilon$ this implies $\sum_{i=2}^n |f(x_i)-f(x_{i-1})|>V-2\varepsilon$. So, $\var_f[x_0,x_1)<2\varepsilon$. Similarly, there exists $x_{-1}<x_0$ so that $\var_f(x_{-1},x_0)<2\varepsilon$. So, $\var_f(x_{-1},x_1)<4\varepsilon$ which is arbitrarily small.
\end{proof}

For a useful function $F$, recall that $u^-_a=F(a)-\lim_{x\to a-}F(x)$ and $u_a^+=\lim_{x\to a+}F(x)-F(a)$ (Proposition~\ref{prop:useful functions are useful}).
%\textcolor{job}{Rephrased this prop.}
\begin{proposition}\label{prop: local variation at a is u+(a) + u-(a)}
%Given a function $F$ with left and right limits, let $u^-_a=F(a)-\lim_{x\to a-}F(x)$ and $u_a^+=\lim_{x\to a+}F(x)-F(a)$.
Let $F$ be a useful function.
Then, the local variation of $F$ at any point $a$ is
\[
	\var_F(a)=|u_a^-|+|u_a^+|.
\]
\end{proposition}

\begin{proof}
It follows from the triangle inequality that the variation of $f+g$ on any open interval is bounded above and below by the sum and differences of the variations of $f,g$ on that interval. This holds for local variations as well:
\[
	|\var_g(x)-\var_f(x)|\le \var_{f+g}(x)\le \var_f(x)+\var_g(x)
\]
Let $g_x=u_x^-\Delta_x^-+u_x^+\Delta_x^+$. Then
\[
	\var_F(x)=\var_{g_x}(x)=|u_x^-|+|u_x^+|
\]
since $F-g_x$ is continuous at $x$ and thus, by Lemma \ref{lem: continuous at x iff local variation is 0}, has $\var_{F-g_x}(x)=0$.
\end{proof}

We can say slightly more for the functions $R$ and $B'$. (See also Figure \ref{Fig6: flipping spikes}.)

\begin{lemma}\label{lem: B'=min B'}
For any $a\in K=[-\infty,s)$ let $b=\mathfrak t(a)\in K^\ast$. Then $v_b^-=B'(b)-B'_-(b)\le0$ and $v_b^+=B'_+(b)-B'(b)\ge0$. In particular, $B'(b)=min\,B'(b)$. 
\end{lemma}

\begin{proof} Since $B'$ is the mirror image of $\widetilde R$, $v_b^-$ and $v_b^+$ for $B'$ are equal to $-v_a^+,-v_a^-$ for $\widetilde R$, respectively, where $v_a^-=\widetilde R(a)-\widetilde R_-(a)$ and $v_a^+=\widetilde R_+(a)-\widetilde R(a)$. Thus it suffices to show that $v_a^-\le0$ and $v_a^+\ge0$.

We have $u_a^-=R(a)-R_-(a)\ge0$. Also, $\widetilde R_-=(H+H_+-R)_-=2H-R_-$. So,
\[
\begin{split}
	v_a^-&=(\widetilde R(a)-H_+(a))-(\widetilde R_-(a)-H_+(a))\\
	&= (H(a)-R(a))+R_-(a)-2H(a)+H_+(a)\\
	&=-u_a^--(H(a)-H_+(a))\le0
	\end{split}
\]

Similarly, we have $u_a^+=R_+(a)-R(a)\le0$ and $\widetilde R_+(a)=2H_+(a)-R_+(a)$. So,
\[\begin{split}
	v_a^+&=(\widetilde R_+(a)-H_+(a))-(\widetilde R(a)-H_+(a))\\
	&= (H_+(a)-R_+(a))-(H(a)-R(a))\\
	&= (H_+(a)-H(a))-u_a^+
\end{split}
\]
To show that $v_a^+\ge0$, there are two cases. If $a$ lies in an island $(x,y)$, then $H_+(a)=H(a)=R(y)$ (or $R_-(s)$ if $y=s$) and $v_a^+=-u_a^+\ge0$. If $a$ does not lie in an island then $H(a)=R(a)$ and $H_+(a)\ge R_+(a)$. So, $v_a^+\ge0$.
\end{proof}

\begin{figure}[htbp]
\begin{center}
\begin{tikzpicture}%[scale=3]
%\draw[help lines=1,thick] (-1,-1) grid (10,4);
%\draw[help lines=2,thick,blue] (-1,-1) grid (10,4);
%\foreach \x in {-8,-6,...,8}\draw (\x,0) node{\x};\foreach \y in {-6,-4,...,8}\draw (0,\y) node{\y};
%
\begin{scope}
\coordinate (A) at (0,3); % (1,2)--(2.5,2) (2.9,1)--(4.1,1)
\coordinate (B) at (1,2.5);
\coordinate (B1) at (1,2.5);
\coordinate (B2) at (1,2);
\coordinate (B3) at (1,1);
\coordinate (C) at (2.9,1);
\coordinate (C2) at (2,2);
\coordinate (Z2) at (1.5,2);
\coordinate (C3) at (2,1.5);
\coordinate (C4) at (2,1);
\coordinate (D) at (3,.5);
\coordinate (X) at (1,.85);
\coordinate (Y) at (2,.85);
\coordinate (Z) at (.1,2.5);
\draw[very thick,red] (A)--(B1)--(B3)--(C3)--(C2)--(C4)--(D);
\draw[red,fill] (B1) circle[radius=1mm];
\draw[red,fill] (C2) circle[radius=1mm];
\draw (X) node[below]{$a$};
\draw (Y) node[below]{$b$};
\draw (Z2) node[above]{$H$};
\draw[thick,dashed] (B2)--(C2);
\draw[red] (Z) node{$R$};
\end{scope}
\begin{scope}[xshift=3.5cm]
\coordinate (A) at (0,3); % (1,2)--(2.5,2) (2.9,1)--(4.1,1)
\coordinate (B) at (1,2.5);
\coordinate (B1) at (1,2.5);
\coordinate (B0) at (1,3);
\coordinate (B2) at (1,2);
\coordinate (B3) at (1,1);
\coordinate (C) at (2.9,1);
\coordinate (C1) at (2,2.5);
\coordinate (C2) at (2,2);
\coordinate (C3) at (2,1.5);
\coordinate (C4) at (2,1);
\coordinate (D) at (3,.5);
\coordinate (Z) at (.1,2.5);
\coordinate (Z2) at (1.5,2);
\coordinate (X) at (1,.85);
\coordinate (Y) at (2,.85);
\draw (X) node[below]{$a$};
\draw (Y) node[below]{$b$};
\draw (Z2) node[below]{$H_+$};
\draw[fill] (B2) circle[radius=1mm];
\draw[fill] (C4) circle[radius=1mm];
\draw[thick,dashed] (B2)--(C2);
\draw[very thick] (A)--(B1)--(B2)--(B0)--(C1)--(C4)--(D);
\draw (Z) node{$\widetilde R$};
\end{scope}
\begin{scope}[xshift=10.5cm,xscale=-1]
\coordinate (A) at (0,3); % (1,2)--(2.5,2) (2.9,1)--(4.1,1)
\coordinate (B) at (1,2.5);
\coordinate (B1) at (1,2.5);
\coordinate (B0) at (1,3);
\coordinate (B2) at (1,2);
\coordinate (B3) at (1,1);
\coordinate (C) at (2.9,1);
\coordinate (C1) at (2,2.5);
\coordinate (C2) at (2,2);
\coordinate (C3) at (2,1.5);
\coordinate (C4) at (2,1);
\coordinate (D) at (3,.5);
\coordinate (Z) at (2.5,1);
\coordinate (X) at (1,.85);
\coordinate (Y) at (2,.85);
\draw (X) node[below]{$\mathfrak t(a)$};
\draw (Y) node[below]{$\mathfrak t(b)$};
\draw[fill,blue] (B2) circle[radius=1mm];
\draw[fill,blue] (C4) circle[radius=1mm];
\draw[dashed,black] (B2)--(C2);
\draw[very thick,blue] (A)--(B1)--(B2)--(B0)--(C1)--(C4)--(D);
\draw[blue] (Z) node[left]{$B'$};
\end{scope}
\end{tikzpicture}
\caption{This red function $R$ has a spike on the right end $b$ of an island $(a,b)$ and a discontinuity at the left end $a$. When the island is flipped, we get a downward spike at $a$ and a discontinuity at $b$. The function $R$ is the maximum and the tilted functions $\widetilde R$ and $B'$ are minimums on vertical lines.}
\label{Fig6: flipping spikes}
\end{center}
\end{figure}
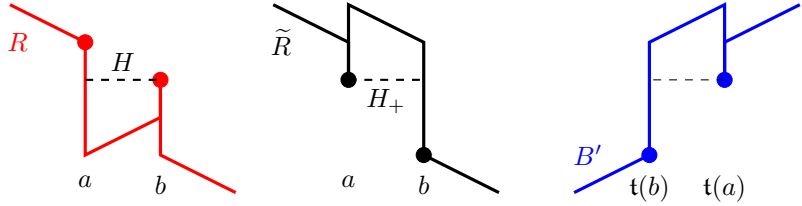

\begin{theorem}\label{thm:stability tilting}
    The new pair $(R',B')$ is a red-blue pair for the quiver $Q'$ and the $\sigma'$-semistable $Q'$ modules given by this pair are the continuous tilts of the $\sigma$-semistable $Q$-modules given by the original pair $(R,B)$.
\end{theorem}

\begin{proof}
Lemmas \ref{lem: B' has same variation as R} implies that $R$ and $B'$ have the same local variation at the corresponding points $x$ and $\mathfrak t(x)$. In particular, $R$ and $B'$ have discontinuities at corresponding points by Lemma \ref{lem: continuous at x iff local variation is 0} and the $B'(a)=min\,B'(a)$ by Lemma \ref{lem: B'=min B'}. 

The new red function $R'$ is constantly equal to $R_-(s)$ on $K^\ast$ and equal to the old function $R$ on the complement $Z$. So, $B'(x)\ge R'(x)$ and they have the same limit as $x\to -\infty$ by Remark \ref{rem: H>H+>R-(s)}. Thus $(R',B')$ form a red-blue pair for $Q'$.

Let $\sigma,\sigma'$ be the stability conditions on $Q,Q'$ given by the red-blue pairs $(R,B)$ and $(R',B')$, resp. It remains to show that the admissible interval $I=\widehat I(a,b)$ is $\sigma$-semistable for $Q$ if and only if the corresponding interval $I'$ is $\sigma'$-semistable for $Q'$ where $I'=\widehat I'(\overline{\mathfrak t}(a),\overline{\mathfrak t}(b))$ if $\overline{\mathfrak t}(a)<\overline{\mathfrak t}(b)$ in $\widehat Q'$ and $I'=\widehat I'(\overline{\mathfrak t}(b),\overline{\mathfrak t}(a))$ if $\overline{\mathfrak t}(b)<\overline{\mathfrak t}(a)$.

Consider $a<b$ in $\overline \RR$. There are three cases.
\begin{enumerate}
\item $a=\overline{\mathfrak t}(a)$ and $b=\overline{\mathfrak t}(b)$ both lie in $Z$.
\item $-\infty \le a<b<s$ ($a,b\in K$) and $-\infty <\overline{\mathfrak t}( b)<\overline{\mathfrak t}( a)\le s$ ($\overline{\mathfrak t}( a),\overline{\mathfrak t}( b)\in K^\ast$).
\item $a\in K$, $\overline{\mathfrak t}(a)\in K^\ast$ and $b=\overline{\mathfrak t}(b)\in Z$.
\end{enumerate}

In Case (1), the stability conditions $\sigma,\sigma'$ given by the red and blue functions are the same on $Z$. So, $\widehat I(a,b)$ is $\sigma$-semistable if and only if $\widehat{I'}(a,b)=\widehat{I'}(\overline{\mathfrak t}(a),\overline{\mathfrak t}(b))$ is $\sigma'$-semistable for $Q'$.

In Case (2), we claim that $\widehat I(a,b)$ at height $h$ is $\sigma$-semistable if and only if $\widehat I'(\overline{\mathfrak t}(b),\overline{\mathfrak t}(a))$ is $\sigma'$-semistable at height $h'$ where $h'=2H(b)-h$. 

An example can be visualized in Figure \ref{Fig6: flipping spikes} by drawing horizontal lines at height $h<H$ and $h'>H_+$ under the line $H$ on the left and over $H_+$ on the right.

To see this in general, note that if $\widehat I(a,b)$ at height $h$ is $\sigma$-semistable then, for all $z\in (a,b)$, $R(z)\le h$ (with equality holding for at most one value of $z$, call it $z=c$) and $R(a),R(b)\ge h$. Then for each $z\in [a,b]$, $H(z)\ge h$. So, for each $z\in (a,b)$, $z\neq c$, we have $H(z)>R(z)$. By Remark \ref{rem: when z is in an island}, $z$ lies in the interior of an island for $R$. But $\widetilde R(z)-H_+(z)=H(z)-R(z)>0$. So, the same values of $z$ lie in islands for $\widetilde R$ and $\widetilde R(z)-h'=h-R(z)\ge0$. Also, $\widetilde R(a),\widetilde R(b)\le h'$ since:
\[\begin{split}
	h'-\widetilde R(b)&=2H(b)-h-H(b)-H_+(b)+R(b)\\
	&= (H(b)-H_+(b))+(R(b)-h)\ge0
\end{split}
\]
and, since $H_+(a)=H(b)$ and $H(a)=$ either $H(b)$ or $R(a)$,
\[\begin{split}
	h'-\widetilde R(a)&=2H(b)-h-H(a)-H_+(a)+R(a)\\
	&= R(a)-h+H(b)-H(a)\\
	&\text{either } =R(a)-h \ge0\\
	& \text{or } =H(b)-h\ge0
\end{split}
\]
Therefore, $[a,b]\times h'$ is a chord for $\widetilde R$, making its mirror image $[\overline{\mathfrak t}(b), \overline{\mathfrak t}(a)]\times h'$ a chord for $B'$ and thus $\widehat I'(\overline{\mathfrak t}(b), \overline{\mathfrak t}(a))$ is $\sigma'$-semistable for $Q'$ at height $h'$. An analogous argument shows the converse. So, $\widehat I(a,b)$ at height $h$ is $\sigma$-semistable for $Q$ if and only if $\widehat I'(\overline{\mathfrak t}(b),\overline{\mathfrak t}(a))$ is $\sigma'$-semistable at height $h'$ for $Q'$.

In Case (3), we change notation to match Figure \ref{Fig6: flipping spikes}. Suppose we have $b\in K$, $\overline{\mathfrak t}(b)\in K^\ast$ and $c=\overline{\mathfrak t}(c)\in Z$. We claim that $\widehat I(b,c)$ is $\sigma$-semistable at height $h$ if and only if $\widehat I'(\overline{\mathfrak t}(b),\overline{\mathfrak t}(c))$ is $\sigma'$-semistable at the same height $h$.

In Figure \ref{Fig6: flipping spikes}, the chord $[b,c]\times h$ would be a horizontal line starting at any point on the vertical red line at $b$ and going to the right. For $\widetilde R$, we have $H(b)\ge h\ge H_+(b)$, so a horizontal line at height $h$ starting anywhere on the vertical segment $b\times [H_+(b),H(b)]$ could go left without hitting the function $\widetilde R$ except at height $h=H_+(a)=H(b)$ where it would touch the function at $(a,H_+(a))$ then continue. For $B'$, the horizontal line starting at $(\overline{\mathfrak t}(b),h)$ would go right, possibly touch the curve at $\overline{\mathfrak t}(a)$ and continue to the point $(c,h)$.

The situation in general is very similar. $\widehat I(b,c)$ is $\sigma$-semistable at height $h$ for some $c\in Z$ if and only if $H_+(b)\le h\le H(b)=R(b)$. Since $H_+(b)$ is the supremum of $R(x)$ for all $b<x<s$, this is equivalent to saying the horizontal line at $h$ does not touch the curve $R$ except possibly at one point (not more by the four point condition). If $h=H(b)$, this horizontal line might continue to the left of $(b,h)$ an hit at most one point $(a,h)$ on the curve $R$. 

If $h<H(b)$ then the horizontal line at $(b,h)$ on $\widetilde R$, would go to the left and not hit anything since, for all $x<b$, we have $\widetilde R(x)\ge H_+(x)\ge H(b)>h$. So, the line from $(\overline{\mathfrak t}(b),h)$ to $(\overline{\mathfrak t}(c),h)$ would not hit $B'$. 

If $h=H(b)$, then, for all $x<b$, $\widetilde R(x)\ge H_+(x)\ge H(b)=h$. So, the line going left from $(b,h)=(b,H(b))$ would stay under $\widetilde R$ possibly touching it at most once, say at $(a,h)$. Then $(a,b)$ would be an island and we have the situation in Figure \ref{Fig6: flipping spikes}. By the four point condition we cannot have another point $a'$ with the same property since $(a,h),(b,h),(c,h)$ are already on a line. The horizontal line going right from $(\overline{\mathfrak t}(b),h)$ would touch the curve $B'$ at $(\overline{\mathfrak t}(a),h)$ and continue to $(\overline{\mathfrak t}(c),h)$.

So, $\widehat I(b,c)$ being $\sigma$-semistable at height $h$ implies that $\widehat I'(\overline{\mathfrak t}(b),\overline{\mathfrak t}(c))$ is $\sigma'$-semistable at the same height $h$. The converse is similar since going from $B'$ to $R$ is analogous (change $B'$ to $-B'$ and make it red). This concludes the proof in all cases.
\end{proof}

\section{Measured Laminations and Stability Conditions}\label{sec:measured laminations}
In this section we connect measured laminations of the hyperbolic plane to stability conditions for continuous quivers of type $\mathbb{A}$.
We first define measured laminations (Definition~\ref{def:measured lamination}) of the hyperbolic plane and prove some basic results we need in Section~\ref{sec:measured laminations:measured laminations}.
In Section~\ref{sec:the correspondence} we describe the correspondence that connects stability conditions to measured laminations.
In Section~\ref{sec:continuous cluster character} we present a candidate for continuous cluster characters.
In Section~\ref{sec:all maximal compatible sets come from a stability condition} we briefly describe how all maximally $\Npi$-compatible sets come from a stability condition.
In Section~\ref{sec:finite map using continuous tilting} we describe maps between cluster categories of type $\mathbb{A}_n$ that factor through our continuous tilting.
We also give an example for type $\mathbb{A}_4$.

\subsection{Measured Laminations}\label{sec:measured laminations:measured laminations}
We denote by $\hyper$ the Poincar\'e disk model of the hyperbolic plane and by $\partial\hyper$ the boundary of the disk such that $\partial\hyper$ is the unit circle in $\mathbb{C}$.
Recall a \textdef{lamination} of $\hyper$ is a maximal set of noncrossing geodesics and that a geodesic in $\hyper$ is uniquely determined by a distinct pair of points on $\partial\hyper$.

Let $L$ be a lamination of $\hyper$.
Choose two open interval subsets $A$ and $B$ of $\partial\hyper$, each of which may be all of $\partial\hyper$ or empty.
Let $O_{A,B}$ be the set of geodesics with one endpoint in $A$ and the other in $B$.
We call $O_{A,B}$ a \textdef{basic open subset} of $L$.
Notice that $O_{A,B}=O_{B,A}$.
The basic open sets define a topology on $L$.

\begin{definition}\label{def:measured lamination}
    Let $L$ be a lamination of $\hyper$ and $\Meas:L\to \RR_{\geq 0}$ a measure on $L$.
    We say $(L,\Meas)$ is a \textdef{measured lamination} if $0 < \Meas(O_{A,B}) < \infty$ for every $O_{A,B}\neq \emptyset$.
\end{definition}

Notice that we immediately see any measured lamination $(L,\Meas)$ has finite measure.
That is, $0<\Meas(L)<\infty$.

We now define some useful pieces of laminations.

\begin{definition}\label{def:discrete, fountain, and rainbow}
    Let $L$ be a lamination of $\hyper$.
    \begin{enumerate}
        \item\label{def:discrete, fountain, and rainbow:discrete}
        Let $\gamma\in L$ be a geodesic determined $a,b\in\partial\hyper$.
        We say $\gamma$ is a \textdef{discrete arc} if there exists non-intersecting open subsets $A\ni a$ and $B\ni b$ of $\partial\hyper$ such that $O_{A,B}=\{\gamma\}$.
        
        \item\label{def:discrete, fountain, and rainbow:fountain}
        Let $a\in\partial\hyper$.
        Let $A$ be some interval subset of $\partial\hyper$ with more than one element such that for every geodesic $\gamma\in L$ determined by some $a'\in A$ and $b\in\partial\hyper$, we have $b=a$.
        Then we define the set $K$ of geodesics determined by the pair $a,A$ to be called a \textdef{fountain}.
        We say $K$ is \textdef{maximal} if a fountain determined by $a,A'$, where $A'\supseteq A$, is precisely $K$.
        
        \item\label{def:discrete, fountain, and rainbow:rainbow}
        Let $A,B$ be interval subsets of $\partial\hyper$ whose intersection contains at most one point.
        Suppose that for every geodeisc $\gamma\in L$ determined by $a,b\in\partial\hyper$, we have $a\in A\setminus\partial A$ if and only if $b\in B\setminus\partial B$.
        If there is more than one such geodesic, we call the set $K$ of all such geodesics determined by $a,b$ with $a\in A$ and $b\in B$ a \textdef{rainbow}.
        We say $K$ is \textdef{maximal} if a rainbow determined by $A'\supseteq A$ and $B'\supseteq B$ is precisely $K$.
    \end{enumerate}
\end{definition}

From the definitions we have a result about discrete arcs, fountains, and rainbows.

\begin{proposition}\label{prop:discrete, fountain, rainbow are positive}
    Let $L$ be a lamination of $\hyper$ and let $K$ be a discrete geodesic, a fountain, or a rainbow.
    Then $\Meas(K)>0$.
\end{proposition}
\begin{proof}
    By definition, if $K=\{\gamma\}$ is a discrete arc then $K=O_{A,B}$ and so $\Meas(K)>0$.
    Additionally, if $K=L$ then $K=O_{\partial\hyper,\partial\hyper}$ and so $\Meas(K)>0$.
    So we will assume $K$ is either a fountain or a rainbow and $K\neq L$; in particular $K$ has more than one element.

    First suppose $K$ is a fountain determined by $a\in\partial\hyper$ and $A\subset\partial\hyper$.
    By definition $K$ has more than one element and so $A\setminus\partial A\neq\emptyset$.
    If $a\notin A$ then let $B\ni a$ be a small open ball around $a$ in $\partial\hyper$ such that $B\cap A=\emptyset$.
    Now consider $O_{A\setminus\partial A,B}$.
    We see $O_{A\setminus\partial A,B}\subset K$ and $\Meas(O_{A\setminus\partial A,B})>0$.
    If $a\in A$ then every geodesic determined by an $a'$ and $b$ with $a'\in A\setminus(\{a\}\cup\partial A)$ has $b=a$.
    Let $A'=A\setminus(\{a\}\cup\partial A)$ and let $B\ni a$ be an open ball such that $A\setminus\partial A\not\subset B$.
    Now we have $O_{A',B}\subset K$ and $\Meas(O_{A',B})>0$.
    Therefore $\Meas(K)>0$.

    Now suppose $K$ is a rainbow determined by $A$ and $B$.
    Again we know $K$ has more than one element so both $A\setminus\partial A$ and $B\setminus\partial B$ are nonempty.
    Take $A'=A\setminus \partial A$ and $B'=B\setminus\partial B$.
    Then $O_{A',B'}\subset K$ and $\Meas(O_{A',B'})>0$.
    Therefore, $\Meas(K)>0$.
\end{proof}

\subsection{The Correspondence}\label{sec:the correspondence}
In this section, we recall the connection between $\Npi$-clusters and (unmeasured) laminations of $\hyper$ for the straight descending orientation of a continuous quiver of type $\mathbb{A}$, from \cite{IT15}.
We then extend this connection to measured laminations and stability conditions stability conditions that satisfy the four point condition, obtaining a ``2-bijection'' (Theorem~\ref{thm:measured laminations and straight four point conditions}).
Then we further extend this ``2-bijection'' between measured laminations and stability conditions to all continuous quivers of type $\mathbb{A}$ with finitely many sinks and sources (Corollary~\ref{cor:measured laminations and straight four point conditions for all orientations}).
We conclude that section with an explicit statement that tilting a stability condition $\sigma\in\mathcal{S}_\fpc(Q)$ to a stability condition $\sigma'\in\mathcal{S}_\fpc(Q')$ yields the \emph{same} measured lamination, for continuous quivers $Q,Q'$ of type $\mathbb{A}$ (Theorem~\ref{thm:explicit statement about stability conditions, tilting, and measured laminations}).

\begin{theorem}[from \cite{IT15}]\label{thm:igusa todorov}
    There is a bijection $\Phi$ from maximally $\Npi$-compatible sets to laminations of $\hyper$.
    For each maximally $\Npi$-compatible set $T$ and corresponding lamination $\Phi(T)$, there is a bijection $\phi_T:T\to \Phi(T)$ that takes objects in $T$ to geodesics in $\Phi(T)$.
\end{theorem}

Before we proceed we introduce some notation to make some remaining definitions and proofs in this section more readable.
First, we fix an indexing on $\partial\hyper$ in the following way.
To each point $x\in\RR\cup\{-\infty\}$ we assign the point $e^{i\arctan(x)}$ in $\partial\hyper$.
We now refere to points in $\partial\hyper$ as points in $\RR\cup\{-\infty\}$.
\begin{notation}
    Let $(L,\Meas)$ be a measured lamination of $\hyper$.
    \begin{itemize}
        \item For each $\gamma\in L$ we denote by $a_\gamma$ and $b_\gamma$ the unique points in $\partial\hyper$ that determine $\gamma$ such that $a_\gamma<b_\gamma$ in $\RR\cup\{-\infty\}$.
        \item For each $x\in\partial\hyper$ such that $x\neq-\infty$,
        \begin{align*}
            \frac{L}{x} :=& \{\gamma\in L \mid \gamma_a < x < \gamma_b\} \\
            L \cdot x :=& \{\gamma\in L \mid \gamma_b = x\} \\
            x \cdot L :=& \{\gamma\in L \mid \gamma_a = x\}.
        \end{align*}
        \item For $-\infty$,
        \begin{align*}
            \frac{L}{-\infty} :=& \emptyset \\
            L \cdot (-\infty) :=& \emptyset \\
            (-\infty) \cdot L :=& \{\gamma\in L\mid \gamma_a = -\infty\}.
        \end{align*}
        \item Finally, for some interval $I\subset \RR$,
        \begin{align*}
            I\cdot L := & \bigcup_{x\in I} x\cdot L = \{\gamma\in L\mid \gamma_b\in I\} \\
            L\cdot I := & \bigcup_{x\in I} L\cdot x = \{\gamma\in L \mid \gamma_a\in I\}.
        \end{align*}
    \end{itemize}

    We denote by $\mathcal{L}$ the set of measured laminations of $\hyper$ and by $\overline{\mathcal{L}}$ the set of laminations of $\hyper$ (without a measure).
\end{notation}

Now we define how to obtain a useful function $F$ from any measured lamination $L\in\mathcal{L}$.
We will use this to define a function $\mathcal{L}\to \mathcal{S}_\fpc(Q)$, where $Q$ is the continuous quiver of type $\mathbb{A}$ with straight descending orientation.

\begin{definition}\label{def:measured laminations to useful functions}
    Let $(L,\Meas)\in\mathcal{L}$.
    We will define a useful function $F$ on $-\infty$, $+\infty$, and then all of $\RR$.
    For $-\infty$, define
    \begin{align*}
        u_{-\infty}^- :=& 0 & u_{-\infty}^+ :=& -\Meas( (-\infty) \cdot L ) \\
        F(-\infty) :=& 0 & f(-\infty) :=& 0.
    \end{align*}
    For $+\infty$, define
    \begin{displaymath}
        u_{+\infty}^- = u_{+\infty}^+ = F(+\infty) = f(+\infty) = 0.
    \end{displaymath}
    For each $a\in\RR$, define
    \begin{align*}
        u_a^- := & \Meas( L \cdot a ) & u_a^+ := & -\Meas( a \cdot L ) \\
        F(a) := & -\Meas\left(\frac{L}{a}\right) & f(a) := & F(a) - \left(\sum_{x\leq a} u_x^-\right) - \left(\sum_{x<a} u_x^+\right).
    \end{align*}
\end{definition}

First, note that since $\Meas(L)<\infty$, each of the assignments is well-defined.
It remains to show that $F$ is a useful function.

\begin{proposition}\label{prop:measured laminations to useful functions}
    Let $(L,\Meas)\in\mathcal{L}$ and let $F$ be as in Definition~\ref{def:measured laminations to useful functions}.
    Then $F$ is useful.
\end{proposition}
\begin{proof}
    Since $\Meas(L)<\infty$, we see $\sum_{x\in\RR\cup\{+\infty\}} |u_x^-| + \sum_{x\in\RR\cup\{-\infty\}} |u_x^+| <\infty$.
    Now we show $f$ is continuous.
    Consider $\lim_{x\to a^-}f(x)$ for any $a\in\RR$:
    \begin{align*}
        \lim_{x\to a^-}f(x)
        &= \lim_{x\to a^-} \left[ F(x) - \left(\sum_{y\leq x} u_y^-\right) - \left(\sum_{y<x} u_y^+\right) \right] \\
        &= \lim_{x\to a^-} \left[ -\Meas\left(\frac{L}{x}\right) -\left(\sum_{y\leq x} \Meas(L\cdot y)\right) - \left( \sum_{y< x} \Meas(y\cdot L) \right) \right] \\
        &= -\Meas\left(\frac{L}{a}\right) - \Meas(L\cdot a) - \left(\sum_{x<a} \Meas(L\cdot a)\right) -\left(\sum_{x<a} \Meas(a\cdot L) \right) \\
        &= F(a) - \left(\sum_{x\leq a} u_x^-\right) - \left(\sum_{x <a} u_x^+\right) \\
        &= f(a).
    \end{align*}
    A similar computation shows $\lim_{x\to a^+}f(x) = f(a)$.
    Therefore, $f$ is continuous on $\RR$.
    We also note that $\lim_{x\to\pm\infty} f(x) = 0$, using similar computations.

    It remains to show that $f$ has bounded variation.
    Let $a<b\in \RR$ and let $F_0=f$.
    Denote by $\var_f([a,b))$ the variance of $f$ over $[a,b)$.
    We see that
    \[ \var_f([a,b)) = \Meas(([a,b)\cdot L)\cup(L\cdot[a,b))) - \sum_{x\in[a,b)} (\Meas(x\cdot L)+\Meas(L\cdot x). \]
    That is, $\var_f([a,b)$ is the measure the geodesics with endpoints in $[a,b)$ that are not discrete and do not belong to a fountain.
    So, \[ \var_f([a,b)) \leq \Meas(([a,b)\cdot L)\cup(L\cdot [a,b))).\]
    Then we have
    \[ \var_f(\RR) = \sum_{i\in\ZZ} \var_f([i,i+1)) \leq \sum_{i\in\ZZ} \Meas(([i,i+1)\cdot L)\cup(L\cdot [i,i+1))) < \infty. \]
    Thus, $f$ has bounded variation.
\end{proof}

We state the following lemma without proof, since the proof follows directly from Definition~\ref{def:red-blue function pair}~and~\ref{def:measured laminations to useful functions} and Proposition~\ref{prop:measured laminations to useful functions}.
\begin{lemma}
    Let $(L,\Meas)\in\mathcal{L}$ and let $F$ be as in Definition~\ref{def:measured laminations to useful functions}.
    Then $(F,0)$ is a red-blue function pair for the continuous quiver $Q$ of type $\mathbb{A}$ with straight descending orientation.
\end{lemma}

Now wow define the function $\mathcal L\to \mathcal{S}(Q)$.
\begin{definition}\label{def:measured lamination to stability conditions}
    Let $(L,\Meas)\in \mathcal{L}$, let $F$ be as in Definition~\ref{def:measured laminations to useful functions}, and let $Q$ be the continuous quiver of type $\mathbb{A}$ with straight descending orientation.
    The map $\Phi:\mathcal{L}\to \mathcal{S}(Q)$ is defined by setting $\Phi((L,\Meas))$ equal to the equivalence class of $(F,0)$.
\end{definition}

\begin{lemma}\label{lem:measured lamination and geodesics}
    Let $L\in \mathcal{L}$ and let $\partial\hyper$ be indexed as $\RR\cup\{-\infty\}$, as before.
    Suppose there are points $a,b\in \partial\hyper$ such that for all $x\in(a,b)$ we have $\Meas(\frac{L}{x})\geq \Meas(\frac{L}{a})$ and $\Meas(\frac{L}{x})\geq \Meas(\frac{L}{b})$.
    Then the geodesic in $\hyper$ uniquely determined by $a$ and $b$ is in $L$.
\end{lemma}
\begin{proof}
    For contradiction, suppose there is $\alpha\in L$ such that $\alpha$ is uniquely determined by $c$ and $d$, where $a<c<b<d$.
    Then, we must have $\beta\in L$ uniquely determined by $c$ and $d$, or else there is a set $K$ with positive measure such that $K\subset \frac{L}{b}$ but $K\not\subset\frac{L}{c}$.
    Similarly, we must have $\gamma\in L$ uniquely determineed by $a$ and $c$.
    Now, we cannot have a fountain at $c$ or else we will have a set with positive measure $K$ such that $K\subset \frac{L}{b}$ or $K\subset\frac{L}{a}$ but $K\not\subset\frac{L}{c}$.
    Since $c$ has a geodesic to both the left and right, both $\alpha$ must be discrete.
    But then $\{\alpha\}$ has positive measure, a contradiction.
    Thus, there is no $\alpha\in L$ such that $\alpha$ is uniquely determined by $c$ and $d$, where $a<c<b<d$.
    Similarly, there is no $\alpha\in L$ such that $\alpha$ is uniquely determined by $c$ and $d$, where $c<a<d<b$.
    Therefore, since $L$ is maximal, we must have the geodesic uniquely determined by $a$ and $b$ in $L$.
\end{proof}

\begin{proposition}\label{prop:measured lamination to four point conditions}
    Let $(L,\Meas)\in \mathcal{L}$, let $F$ be as in Definition~\ref{def:measured laminations to useful functions}, and let $Q$ be the continuous quiver of type $\mathbb{A}$ with straight descending orientation.
    Then $\Phi((L,\Meas))\in\mathcal{S}_\fpc(Q)$.
\end{proposition}
\begin{proof}
    For contradiction, suppose there exists a $\Phi((L,\Meas))$-semistable module $M_I$ such that $\lvert (\widehat{I}\times\{h\})\cap (\widehat{Q}\times \RR)\rvert \geq 4$.
    Choose 4 points $a<b<c<d$ in $\widehat{Q}$ corresponding to four intersection points.
    
    For the remainder of this proof, write $x\geo y$ to mean the geodesic in $\hyper$ uniquely determined by $x\neq y\in\partial\hyper$.
    By Lemma~\ref{lem:measured lamination and geodesics}, we have the following geodesics in $L$: $a\geo b$, $a\geo c$, $a\geo d$, $b\geo c$, $b\geo d$, and $c\geo d$.
    However, this is a quadrilateral with \emph{both} diagonals, as shown in Figure~\ref{fig:quadrilateral with diagonals}.
    Since $L$ is a lamination, this is a contradiction.
    \begin{figure}
    \centering
    \begin{tikzpicture}[scale=1.3]
        \draw(1,0) arc (0:360:1);
        \filldraw[fill=black, draw=black] (.707,.707) circle[radius=.4mm];
        \draw(.707,.707) node[anchor=south west] {$c$};
        \filldraw[fill=black, draw=black] (-.707,.707) circle[radius=.4mm];
        \draw(-.707,.707) node[anchor=south east] {$d$};
        \filldraw[fill=black, draw=black] (.707,-.707) circle[radius=.4mm];
        \draw(.707,-.707) node[anchor=north west] {$b$};
        \filldraw[fill=black, draw=black] (-.707,-.707) circle[radius=.4mm];
        \draw(-.707,-.707) node[anchor=north east] {$c$};
        \filldraw[draw=black,fill=black](-1,0) circle[radius=.4mm];
        \draw(-1,0) node[anchor=east] {$-\infty$};
        \draw (.707,.707) arc (315:225:1);
        \draw (.707,.707) arc (135:225:1);
        \draw (.707,-.707) arc (45:135:1);
        \draw (-.707,.707) arc (45:-45:1);
        \draw(-.707,-.707) -- (.707,.707);
        \draw (-.707,.707) -- (.707,-.707);
    \end{tikzpicture}
    \caption{The geodesics $a\geo b$, $a\geo c$, $a\geo d$, $b\geo  c$, $b\geo d$, and $c\geo d$ used in the proof of Proposition~\ref{prop:measured lamination to four point conditions}.
    Notice $a\geo b$, $b\geo c$, $c\geo d$, and $a\geo d$ form a quadrilateral and its diagonals, $a\geo c$ and $b\geo d$, cross.}\label{fig:quadrilateral with diagonals}
    \end{figure}
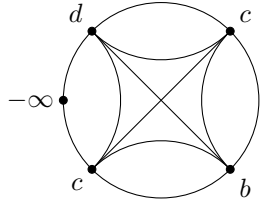
\end{proof}

\begin{theorem}\label{thm:measured laminations and straight four point conditions}
    Let $Q$ be the continuous quiver of type $\mathbb{A}$ with straight descending orienation.
    Then $\Phi:\mathcal{L}\to \mathcal{S}_\fpc(Q)$ is a bijection.
    Furthermore, for a measured lamination $L$ and stability condition $\Phi(L)$, there is a bijection $\phi_L$ from $L$ to $\Phi(L)$-semistable indecomposable modules.
\end{theorem}
\begin{proof}
    By the proof of Proposition~\ref{prop:measured lamination to four point conditions}, we see that the second claim follows.
    Thus, we now show $\Phi$ is a bijection.

    \textbf{Injectivity.}
    Consider $(L,\Meas)$ and $(L',\Meas')$ in $\mathcal{L}$.
    Let $\sigma=\Phi(L,\Meas)$ and $\sigma'=\Phi(L',\Meas')$.
    If $L\neq L'$ then we see that the set of $\sigma$-semistable modules is different from the set of $\sigma'$-semistable modules.
    Thus, $\sigma\neq\sigma'$.
    If $L=L'$ but $\Meas\neq\Meas'$ there must be some $x\in\RR\cup\{-\infty\}$ such that $\Meas(\frac{L}{x})\neq \Meas'(\frac{L'}{x})$.
    But the functions $F$ and $F'$ from $L$ and $L'$, respectively using Definition~\ref{def:measured laminations to useful functions}, both have the same limits at $\pm\infty$.
    Thus, $\widehat{\mathcal{G}}(F,0)$ is not a vertical translation of $\widehat{\mathcal{G}}(F',0)$ in $\RR^2$.
    Therefore, $\sigma\neq \sigma'$.

    \textbf{Surjectivity.}
    Let $\sigma$ be a stability condition.
    Let $T$ be the maximal $\Npi$-compatible set of indecomposable modules determined by $\sigma$ (Theorem~\ref{thm:four point equivalent to cluster}).
    Let $L$ be the lamination of $\hyper$ uniquely determined by $T$ (Theorem~\ref{thm:igusa todorov}).
    In particular, the indecomposable $M_I$ corresponds to the geodesic uniquely determined by $\inf I$ and $\sup I$.

    Let $(R,B)$ be the representative of $\sigma$ such that $B=0$; that is, $(R,B)=(R,0)$.
    For each $x\in \partial\hyper$, let
    \begin{align*}
        \Meas(L\cdot x) &= u_x^- &
        \Meas(x\cdot L) &= -u_x^+ \\
        \Meas\left(\frac{L}{x}\right) &= R(x).
    \end{align*}
    Since $r$ must have bounded variation and $\sum_{x\in\overline{R}} |u_x^-|+|u_x^+|<\infty$, we see $\Meas(L) <\infty$.

    Let $O_{A,B}\subset L$ be a basic open subset.
    If $O_{A,B}=\emptyset$ then we're done.

    Now we assume $O_{A,B}\neq\emptyset$ and let $\gamma\in O_{A,B}$.
    If there exist two stability indicators for the indecomposable $M_I$ corresponding to $\gamma$, with heights $h_0<h_1$, then we know $\Meas(\{\gamma\})>|h_1-h_0|>0$ and so $\Meas(O_{A,B})>0$.

    We now assume there is a unique stability indicator of height $h$ for the indecomposable $M_I$ corresponding to $\gamma$.
    Without loss of generality, since $O_{A,B}=O_{B,A}$, assume $a=\gamma_a\in A$ and $b=\gamma_b\in B$.
    We know that, for all $a<x<b$, we have $R(x)\leq R(a)$ and $R(x)\leq R(b)$.
    There are two cases: (1) $R(x)<R(a)$ and $R(x)<R(b)$, for all $x\in(a,b)$, and (2) there exists $x\in(a,b)$ such that $R(x)=R(a)$ or $R(x)=R(b)$.

    \textbf{Case (1).}
    Let $e=\tan (\frac{1}{2}(\tan^{-1}(a)+\tan^{-1}(b)))$.
    Let $\{h_i\}_{i\in\NN}$ be a strictly increasing sequence such that $h_0=R(e)$ and $\lim_{i\to\infty} h_i = h$.
    By Lemma~\ref{lem:local max and local min}(\ref{lem:local max and local min:closed box}) and our assumption that $\Meas(\{\gamma\})=0$, for each $i>0$, there is a stability indicator with height $h_i$ and endpoints $a_i, b_i$ such that $a < a_i < b_i < b$.
    Then $\lim_{i\to\infty} a_i = a$ and $\lim_{i\to\infty} b_i = b$, again by Lemma~\ref{lem:local max and local min}(\ref{lem:local max and local min:closed box}).
    Since $A$ and $B$ are open, there is some $N\in\NN$ such that, for all $i\geq N$, we have $a_i\in A$ and $b_i\in B$.
    Let $C=(a,a_N)$ and $D=(b_N,b)$.
    Then, $\Meas(O_{C,D})\geq |h-h_N|$ and so $\Meas(O_{A,B})>0$.

    \textbf{Case (2).}
    Assume there exists $x\in(a,b)$ such that $R(x)=R(a)$ or $R(x)=R(b)$.
    Let $e$ be this $x$.
    If $R(x)=R(a)$ and $a=-\infty$, then $R(b)=0$ (or else $\gamma\notin O_{A,B}\subset L$).
    Then we use the technique from Case (1) with $b$ and $+\infty$ to obtain some $C=(b,d)$ and $D=(c,+\infty)$ such that $\Meas(O_{C,D})>0$.
    Thus, $\Meas(O_{A,B})>0$.

    Now we assume $a>-\infty$ and $R(x)=R(a)$ or $R(x)=R(b)$.
    We consider $R(x)=R(b)$ as the other case is similar.
    Since $\sigma$ satisfies the four point condition, we know that for any $\varepsilon > 0$ such that $R(b+\varepsilon) < R(b)$ we must have $0<\lambda<\varepsilon$ such that $R(b+\lambda)> R(b)$.
    Similarly, for any $\varepsilon > 0$ such that $R(a-\varepsilon) < R(b)$ we must have $0\leq\lambda<\varepsilon$ such that $R(a-\lambda)> R(b)$.
    Notice the strict inequality in the statement about $R(b+\lambda)$ and the weak inequality in the statement about $R(a-\lambda)$.

    Let $\{h_i\}$ be a strictly decreasing sequence such that $h_0=0$ and $\lim_{i\to\infty} h_i = h$.
    By Lemma~\ref{lem:local max and local min}(\ref{lem:local max and local min:closed box}) and our assumption that $\Meas(\{\gamma\})=0$, for each $i>0$, there is a stability indicator with height $h_i$ and endpoints $a_i, b_i$ such that $a_i \leq a < b < b_i$.
    Since $\sigma$ satisfies the four point condition, and again by Lemma~\ref{lem:local max and local min}(\ref{lem:local max and local min:closed box}), $\lim_{i\to\infty} b_i = b$.
    Since $A$ and $B$ are open, there is $N\in\mathbb{N}$ such that, if $i\geq N$, we have $a_i\in A$ and $b_i \in B$.
    If $a_i=a$ for any $i>N$, let $C$ be a tiny epsilon ball around $a$ that does not include $b$.
    Otherwise, let $C=(a_N,a)$.
    Let $D=(b,b_N)$.
    Then $\Meas(O_{C,D})\geq |h_N-h|$ and so $\Meas(O_{A,B})>0$.
    
    \textbf{Conclusion.}
    Since $\Meas(L)<\infty$, we know $\Meas(O_{A,B})<\infty$ for each $O_{A,B}$.
    This proves $(L,\Meas)$ is a measured lamination.
    By the definition of $\Phi$, we see that $\Phi(L,\Meas)=\sigma$.
    Therefore, $\Phi$ is surjective and thus bijective.
\end{proof}

\begin{corollary}[to Theorems~\ref{thm:stability tilting}~and~\ref{thm:measured laminations and straight four point conditions}]\label{cor:measured laminations and straight four point conditions for all orientations}
    Let $Q$ be a continuous quiver of type $\mathbb{A}$.
    Then there is a bijection $\Phi:\mathcal{L}\to \mathcal{S}_\fpc(Q)$.
    Furthermore, for a measured lamination $L$ and stability condition $\Phi(L)$, there is a bijection $\phi_L$ from $L$ to $\Phi(L)$-semistable indecomposable modules.
\end{corollary}

\begin{theorem}\label{thm:explicit statement about stability conditions, tilting, and measured laminations}
    Let $\sigma\in \mathcal S_{fpc}(Q)$ be the stability condition given by $(R,B)$ and let $\sigma'\in \mathcal S_{fpc}(Q')$ be given by $(R',B')$. Then $\sigma,\sigma'$ give the same measured lamination on the Poincar\'e disk.
\end{theorem}

\begin{proof}
The set of geodesics going from intervals $(a,b)$ to $(x,y)$ has the same measure as those going from $(\mathfrak t(a),\mathfrak t(b))$ to $(\mathfrak t(x),\mathfrak t(y))$ where we may have to reverse the order of the ends. We can break up the intervals into pieces and assume that $(a,b),(x,y)$ are either both in $K$, both in $Z$ or one is in $K$ and the other in $Z$. The only nontrivial case is when $(a,b)$ is in $K$ and $(x,y)$ is in $Z$. In that case, the measure of this set of geodesics for $\sigma$ is equal to the variation of $H$ on $(a,b)$ since the islands don't ``see'' $Z$. Similarly, the measure of the same set of geodesics for $\sigma'$, now parametrized as going from $(\mathfrak t(b),\mathfrak t(a))$ to $(x,y)$ is equal to the variation of $H'$ on $(\mathfrak t(b),\mathfrak t(a))$ where $H'(z)=H_+(\mathfrak t(z))$. 

There is one other case that we need to settle: We need to know that the local variation of $H$ at $-\infty$ is equal to the local variation of $H'$ at $r$. But this holds by definition of $H,H'$.
\end{proof}

An example of a stability condition $\sigma$ and corresponding measured lamination are shown in Figures \ref{fig:integral 1}, \ref{fig:lamination 1}. The continuously tilted stability condition $\sigma'$ is shown in Figure \ref{fig:integral 2}.

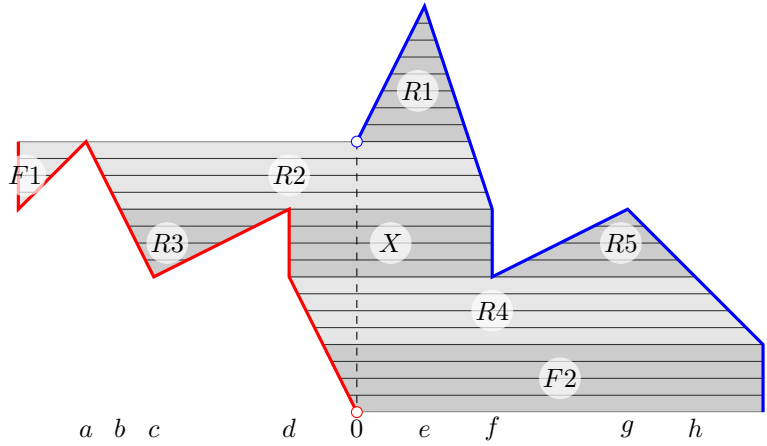
\begin{figure}
    \centering

    \begin{tikzpicture}[scale=.9]
        \begin{scope}
            \clip (-5,2) --(0,2)-- (1,4) -- (2,1) -- (2,0) -- (4,1) -- (6,-1) -- (6,-2) -- (0,-2) -- (-1,0) -- (-1,1) -- (-3,0) -- (-4,2) -- (-5,1) -- (-5,2);
            \filldraw[opacity=.1] (-5,-1) -- (6,-1) -- (6,0) -- (-5,0) -- (-5,-1);
            \filldraw[opacity=.2] (-5,-2) -- (-5, -1) -- (6,-1) -- (6,-2) -- (-5, -2);
            \filldraw[opacity=.2] (-5,1) -- (-5, 0) -- (6,0) -- (6,1) -- (-5, 1);
            \filldraw[opacity=.1] (-5,1) -- (-5, 2) -- (6,2) -- (6,1) -- (-5, 1);
            \filldraw[opacity=.2] (-5,4) -- (-5, 2) -- (6,2) -- (6,4) -- (-5, 4);
            \foreach \x in {0,...,30}
                \draw[white!20!black] (-5,\x*0.25-2) -- (6,\x*0.25-2);
        \end{scope}
        
     %   \draw[very thick] (-1,0) -- (2,0);
        
        \draw [blue, very thick] (0,2)--(1,4) -- (2,1) -- (2,0) -- (4,1) -- (6,-1) -- (6,-2);
        \draw [red, very thick] (-5,2) -- (-5,1) -- (-4,2) -- (-3,0) -- (-1,1) -- (-1,0) -- (0,-2);
        \draw[dashed] (0,2) -- (0,-2);
        \filldraw[fill=white, draw=blue] (0,2) circle[radius=.8mm];
        \filldraw[fill=white, draw=red] (0,-2) circle[radius=.8mm];
        
        \filldraw[opacity=.7, draw=white, fill=white] (.9,2.75) circle[radius=.3cm];
        \draw(.9,2.75) node {$R1$};
        \filldraw[opacity=.7, draw=white, fill=white] (-4.9,1.5) circle[radius=.3cm];
        \draw (-4.9,1.5) node {$F1$};
        \filldraw[opacity=.7, draw=white, fill=white] (-1,1.5) circle[radius=.3cm];
        \draw (-1,1.5) node {$R2$};
        \filldraw[opacity=.7, draw=white, fill=white] (3.9,0.5) circle[radius=.3cm];
        \draw (3.9,0.5) node {$R5$};
        \filldraw[opacity=.7, draw=white, fill=white] (-2.8,0.5) circle[radius=.3cm];
        \draw (-2.8,0.5) node {$R3$};
        \filldraw[opacity=.7, draw=white, fill=white] (2,-0.5) circle[radius=.3cm];
        \draw (2,-0.5) node {$R4$};
        \filldraw[opacity=.7, draw=white, fill=white] (3,-1.5) circle[radius=.3cm];
        \draw (3,-1.5) node {$F2$};
        \filldraw[opacity=.7, draw=white, fill=white] (0.5,0.5) circle[radius=.3cm];
        \draw (0.5,0.5) node {$X$};
        
             \coordinate (A) at (-4,-2.5);
        \coordinate (B) at (-3.5,-2.5);
        \coordinate (C) at (-3,-2.5);
        \coordinate (D) at (-1,-2.5);
        \coordinate (O) at (0,-2.5);
        \coordinate (E) at (1,-2.5);
        \coordinate (F) at (2,-2.55);
        \coordinate (G) at (4,-2.5);
        \coordinate (H) at (5,-2.5);
        \foreach \x/\xtext in {A/a,B/b,C/c,D/d,E/e,F/f,G/g,H/h,O/0}
            \draw (\x) node[anchor=south]{$\xtext$};
    \end{tikzpicture}
    \caption{The modified~graph of the red-blue function pair $(R,B)$. 
    Horizontal lines indicated semistable indecomposable representations.
   The rectangle labeled $X$ represents one object with positive measure. The measure of a region is given by its height.
   }
\label{fig:integral 1}
\end{figure}

\begin{figure}
\centering
    \begin{tikzpicture}[scale=.8]
    
        \begin{scope}
            \clip (0,0) circle[radius=5cm];
            \filldraw[opacity=.1, rotate around={-27-9:(0,0)}] (-5,0) arc(90:-90:2.714+1.556); %R1 top
            \filldraw[fill=white, rotate around={-27:(0,0)}] (-5,0) arc(90:-90:2.714); %R1 bottom
            \filldraw[opacity=.1] (-5,0) arc(90:-90:1.339); %F2
            \filldraw[opacity=.1] (-5,0) arc(-90:90:1.2); %F1
            \filldraw[opacity=.1, rotate around={45:(0,0)}] (-5,0) arc (90:-90:3.836);
            \filldraw[opacity=.1, rotate around={36:(0,0)}] (5,0) arc (270:90:6.882);
            \filldraw[opacity=.1, rotate around={36:(0,0)}] (5,0) arc (90:270:1.339+4.214);
            \foreach \x in {0, 0.1, 0.2, 0.3, 0.4, 0.5, 0.6, 0.7, 0.8, 0.9, 1}
            {
                                        \draw[rotate around={-16+(\x*16):(0,0)}] (-5,0) arc(-90:90:0.1+\x*1.1); %F1, now R1 
                \draw (-5,0) arc(90:-90:\x*1.339); %F2, now F1
               \draw[rotate around={-27-(\x*9):(0,0)}] (-5,0) arc(90:-90:2.714+\x*1.556); %R2
               \draw[rotate around={45+\x*15:(0,0)}] (-5,0) arc (90:-90:3.836-\x*3.836); %R3
                \draw[rotate around={36+\x*36:(0,0)}] (5,0) arc (270:90:6.882-\x*6.882); %R5
               \draw[rotate around={\x*36:(0,0)}] (5,0) arc (90:270:1.339+\x*4.214); %R4
                \draw (5,0) arc (90:270:\x*1.339); %F3, now F2
            }
            \draw[very thick,rotate around={-36:(0,0)}] (-5,0) arc (90:-90:23.523);
            
            \filldraw[opacity=.7, draw=white, fill=white] (-4.5,1) circle[radius=.3cm];
            \draw (-4.5,1) node {$R1$};
            \filldraw[opacity=.7, draw=white, fill=white] (-4.5,-1) circle[radius=.3cm];
            \draw (-4.5,-1) node {$F1$};
            \filldraw[opacity=.7, draw=white, fill=white] (-2.65,0) circle[radius=.3cm];
            \draw (-2.65,0) node {$R2$};
            \draw (-.3,0) node {$X$};
            \filldraw[opacity=.7, draw=white, fill=white] (0,3) circle[radius=.5cm];
            \draw (0,3) node {$R5$};
            \filldraw[opacity=.7, draw=white, fill=white] (-1,-3.5) circle[radius=.5cm];
            \draw (-1,-3.5) node {$R3$};
            \filldraw[opacity=.7, draw=white, fill=white] (2.5,-1) circle[radius=.3cm];
            \draw (2.5,-1) node {$R4$};
            \filldraw[opacity=.7, draw=white, fill=white] (4.5,-1) circle[radius=.3cm];
            \draw (4.5,-1) node {$F2$};
        \end{scope}
    
        \draw[blue, thick] (5,0) arc (0:180:5);
        \draw[red, thick] (-5,0) arc(180:360:5);
        \filldraw[fill=blue, draw=red, thick] (5,0) circle[radius=1mm];
        \draw (5,0) node[anchor=west] {$+\infty$ ($1^-$)};
        \filldraw[fill=red, draw=blue, thick] (-5,0) circle[radius=1mm];
        \draw (-5,0) node[anchor=east] {($1^+$) $-\infty$};
        \filldraw[rotate around={-150:(0,0)}, red] (5,0) circle[radius=1mm];
             \coordinate (A)  at  (-4.503, -2.6);
        \filldraw[rotate around={-120:(0,0)}, red] (5,0) circle[radius=1mm];
        \coordinate (B) at (-3.7, -3.7);
        \filldraw[rotate around={-135:(0,0)}, red] (5,0) circle[radius=1mm];
        \coordinate (C) at (-2.6, -4.503);
        \filldraw[rotate around={-60:(0,0)}, red] (5,0) circle[radius=1mm];
        \coordinate (D) at (2.6, -4.503);
        \filldraw[rotate around={72:(0,0)}, blue] (5,0) circle[radius=1mm];
       \coordinate (E) at (-5, 1.503);
        \filldraw[rotate around={162.8:(0,0)}, blue] (5,0) circle[radius=1mm];
        \coordinate (G) at (1.65, 5);
        \filldraw[rotate around={144:(0,0)}, blue] (5,0) circle[radius=1mm];
        \coordinate (H) at (4.25, 3.2);
        \filldraw[rotate around={36:(0,0)}, blue] (5,0) circle[radius=1mm];
        \coordinate (F) at (-4.206, 3.056);
        
        \foreach \x/\xtext in {A/a,B/b,C/c,D/d,E/e,F/f,G/g,H/h}
            \draw (\x) node{$\xtext$};
    \end{tikzpicture}
    \caption{The lamination of the hyperbolic plain \{corresponding to the stability condition shown in Figure \ref{fig:integral 1}. 
   The thick arc labeled $X$ is an isolated geodesic with positive measure. The measure is equal to the height of the rectangle labeled $X$ in Figure \ref{fig:integral 1}.}
    \label{fig:lamination 1}
\end{figure}
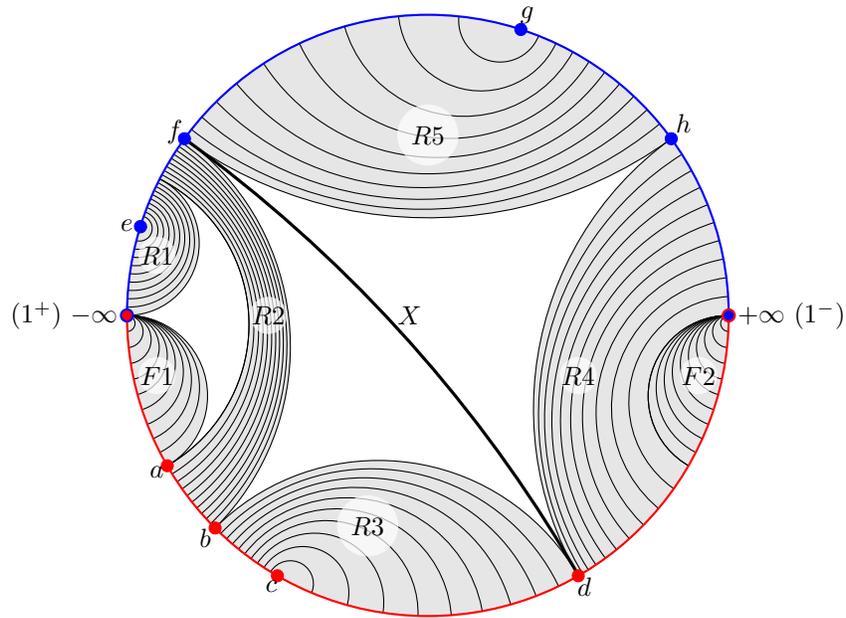

% Integral 2: continuous tilting of integral 1 (same lamination)

\begin{figure}
    \centering

    \begin{tikzpicture}[scale=.9]
        \begin{scope}
 %           \clip (-5,2) --(0,2)-- (1,4) -- (2,1) -- (2,0) -- (4,1) -- (6,-1) -- (6,-2) -- (0,-2) -- (-1,0) -- (-1,1) -- (-3,0) -- (-4,2) -- (-5,1) -- (-5,2);
 %(0,2) -- (0,3) -- (-1,2)--(-1.5,1) -- (-2,2) -- (-4,1) -- (-4,0) -- (-5,-2)
            \clip (-5,-2)--(-4,0)--(-4,1)--(-2,2)--(-1.5,1)--(-1,2)-- (0,3)--(0,2)-- (1,4) -- (2,1) -- (2,0) -- (4,1) -- (6,-1) -- (6,-2) -- (0,-2) -- (-5,-2);% (-1,0) -- (-1,1) -- (-3,0) -- (-4,2) -- (-5,1) -- (-5,2);
            \filldraw[opacity=.1] (-5,-1) -- (6,-1) -- (6,0) -- (-5,0) -- (-5,-1);
            \filldraw[opacity=.2] (-5,-2) -- (-5, -1) -- (6,-1) -- (6,-2) -- (-5, -2);
            \filldraw[opacity=.2] (-5,1) -- (-5, 0) -- (6,0) -- (6,1) -- (-5, 1);
            \filldraw[opacity=.1] (-5,1) -- (-5, 2) -- (6,2) -- (6,1) -- (-5, 1);
            \filldraw[opacity=.2] (-5,4) -- (-5, 2) -- (6,2) -- (6,4) -- (-5, 4);
            \foreach \x in {0,...,30}
                \draw[white!20!black] (-5,\x*0.25-2) -- (6,\x*0.25-2);
        \end{scope}
        
     %   \draw[very thick] (-1,0) -- (2,0);
        
        \draw [blue, very thick] (0,2)--(1,4) -- (2,1) -- (2,0) -- (4,1) -- (6,-1) -- (6,-2);
%        \draw [red, very thick] (-5,2) -- (-5,1) -- (-4,2) -- (-3,0) -- (-1,1) -- (-1,0) -- (0,-2);
        \draw [blue, very thick] (0,2) -- (0,3) -- (-1,2)--(-1.5,1) -- (-2,2) -- (-4,1) -- (-4,0) -- (-5,-2);
        \draw[dashed] (0,2) -- (0,-2);
%        \filldraw[fill=white, draw=blue] (0,2) circle[radius=.8mm];
%        \filldraw[fill=white, draw=red] (0,-2) circle[radius=.8mm];
        
        \filldraw[opacity=.7, draw=white, fill=white] (.9,2.75) circle[radius=.3cm];
        \draw(.9,2.75) node {$R1$};
        \filldraw[opacity=.7, draw=white, fill=white] (-.3,2.3) circle[radius=.3cm];
        \draw (-.3,2.3) node {$F1$};
        \filldraw[opacity=.7, draw=white, fill=white] (-.5,1.5) circle[radius=.3cm];
        \draw (-.5,1.5) node {$R2$};
        \filldraw[opacity=.7, draw=white, fill=white] (3.9,0.5) circle[radius=.3cm];
        \draw (3.9,0.5) node {$R5$};
        \filldraw[opacity=.7, draw=white, fill=white] (-2.3,1.4) circle[radius=.3cm];
        \draw (-2.3,1.4) node {$R3$};
        \filldraw[opacity=.7, draw=white, fill=white] (2,-0.5) circle[radius=.3cm];
        \draw (2,-0.5) node {$R4$};
        \filldraw[opacity=.7, draw=white, fill=white] (3,-1.5) circle[radius=.3cm];
        \draw (3,-1.5) node {$F2$};
        \filldraw[opacity=.7, draw=white, fill=white] (0.5,0.5) circle[radius=.3cm];
        \draw (0.5,0.5) node {$X$};
        
%             \coordinate (A) at (-4,-2.5);
            \coordinate (A) at (-0.93,-2.57);
%        \coordinate (B) at (-3.5,-2.5);
        \coordinate (B) at (-1.5,-2.57);
 %       \coordinate (C) at (-3,-2.5);
        \coordinate (C) at (-2.07,-2.57);
%        \coordinate (D) at (-1,-2.5);
        \coordinate (D) at (-4,-2.57);
        \coordinate (O) at (0,-2.5);
        \coordinate (E) at (1,-2.5);
        \coordinate (F) at (2,-2.55);
        \coordinate (G) at (4,-2.5);
        \coordinate (H) at (5,-2.5);
        \foreach \x/\xtext in {A/\overline{\mathfrak t}(a),B/\overline{\mathfrak t}(b),C/\overline{\mathfrak t}(c),D/\overline{\mathfrak t}(d),E/e,F/f,G/g,H/h,O/0}
            \draw (\x) node[anchor=south]{$\xtext$};
    \end{tikzpicture}
    \caption{This is the continuous tilting of Figure \ref{fig:integral 1}. There are two islands $F1$ and $R3$ which have been flipped up. The measured lamination (Figure \ref{fig:lamination 1}) is unchanged, only relabeled. }
\label{fig:integral 2}
\end{figure}
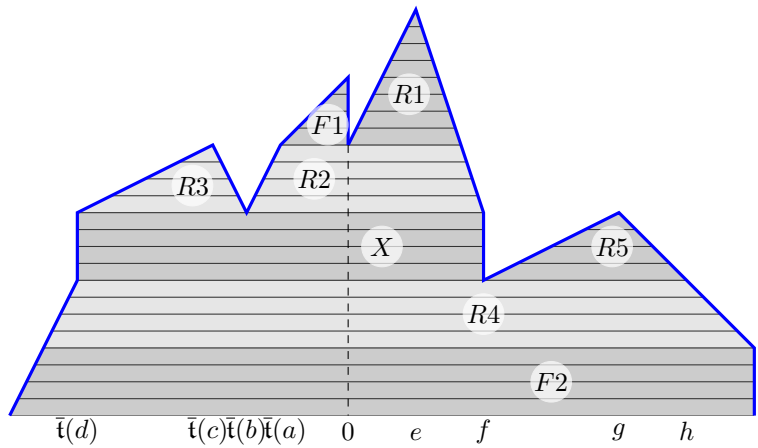

\subsection{Continuous cluster character}\label{sec:continuous cluster character}

We present a candidate for the continuous cluster character using the formula from \cite{maresca5} which applies in the continuous case. This lives in a hypothetical algebra having a variable $x_t$ for every real number $t$. In this algebra which we have not defined, we give a simple formula for the cluster variable of an admissible module $M_{ab}$ where $a<b$ and the quiver $Q$ is oriented to the left (is red) in a region containing $(a,b]$ in its interior. In analogy with the cluster character in the finite case \eqref{eq: CC(M) finite straight} or \cite{maresca5}, {replacing summation with integration}, we define $\chi(M_{ab})$ to be the formal expression:
\begin{equation}\label{eq: continuous cluster character-straight}
\chi(M_{ab})=\int_a^b \frac{x_ax_b\dd t}{x_t^2}.
\end{equation}
This could be interpreted as an actual integral of some function $x_t$. For example, if we let $x_t=t$ then we get $\chi(M_{ab})=b-a$, the length of the support of $M_{ab}$. The constant function $x_t=1$ gives the same result.

The same cluster character formula will be used for modules with support $[a,b)$ in the blue region (where the quiver is oriented to the right).

This can also be written as
\[
    \chi(M_{ab})=x_a\chi(P_b)-x_b\chi(P_a)
\]
where $P_b$ is the projective module at $b$ with cluster character
\[
\chi(P_b)=\int_{-\infty}^b \frac{x_b \dd t}{x_t^2}.
\]
Then the cluster mutation equation
\[
\chi(M_{ac})\chi(M_{bd})=\chi(M_{ab})\chi(M_{cd})+\chi(M_{bc})\chi(M_{ad})
\]
follows, as in the finite $A_n$ case, from the Pl\"ucker relation on the matrix:
\[
\left[\begin{matrix}
x_a & x_b & x_c& x_d\\
\chi(P_a) &\chi(P_b) &\chi(P_c) &\chi(P_d)
\end{matrix}
\right].
\]

In Figures \ref{fig:integral 1} and \ref{fig:lamination 1}, if the measure of $X=M_{df}$ is decreased to zero, the height of the rectangle in Figure \ref{fig:integral 1} will go to zero, the four point condition will be violated and we can mutate $X$ to $X^\ast=M_{bh}$. Then the cluster characters are mutated by the Ptolemy equation:
\[
\chi(X)\chi(X^\ast)=\chi(M_{fh})\chi(M_{bd})+\chi(M_{dh})\chi(M_{bf})
\]
where $\chi(M_{bd})$ and $\chi(M_{fh})$ are given by \eqref{eq: continuous cluster character-straight} and the other four terms have a different equation since there is a source (0) in the middle ($d<0<f$):
\[
\chi(X)=\chi(M_{df})=\int_d^0\int_0^f\frac{x_dx_0x_f}{x_s^2x_t^2}\dd s\dd t+\frac{x_dx_f}{x_0}.
\]
The double integral counts the proper submodules $M_{ds}\oplus M_{tf}\subset X$ and there is one more term for the submodule $X\subseteq X$.

The continuous cluster character will be explained in more detail in another paper.

\subsection{Every $\Npi$-cluster comes from a stability condition}\label{sec:all maximal compatible sets come from a stability condition}

Let $L$ be a lamination of $\hyper$.
Then there exists a measured lamination $(L,\Meas)$ in the following way.
There are at most countably many discrete arcs in $L$.
Assign each discrete arc a natural number $n$.
Then, set $\Meas(\{\gamma_n\})=\frac{1}{1+n^2}$, for $n\in\NN$.
Let $K$ be the set of all discrete geodesics in $L$.
On $L\setminus K$, give each $O_{A,B}$ it's transversal measure.
Thus, we have given $L$ a finite measure satisfying Definition~\ref{def:measured lamination}.
Therefore, $(L,\Meas)$ is a measured lamination.
This means the set of measured laminations, $\mathcal{L}$, surjects on to the set of laminations, $\overline{\mathcal{L}}$, by ``forgetting'' the measure.
Then, the set $\mathcal{S}_\fpc(Q)$, for some continuous quiver of type $\mathbb{A}$ with finitely many sinks and sources, surjects onto the set of $\Npi$-clusters, $\mathcal{T}_{\Npi}$, in the following way.
\begin{displaymath}
    \xymatrix{
        \mathcal{S}_\fpc(Q) \ar@{<->}[r]^-{\Phi} \ar@{-->>}[d]_-{\exists} & \mathcal{L} \ar@{->>}[d] \\
        \mathcal{T}_{\Npi} \ar@{<->}[r]_-{\overline{\Phi}} & \overline{\mathcal{L}}.
    }
\end{displaymath}
Essentially, there is a surjection $\mathcal{S}_\fpc(Q)\twoheadrightarrow \mathcal{T}_{\Npi}$ defined using the surjection $\mathcal{L}\twoheadrightarrow \overline{\mathcal{L}}$.
If we follow the arrows around, we see that each stability condition $\sigma$ is set to the set of $\sigma$-semistable modules, which form an $\Npi$-cluster.

\subsection{Maps between cluster categories of type $\mathbb{A}_n$}\label{sec:finite map using continuous tilting}
Let $Q$ be a quiver of type $\mathbb{A}_n$, for $n\geq 2$.
Label the vertices $1,\ldots,n$ in $Q$ such that there is an arrow between $i$ and $i+1$ for each $1\leq i <n$.

For each $i\in\{-1,0,\ldots,n,n+1,n+2\}$ let
\[ x_i = \tan\left(\frac{i+1}{n+3}\pi - \frac{\pi}{2}\right). \]
We define a continuous quiver $\mathcal{Q}$ of type $\mathbb{A}$ based on $Q$, called the \textdef{continuification} of $Q$.
If $1$ is a sink (respectively, source) then $-\infty$ is a sink (respectively, source) in $\mathcal{Q}$.
If $n$ is a sink (respectively, source) then $+\infty$ is a sink (respectively, source in $\mathcal{Q}$).
For all $i$ such that $2\leq i \leq n-1$, we have $x_i$ is a sink (respectively, source) in $\mathcal{Q}$ if and only if $i$ is a sink (respectively, source) in $Q$.

Define a map $\Omega:\Ind(\mathcal{C}(Q))\to\Indr(\mathcal{Q})$ in Figure~\ref{fig:omega map} on page \pageref{fig:omega map}.
\begin{sidewaysfigure*}
\vspace{82.5ex}
\begin{align*}
P_i &\stackrel{\Omega}{\mapsto} P_{x_i} \\
M_{[i,j]} &\stackrel{\Omega}{\mapsto}
    \begin{cases}
        M_{[x_i,x_j]} & i\text{ is blue},\ j\text{ is red}\\
        M_{[x_i,x_{j+2})} &  i\text{ is blue},\ j\text{ is blue},\ j+1\text{ is not a sink or }j+1=n\\
        M_{[x_i,x_\ell)} & i\text{ is blue},\ j\text{ is blue},\ j+1< n\text{ is a sink},\ \ell=(\text{smallest source}>j)+1\\
        M_{(x_{i-2},x_j]} & i\text{ is red},\ i-1\text{ is not a sink or }i-1=1,\ j\text{ is red}\\
        M_{(x_{i-2},x_{j+2})} & i\text{ is red},\ i-1\text{ is not a sink or }i-1=1,\ j\text{ is blue},\ j+1\text{ is not a sink or }j+1=n\\
        M_{(x_{i-2},x_\ell)} & i\text{ is red},\ i-1\text{ is not a sink or }i-1=1,\ j\text{ is blue},\ j+1< n\text{ is a sink},\ \ell=(\text{smallest source}>j)+1\\
        M_{(x_k,x_j]} & i\text{ is red},\ i-1> 1\text{ is a sink},\ k=(\text{largest source})<i)-1,\ j\text{ is red}\\
        M_{(x_k,x_{j+2})} & i\text{ is red},\ i-1> 1\text{ is a sink},\ k=(\text{largest source})<i)-1,\ j\text{ is blue},\ j+1\text{ is not a sink or }j+1=n\\
        M_{(x_k,x_\ell)} & {\scriptsize i\text{ is red},\ i-1> 1\text{ is a sink},\ k=(\text{largest source})<i)-1,\ j\text{ is blue},\ j+1< n\text{ is a sink},\ \ell=(\text{smallest source}>j)+1 }
    \end{cases} \\
    P_i[1] & \stackrel{\Omega}{\mapsto}
    \begin{cases}
        M_{(x_{i-1},x_{i+1})} & i\text{ is a source} \\
        M_{[x_0, x_{n+1}]} & i\text{ is the only sink} \\
        M_{(x_j,x_{n+1}]} & i\text{ is a sink},\ \not\exists\text{source}>i,\ \exists\text{sink}<i,\ j=(\text{largest source}<i)-1\\
        M_{[0,x_j)} & i\text{ is a sink},\ \not\exists\text{source}<i,\ \exists\text{sink}>i\, j=(\text{smallest source}>i)+1\\
        M_{(x_k,x_\ell)} & i\text{ is a sink},\ \exists\text{source}>i,\ \exists\text{source}<i,\ j=(\text{largest source}<i)-1,\ k=(\text{smallest source}>i)+1 \\
        M_{[x_0,x_{i+1})} & i\text{ is not a sink/source},\ i\text{ is blue},\ \not\exists\text{sink}<i\\
        M_{(x_{i-1},x_{n+1}]} & i\text{ is not a sink/source},\ i\text{ is red},\ \not\exists\text{sink}>i\\
        M_{(x_j,x_{i+1})} & i\text{ is not a sink/source},\ i\text{ is blue},\ \exists\text{sink}<i\, j=(\text{biggest source}<i)-1\\
        M_{(x_{i-1},x_j)} & i\text{ is not a sink/source},\ i\text{ is red},\ \exists\text{sink}>i\, j=(\text{smallest source}>i)+1
    \end{cases}
\end{align*}
\caption{The map $\Omega:\Ind(\mathcal{C}(Q))\to\Indr(\mathcal{Q})$.
In the formulae above, $M_{[i,j]}$ is assumed to be not projective.
We say a discrete vertex $i$ is \textdef{red} if there exists arrows $i-1\leftarrow i \leftarrow i+1$.
We say $i$ is blue if there exist arrows $i-1\rightarrow i \rightarrow i+1$.
Additionally, for a discrete interval $[i,j]$ we say $i$ is red if it meets the above criteria or $i$ is a source.
We say $i$ is blue if it meets the above criteria or $i$ is a sink.
Dually, we say $j$ is red if it satisfies the above criteria or $j$ is a sink.
We say $j$ is blue if it satisfies the above criteria or $j$ is a source.
The phrase ``largest source$<i$'' means the maximal element in the set of sources $s$ such that $s<i$.
The phrase ``smallest source$>i$'' means the minimal element in the set of sources $s$ such that $s<i$.
When dealing with the discrete values we always mean sinks and sources in $Q$.}\label{fig:omega map}
\end{sidewaysfigure*}

Let $1\leq m < n$ such that there is a path $1\to m$ in $Q$ or a path $m\to 1$ in $Q$ (possibly trivial).
Let $Q'$ be obtained from from $Q$ by reversing the path between $1$ and $m$ (if $m=1$ then $Q=Q'$).
It is well known that $\mathcal{D}^b(Q)$ and $\mathcal{D}^b(Q')$ are equivalent as triangulated categories.
Let $F:\mathcal{D}^b(Q)\to\mathcal{D}^b(Q')$ be a triangulated equivalence determined by sending $P_n[0]$ to $P_n[0]$.
Furthermore, we know $\tau\circ F(M) \cong  F \circ \tau(M)$ for every object $M$ in $\mathcal{D}^b(Q)$, where $\tau$ is the Auslander--Reiten translation.
Then this induces a functor $\overline{F}:\mathcal{C}(Q)\to\mathcal{C}(Q')$.
Overloading notation, we denote by $\overline{F}:\Ind(\mathcal{C}(Q))\to\Ind(\mathcal{C}(Q'))$ the induced map on isomorphism classes of indecomposable objects.

Let $\mathcal{Q}'$ be the continuification of $Q'$ and $\Omega':\Ind(\mathcal{C}(Q'))\to\Indr(\mathcal{Q}')$ the inclusion defined in the same way as $\Omega$.
Notice that that orientation of $\mathcal{Q}$ and $\mathcal{Q}'$ agree above $x_m$.
Furthermore, if $m>1$, the interval $(-\infty,x_m)$ is blue in $\mathcal{Q}$ if and only if it is red in $\mathcal{Q}'$ and vice versa.
Using Theorem~\ref{thm: continuous tilting gives bijection on max compatible sets}, there is a map $\phi:\Indr(\mathcal{Q})\to\Indr(\mathcal{Q}')$ such that $\{M,N\}\subset\Indr(\mathcal{Q})$ are $\Npi$-compatible if and only if $\{\phi(M),\phi(N)\}\subset\Indr(\mathcal{Q}'))$ are $\Npi$-compatible.
Following tedious computations, we have the following commutative diagram that preserves compatibility:
\begin{displaymath}
    \xymatrix{
        \Indr(\mathcal{Q}) \ar[r]^-{\phi} &
        \Indr(\mathcal{Q}) \\
        \Ind(\mathcal{C}(Q)) \ar[r]_-{\overline{F}} \ar[u]^-{\Omega} &
        \Ind(\mathcal{C}(Q)). \ar[u]_-{\Omega'}
    }
\end{displaymath}

\subsubsection{An example for $\mathbb{A}_4$ quivers}\label{sec:finite map and continuous tilting example}
Let $Q,Q'$ be the following quivers and let $\mathcal{Q},\mathcal{Q}'$ be the respective continuifications defined above with functions $\Omega,\Omega'$.

\begin{align*}
	Q&=\xymatrix{
	1 \ar[r] & 2 \ar[r] & 3 \ar[r] & 4
	} &
	Q'&=\xymatrix{
	1 & 2 \ar[l] & 3 \ar[l] \ar[r] & 4.
	} \\
	&\Omega: \Ind(\mathcal{C}(Q)) \to \Indr(\mathcal{Q}) &
	&\Omega': \Ind(\mathcal{C}(Q')) \to \Indr(\mathcal{Q'})
\end{align*}

Let $\overline{F}:\mathcal{C}(Q)\to\mathcal{C}(Q')$ be defined as above.
A visualization of the commutative diagram above in $\hyper$ is contained in Figure~\ref{fig:finite embedding in hyperbolic plane} on page~\pageref{fig:finite embedding in hyperbolic plane}.

For $\Omega: \Ind(\mathcal{C}(Q)) \to \Indr(\mathcal{Q})$:
\begin{align*}
	A &= \Omega(P_4) &
	F &= \Omega(M_{23}) &
	K &= \Omega(P_3[1])  \\
	B &= \Omega(P_3) &
	G &= \Omega(I_3) &
	L &= \Omega(I_1) \\
	C &= \Omega(P_2) &
	H &= \Omega(P_4[1])) &
	M &= \Omega(P_2[1]) \\
	D &= \Omega(P_1) &
	I &= \Omega(S_2) &
	N &= \Omega(P_1[1]) \\
	E &= \Omega(S_3) &
	J &= \Omega(I_2).
\end{align*}

To save space we will indicate an indecomposable module by its support interval.

In $\modr(\mathcal{Q}))$:
{\scriptsize
\begin{align*}
	A &= [\tan(3\pi/14),+\infty) &
	F &= [\tan(-\pi/14),\tan(5\pi/14)) &
	K &= [\tan(-5\pi/14),\tan(3\pi/14)) \\
	B &= [\tan(\pi/14),+\infty) &
	G &= [\tan(-3\pi/14),\tan(5\pi/14)) &
	L &= [\tan(-3\pi/14),\tan(\pi/14)) \\
	C &= [\tan(-\pi/14),+\infty) &
	H &= [\tan(-5\pi/14),\tan(5\pi/14)) &
	M &= [\tan(-5\pi/14),\tan(\pi/14)) \\
	D &= [\tan(-3\pi/14),+\infty) &
	I &= [\tan(-\pi/14),\tan(3\pi/14)) &
	N &= [\tan(-5\pi/14),\tan(-\pi/14)) \\
	E &= [\tan(\pi/14),\tan(5\pi/14)) &
	J &= [\tan(-3\pi/14),\tan(3\pi/14)).
\end{align*}
}

For $\Omega': \Ind(\mathcal{C}(Q')) \to \Indr(\mathcal{Q}')$:
\begin{align*}
	A &= \Omega'(P'_4) &
	F &= \Omega'(\textcolor{orange}{I'_2}) &
	K &= \Omega'(P'_3[1]) \\
	B &= \Omega'(P'_3) &
	G &= \Omega'(I'_3) &
	L &= \Omega'(\textcolor{orange}{P'_1}) \\
	C &= \Omega'(\textcolor{orange}{M'_{23}}) &
	H &= \Omega'(P'_4[1]) &
	M &= \Omega'(\textcolor{orange}{P'_2}) \\
	D &= \Omega'(\textcolor{orange}{I'_4}) &
	I &= \Omega'(\textcolor{orange}{P'_1[1]}) &
	N &= \Omega'(\textcolor{orange}{S'_2}) \\
	E &= \Omega'(\textcolor{orange}{I'_1}) &
	J &= \Omega'(\textcolor{orange}{P'_2[1]}).
\end{align*}

In $\modr(\mathcal{Q}'))$:
{\scriptsize
\begin{align*}
	A &= [\tan(3\pi/14),+\infty) &
	F &= (\textcolor{orange}{\tan(-5\pi/14)},\tan(5\pi/14)) &
	K &= (\textcolor{orange}{\tan(-\pi/14)},\tan(3\pi/14)) \\
	B &= (\textcolor{orange}{-\infty},+\infty) &
	G &= (\textcolor{plurp}{\tan(-3\pi/14)},\tan(5\pi/14)) &
	L &= (\textcolor{orange}{-\infty},\textcolor{orange}{\tan(-3\pi/14)}] \\
	C &= (\textcolor{orange}{\tan(-5\pi/14)},+\infty) &
	H &= (\textcolor{orange}{\tan(-\pi/14)},\tan(5\pi/14)) &
	M &= (\textcolor{orange}{-\infty},\tan(\pi/14)) \\
	D &= (\textcolor{plurp}{\tan(-3\pi/14)},+\infty) &
	I &= (\textcolor{orange}{\tan(-5\pi/14)},\tan(3\pi/14)) &
	N &= (\textcolor{plurp}{\tan(-5\pi/14)},\textcolor{plurp}{\tan(-\pi/14)}] \\
	E &= (\textcolor{orange}{-\infty},\tan(5\pi/14)) &
	J &= (\textcolor{plurp}{\tan(-3\pi/14)},\tan(3\pi/14)).
\end{align*}
}
The \textcolor{orange}{orange} highlights changes due to tilting.
The \textcolor{plurp}{purple} highlights a \emph{coincidental} fixed endpoint (but notice the change in open/closed).

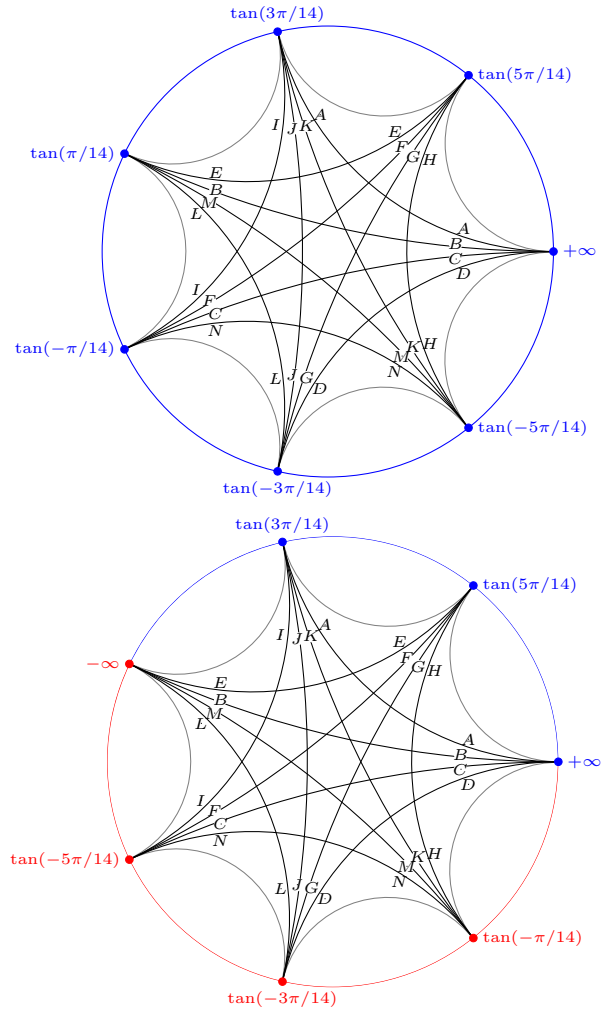
\begin{figure}[h]
\centering
\begin{tikzpicture}
	\begin{scope}
		\clip (3,0) arc (0:360:3);
		\foreach \x in 	{-154.28, -51.42, -102.85, 0, 51.42, 102.85, 154.28}
		{
			\draw [rotate around = {\x:(0,0)}, draw opacity = .5] (3,0) arc (270:-90:1.44);
			\draw [rotate around = {\x:(0,0)}] (3,0) arc (270:-90:3.76);
			\draw [rotate around = {\x:(0,0)}] (3,0) arc (270:-90:13.08);
		}
	\end{scope}
	\draw[blue] (3,0) arc (0:360:3);
	\foreach \x in {0, 51.42, 102.85, 154.28, -51.42, -102.85, -154.28}
	{
		\filldraw[fill=blue, draw=blue, rotate around = {\x:(0,0)}] (3,0) circle[radius=.5mm];
	}
	\draw[blue] (3,0) -- node[near start, right] {\tiny $+\infty$} (3,0);
	\draw[blue, rotate around = {51.42:(0,0)}] (3,0) -- node[near start, right] {\tiny $\tan(5\pi/14)$} (3,0);
	\draw[blue, rotate around = {102.85:(0,0)}] (3,0) -- node[near start, above] {\tiny $\tan(3\pi/14)$} (3,0);
	\draw[blue, rotate around = {154.28:(0,0)}] (3,0) -- node[near start, left] {\tiny $\tan(\pi/14)$} (3,0);
		\draw[blue, rotate around = {-154.28:(0,0)}] (3,0) -- node[near start, left] {\tiny $\tan(-\pi/14)$} (3,0);
	\draw[blue, rotate around = {-102.85:(0,0)}] (3,0) -- node[near start, below] {\tiny $\tan(-3\pi/14)$} (3,0);
	\draw[blue, rotate around = {-51.42:(0,0)}] (3,0) -- node[near start, right] {\tiny $\tan(-5\pi/14)$} (3,0);
	
	\foreach \x in {-154.28, -51.42, -102.85, 0, 51.42, 102.85, 154.28}
	{
		\filldraw[fill=white, draw opacity = 0, rotate around = {\x:(0,0)}] (1.7,.1) circle[radius=.9mm];
		\filldraw[fill=white, draw opacity = 0, rotate around = {\x:(0,0)}] (1.7,-.1) circle[radius=.9mm];
	}
	x	
	\draw (1.8,.3) -- node[near start] {\tiny $A$} (1.8,.3);
	\draw (1.7,.1) -- node[near start] {\tiny $B$} (1.7,.1);
	\draw (1.7,-.1) -- node[near start] {\tiny $C$} (1.7,-.1);
	\draw (1.8,-.3) -- node[near start] {\tiny $D$} (1.8,-.3);
	
	\draw[rotate around = {51.42:(0,0)}] (1.8,.3) -- node[near start] {\tiny $E$} (1.8,.3);
	\draw[rotate around = {51.42:(0,0)}] (1.7,.1) -- node[near start] {\tiny $F$} (1.7,.1);
	\draw[rotate around = {51.42:(0,0)}] (1.7,-.1) -- node[near start] {\tiny $G$} (1.7,-.1);
	\draw[rotate around = {51.42:(0,0)}] (1.8,-.3) -- node[near start] {\tiny $H$} (1.8,-.3);
	
	\draw[rotate around = {102.85:(0,0)}] (1.8,.3) -- node[near start] {\tiny $I$} (1.8,.3);
	\draw[rotate around = {102.85:(0,0)}] (1.7,.1) -- node[near start] {\tiny $J$} (1.7,.1);
	\draw[rotate around = {102.85:(0,0)}] (1.7,-.1) -- node[near start] {\tiny $K$} (1.7,-.1);
	\draw[rotate around = {102.85:(0,0)}] (1.8,-.3) -- node[near start] {\tiny $A$} (1.8,-.3);
	
	\draw[rotate around = {154.28:(0,0)}] (1.8,.3) -- node[near start] {\tiny $L$} (1.8,.3);
	\draw[rotate around = {154.28:(0,0)}] (1.7,.1) -- node[near start] {\tiny $M$} (1.7,.1);
	\draw[rotate around = {154.28:(0,0)}] (1.7,-.1) -- node[near start] {\tiny $B$} (1.7,-.1);
	\draw[rotate around = {154.28:(0,0)}] (1.8,-.3) -- node[near start] {\tiny $E$} (1.8,-.3);
	
	\draw[rotate around = {-154.28:(0,0)}] (1.8,.3) -- node[near start] {\tiny $N$} (1.8,.3);
	\draw[rotate around = {-154.28:(0,0)}] (1.7,.1) -- node[near start] {\tiny $C$} (1.7,.1);
	\draw[rotate around = {-154.28:(0,0)}] (1.7,-.1) -- node[near start] {\tiny $F$} (1.7,-.1);
	\draw[rotate around = {-154.28:(0,0)}] (1.8,-.3) -- node[near start] {\tiny $I$} (1.8,-.3);
	
	\draw[rotate around = {-102.85:(0,0)}] (1.8,.3) -- node[near start] {\tiny $D$} (1.8,.3);
	\draw[rotate around = {-102.85:(0,0)}] (1.7,.1) -- node[near start] {\tiny $G$} (1.7,.1);
	\draw[rotate around = {-102.85:(0,0)}] (1.7,-.1) -- node[near start] {\tiny $J$} (1.7,-.1);
	\draw[rotate around = {-102.85:(0,0)}] (1.8,-.3) -- node[near start] {\tiny $L$} (1.8,-.3);
	
	\draw[rotate around = {-51.42:(0,0)}] (1.8,.3) -- node[near start] {\tiny $H$} (1.8,.3);
	\draw[rotate around = {-51.42:(0,0)}] (1.7,.1) -- node[near start] {\tiny $K$} (1.7,.1);
	\draw[rotate around = {-51.42:(0,0)}] (1.7,-.1) -- node[near start] {\tiny $M$} (1.7,-.1);
	\draw[rotate around = {-51.42:(0,0)}] (1.8,-.3) -- node[near start] {\tiny $N$} (1.8,-.3);
\end{tikzpicture}

\begin{tikzpicture}
	\begin{scope}
		\clip (3,0) arc (0:360:3);
		\foreach \x in 	{-154.28, -51.42, -102.85, 0, 51.42, 102.85, 154.28}
		{
			\draw [rotate around = {\x:(0,0)}, draw opacity = .5] (3,0) arc (270:-90:1.44);
			\draw [rotate around = {\x:(0,0)}] (3,0) arc (270:-90:3.76);
			\draw [rotate around = {\x:(0,0)}] (3,0) arc (270:-90:13.08);
		}
		\draw[blue] (3,0) arc (0:154.82:3);
		\draw[red] (3,0) arc (0:-205.71:3);
	\end{scope}
	\foreach \x in {0, 51.42, 102.85}
		\filldraw[fill=blue, draw=blue, rotate around = {\x:(0,0)}] (3,0) circle[radius=.5mm];
	\foreach \x in {-154.28, -51.42, -102.85, 154.28}
		\filldraw[fill=red, draw=red, rotate around = {\x:(0,0)}] (3,0) circle[radius=.5mm];
		\draw[blue] (3,0) -- node[near start, right] {\tiny $+\infty$} (3,0);
	\draw[blue, rotate around = {51.42:(0,0)}] (3,0) -- node[near start, right] {\tiny $\tan(5\pi/14)$} (3,0);
	\draw[blue, rotate around = {102.85:(0,0)}] (3,0) -- node[near start, above] {\tiny $\tan(3\pi/14)$} (3,0);
	\draw[red, rotate around = {154.28:(0,0)}] (3,0) -- node[near start, left] {\tiny $-\infty$} (3,0);
		\draw[red, rotate around = {-154.28:(0,0)}] (3,0) -- node[near start, left] {\tiny $\tan(-5\pi/14)$} (3,0);
	\draw[red, rotate around = {-102.85:(0,0)}] (3,0) -- node[near start, below] {\tiny $\tan(-3\pi/14)$} (3,0);
	\draw[red, rotate around = {-51.42:(0,0)}] (3,0) -- node[near start, right] {\tiny $\tan(-\pi/14)$} (3,0);
	
	\foreach \x in {-154.28, -51.42, -102.85, 0, 51.42, 102.85, 154.28}
	{
		\filldraw[fill=white, draw opacity = 0, rotate around = {\x:(0,0)}] (1.7,.1) circle[radius=.9mm];
		\filldraw[fill=white, draw opacity = 0, rotate around = {\x:(0,0)}] (1.7,-.1) circle[radius=.9mm];
	}
	
	\draw (1.8,.3) -- node[near start] {\tiny $A$} (1.8,.3);
	\draw (1.7,.1) -- node[near start] {\tiny $B$} (1.7,.1);
	\draw (1.7,-.1) -- node[near start] {\tiny $C$} (1.7,-.1);
	\draw (1.8,-.3) -- node[near start] {\tiny $D$} (1.8,-.3);
	
	\draw[rotate around = {51.42:(0,0)}] (1.8,.3) -- node[near start] {\tiny $E$} (1.8,.3);
	\draw[rotate around = {51.42:(0,0)}] (1.7,.1) -- node[near start] {\tiny $F$} (1.7,.1);
	\draw[rotate around = {51.42:(0,0)}] (1.7,-.1) -- node[near start] {\tiny $G$} (1.7,-.1);
	\draw[rotate around = {51.42:(0,0)}] (1.8,-.3) -- node[near start] {\tiny $H$} (1.8,-.3);
	
	\draw[rotate around = {102.85:(0,0)}] (1.8,.3) -- node[near start] {\tiny $I$} (1.8,.3);
	\draw[rotate around = {102.85:(0,0)}] (1.7,.1) -- node[near start] {\tiny $J$} (1.7,.1);
	\draw[rotate around = {102.85:(0,0)}] (1.7,-.1) -- node[near start] {\tiny $K$} (1.7,-.1);
	\draw[rotate around = {102.85:(0,0)}] (1.8,-.3) -- node[near start] {\tiny $A$} (1.8,-.3);
	
	\draw[rotate around = {154.28:(0,0)}] (1.8,.3) -- node[near start] {\tiny $L$} (1.8,.3);
	\draw[rotate around = {154.28:(0,0)}] (1.7,.1) -- node[near start] {\tiny $M$} (1.7,.1);
	\draw[rotate around = {154.28:(0,0)}] (1.7,-.1) -- node[near start] {\tiny $B$} (1.7,-.1);
	\draw[rotate around = {154.28:(0,0)}] (1.8,-.3) -- node[near start] {\tiny $E$} (1.8,-.3);
	
	\draw[rotate around = {-154.28:(0,0)}] (1.8,.3) -- node[near start] {\tiny $N$} (1.8,.3);
	\draw[rotate around = {-154.28:(0,0)}] (1.7,.1) -- node[near start] {\tiny $C$} (1.7,.1);
	\draw[rotate around = {-154.28:(0,0)}] (1.7,-.1) -- node[near start] {\tiny $F$} (1.7,-.1);
	\draw[rotate around = {-154.28:(0,0)}] (1.8,-.3) -- node[near start] {\tiny $I$} (1.8,-.3);
	
	\draw[rotate around = {-102.85:(0,0)}] (1.8,.3) -- node[near start] {\tiny $D$} (1.8,.3);
	\draw[rotate around = {-102.85:(0,0)}] (1.7,.1) -- node[near start] {\tiny $G$} (1.7,.1);
	\draw[rotate around = {-102.85:(0,0)}] (1.7,-.1) -- node[near start] {\tiny $J$} (1.7,-.1);
	\draw[rotate around = {-102.85:(0,0)}] (1.8,-.3) -- node[near start] {\tiny $L$} (1.8,-.3);
	
	\draw[rotate around = {-51.42:(0,0)}] (1.8,.3) -- node[near start] {\tiny $H$} (1.8,.3);
	\draw[rotate around = {-51.42:(0,0)}] (1.7,.1) -- node[near start] {\tiny $K$} (1.7,.1);
	\draw[rotate around = {-51.42:(0,0)}] (1.7,-.1) -- node[near start] {\tiny $M$} (1.7,-.1);
	\draw[rotate around = {-51.42:(0,0)}] (1.8,-.3) -- node[near start] {\tiny $N$} (1.8,-.3);
\end{tikzpicture}
\caption{$\mathbb{A}_4$ example -- arcs in $\mathfrak{h}^2$. Continuous tilting doesn't move arcs in the hyperbolic plane. We can see this by relabeling the boundary of $\hyper$ accordingly. We also see how the diagonals of the heptagon (which models the cluster combinatorics for $\mathcal{C}(Q)$ and $\mathcal{C}(Q')$) are preserved by $\overline{F}$.}
\label{fig:finite embedding in hyperbolic plane}
\end{figure}

\section*{Future Work}
There are a few questions that naturally arise from our results.
What is the connection between our tilting and the reflection functors introduced in \cite{LZ22}?
What if we considered \emph{all} modules over a continuous quiver of type $\mathbb{A}$, instead of just those that are representable.
Can we expand Section~\ref{sec:continuous cluster character} and describe a continuous cluster algebra?
The authors plan to explore some of these questions in future research.
 
There is still much work to do with general continuous stability, as well. What can we learn by studying measured laminations of other surfaces? For example, can we connect a continuous type $\mathbb{D}$ quiver to measured laminations of the punctured (Poincar\'e) disk?
In the present paper, we consider stability conditions in the sense of King.
What about other kinds of stability conditions?
Furthermore, can the connections between stability conditions and moduli spaces be generalized to the continuous case?

\bibliographystyle{abbrv}
\bibliography{stability.bib}

\end{document}